\documentclass[11pt]{amsart}
\usepackage{amsmath}
\usepackage{amssymb}
\usepackage{tabularx}
\usepackage{enumerate}
\usepackage[dvipdfm]{graphicx}
\usepackage{texdraw}

\usepackage{mathrsfs}

\usepackage{amsfonts,amssymb,amsmath}
\usepackage{epsfig}

\makeatletter
\@addtoreset{equation}{section}

\makeatother
\newtheorem{thm}{Theorem}[section]
\newtheorem{lem}[thm]{Lemma}
\newtheorem{cor}[thm]{Corollary}
\newtheorem{prop}[thm]{Proposition}
\newtheorem{remark}[thm]{Remark}
\newtheorem{defn}[thm]{Definition}
\newcommand{\R}{\mathbb{R}}

\begin{document}
\title[Fisher-KPP equation with advection]{Long time behavior of solutions of
Fisher-KPP equation with advection and free boundaries$^\S$}
 \thanks{$\S$ This research was partly supported by  NSFC (No. 11271285). }
\author[H. Gu, B. Lou, M. Zhou]{Hong Gu$^\dag$, Bendong Lou$^\dag$ and Maolin Zhou$^\ddag$}
\thanks{$\dag$ Department of Mathematics, Tongji University, Shanghai 200092, China.}
\thanks{$\ddag$ Graduate School of Mathematical Sciences, The University of Tokyo, Tokyo 153-8914, Japan.}
\thanks{{\bf Emails:} {\sf honggu87@126.com} (H. Gu), {\sf blou@tongji.edu.cn} (B. Lou),
{\sf zhouml@ms.u-tokyo.ac.jp} (M. Zhou) }
\date{}

\begin{abstract}
We consider Fisher-KPP equation with advection: $u_t=u_{xx}-\beta u_x+f(u)$ for
$x\in (g(t),h(t))$, where $g(t)$ and $h(t)$ are two free boundaries satisfying
Stefan conditions. This equation is used to describe the population dynamics in
advective environments. We study the influence of the advection coefficient $-\beta$
on the long time behavior of the solutions. We find two parameters $c_0$ and
$\beta^*$ with $\beta^*>c_0>0$ which play key roles in the dynamics, here $c_0$
is the minimal speed of the traveling waves of Fisher-KPP equation. More precisely,
by studying a family of the initial data $\{ \sigma \phi \}_{\sigma >0}$ (where
$\phi$ is some compactly supported positive function), we show that,
(1) in case $\beta\in (0,c_0)$, there exists $\sigma^*\geqslant0$ such that spreading
happens when $\sigma>\sigma^*$ (i.e., $u(t,\cdot;\sigma\phi)\to 1$ locally uniformly
in $\R$)  and vanishing happens when $\sigma \in (0,\sigma^*]$ (i.e., $[g(t),h(t)]$
remains bounded and $u(t,\cdot;\sigma\phi)\to 0$ uniformly in $[g(t),h(t)]$);
(2) in case $\beta\in (c_0,\beta^*)$, there exists $\sigma^*>0$ such that virtual
spreading happens when $\sigma>\sigma^*$ (i.e., $u(t,\cdot;\sigma \phi)\to 0$
locally uniformly in $[g(t),\infty)$ and $u(t,\cdot + ct;\sigma \phi )\to 1$
locally uniformly in $\R$ for some $c>\beta -c_0$), vanishing happens when
$\sigma\in (0,\sigma^*)$, and in the transition case $\sigma=\sigma^*$,
$u(t, \cdot+o(t);\sigma \phi)\to V^*(\cdot-(\beta-c_0)t )$ uniformly, the latter
is a traveling wave with a ``big head" near the free boundary $x=(\beta-c_0)t$ and
with an infinite long \lq\lq tail" on the left;
(3) in case $\beta = c_0$, there exists $\sigma^*>0$ such that virtual spreading
happens when $\sigma > \sigma^*$ and $u(t,\cdot;\sigma \phi)\to 0$ uniformly in
$[g(t),h(t)]$ when $\sigma \in (0,\sigma^*]$;
(4) in case $\beta\geqslant \beta^*$, vanishing happens for any solution.
\end{abstract}

\maketitle

\section{Introduction}
In this paper, we consider the following problem
\begin{equation}\label{p}
\left\{
\begin{array}{ll}
 u_t =u_{xx}- \beta u_{x}+f(u), &  g(t)< x<h(t),\; t>0,\\
 u(t,g(t))=0,\ \ g'(t)=-\mu u_x(t, g(t)), & t>0,\\
 u(t,h(t))=0,\ \ h'(t)=-\mu u_x (t, h(t)) , & t>0,\\
-g(0)=h(0)= h_0,\ \ u(0,x) =u_0 (x),& -h_0\leqslant  x \leqslant  h_0,
\end{array}
\right.
\tag{$P$}
\end{equation}
where $\mu$ and $\beta$ are positive constants, $h_0>0$ and $u_0$ is a nonnegative $C^2$
function with support in $[-h_0,h_0]$, $f:[0,\infty)\to \R$ is a $C^1$ function
satisfying
\begin{equation}\label{f}
\left\{
\begin{array}{ll}
f(0)=f(1)=0,\quad (1-u)f(u)>0 \mbox{ for } u>0 \mbox{ and } u\not= 1,\\
f'(0)>0,\ f'(1)<0 \mbox{ and } f(u)\leqslant  f'(0)u \mbox{ for } u\geqslant  0.
\end{array}
\right.
\tag{$F$}
\end{equation}

The problem \eqref{p} is used to model the spreading of a new or invasive species,
under the influence of diffusion and advection.
The unknown $u(t,x)$ denotes the population density over a one dimensional habitat and
the free boundaries $x=g(t)$ and $x=h(t)$ represent the expanding fronts of the species.
We assume that the free boundaries move according to one-phase Stefan condition, which is a
kind of free boundary conditions widely used in the study of melting of ice \cite{R}, wound healing
\cite{CF}, and population dynamics \cite{BDK, DuLin, DuLou}.
The derivation of one-phase or two-phase Stefan conditions in population models as
singular limits of competition-diffusion systems can be found in \cite{HIMN, HMS} etc.

When $\beta=0$ (i.e., there is no advection in the environment), the qualitative properties of
the problem \eqref{p} was studied by Du and Lin \cite{DuLin} for logistic nonlinearity $f(u)=u(1-u)$.
Among others, they proved that, when $h_0\geqslant  \frac{\pi}{2}$, any solution of
\eqref{p} with $\beta=0$ grows up and converges to $1$ (which is called {\it spreading} phenomena);
when $h_0<\frac{\pi}{2}$, spreading happens if $\mu$ is large and
{\it vanishing} happens if $\mu$ is small (i.e., the solution converges to $0$).
The vanishing phenomena is a remarkable result since it shows that the presence of
free boundaries may avoid the so-called {\it hair-trigger effect}, which is a
phenomena shown in \cite{AW}: spreading always happens for a solution of the Cauchy
problem for $u_t =u_{xx}+f(u)$, no matter how small the positive initial data is.
Recently, Du and Lou \cite{DuLou} extended the results in \cite{DuLin} to the problem
with general monostable, bistable and combustion types of $f$, and gave a rather
complete description on the long time behavior of the solutions. In addition, Kaneko
and Yamada \cite{KY}, Liu and Lou \cite{LL1, LL2} studied the problem \eqref{p} with
$\beta=0$ and with a fixed boundary $g(t)\equiv 0$, Du and Guo \cite{DuGuo, DuGuo2}, Du, Matano
and Wang \cite{DMW}, Zhou and Xiao \cite{ZX}, Wang \cite{Wang} studied the problem
(without advection) in higher dimension spaces and/or in spatial heterogeneous environments.
Besides the qualitative properties, another interesting problem is the
asymptotic spreading speeds of the free boundaries when spreading happens. Du and
Lin \cite{DuLin}, Du and Lou \cite{DuLou} proved that, when spreading happens for
a solution $(u,g,h)$ of the problem \eqref{p} with $\beta=0$,
\begin{equation}\label{speed-000}
c^*:= \lim_{t\to\infty}\frac{h(t)}{t}  = \lim_{t\to\infty}\frac{-g(t)}{t} >0.
\end{equation}
Recently, Du, Matsuzawa and Zhou \cite{DMZ} improved this result to better ones:
\begin{equation}\label{speed-001}
\lim_{t\to \infty} h'(t) = \lim_{t\to \infty} [-g'(t)] =c^*,\quad
\lim_{t\to \infty} [h(t)- c^* t ] =H_\infty,\quad \lim_{t\to \infty} [g(t)+ c^* t ] =G_\infty,
\end{equation}
for some $H_\infty,\ G_\infty \in \R$.

In this paper we consider the problem \eqref{p} with $\beta>0$, which means that
the spreading of a species is affected by advection. In the field of ecology, organisms
can often sense and respond to local environmental cues by moving towards favorable habitats,
and these movement usually depend upon a combination of local biotic and abiotic factors
such as stream, climate, food and predators. For example, some diseases spread along the
wind direction. In 2009, Maidana and Yang \cite{Maiyang} studied the propagation of West
Nile Virus from New York City to California state. It was observed that West Nile Virus
appeared for the first time in New York City in the summer of 1999. In the second year
the wave front travels 187km to the north and 1100km to the south. Therefore, they took
account of the advection movement and showed that bird advection becomes an important
factor for lower mosquito biting rates. Another example is that Averill \cite{Ave}
considered the effect of intermediate advection on the dynamics of two-species competition
system, and provided a concrete range of advection strength for the coexistence of two
competing species. Moreover, three different kinds of transitions from small advection
to large advection were illustrated theoretically and numerically. Many other examples
involving advection can also be found in the field of ecology (cf.
\cite{BS,BC,BP,PL,RSB,SO,SG,VL} etc.).

From a mathematical point of view, to involve the influence of advection, one of the
simplest but probably still realistic approaches is to assume that species can move
up along the gradient of the density, as considered in \cite{BH,HL,HAES,RSB,SO,SG,VL} etc.

Gu, Lin and Lou \cite{GLL1, GLL2} studied the problem \eqref{p} with small advection.
They proved a spreading-vanishing dichotomy result on the long time behavior of positive solutions of \eqref{p},
which is similar as the conclusions in \cite{DuLin,DuLou}
for equations without advection. They also proved that, when spreading happens
for a solution of \eqref{p} with small advection, its rightward spreading
speed is bigger than the leftward one:
\begin{equation}\label{speed-002}
\lim_{t\to\infty}\frac{h(t)}{t}  >  \lim_{t\to\infty}\frac{-g(t)}{t} >0.
\end{equation}
Recently, Kaneko and Matsuzawa \cite{KM} improved this result to some conclusions like \eqref{speed-001}.

Our main purpose in this paper is to study the influence of the advection term
$-\beta u_x$ on the long time behavior of solutions of \eqref{p}.
As we will see below, our study improves the results in \cite{GLL1, GLL2, KM}
since we will study the problem \eqref{p} for all $\beta>0$, not only for small $\beta$.
Especially, when $\beta$ is large, the phenomena is much more complicated
and more interesting than the case where $\beta$ is small.

We point that the problem \eqref{p} for the equations with bistable type of nonlinearity,
or the problems (with monostable or bistable type of nonlinearity) in the interval
$[0,h(t)]$ with $x=h(t)$ a free boundary and $x=0$ a fixed boundary where $u$ satisfies
a general Robin boundary condition
can be considered similarly. In fact, in our forthcoming papers \cite{G, GLiu, GL} we
study these problems and obtain similar results as in this paper.

To sketch the influence of $\beta$, we introduce two important traveling waves.
First, consider the following problem
\begin{equation}\label{c0}
\left\{
\begin{array}{ll}
q''(z) - c q' (z) +f(q)=0,\quad z\in \R,\\
q(-\infty)=0,\  q(+\infty)=1,\ q(0)=1/2,\quad q' (z)>0 \mbox{ for } z \in  \R.
\end{array}
\right.
\end{equation}
It is well known that this problem has a solution $q(z;c)$ if and only if $c\geqslant  c_0 $, where
$$
c_0:= 2\sqrt{f'(0)}
$$
is called the minimal speed of the traveling waves of Fisher-KPP equation.
Denote $Q(z) := q(z;c_0)$, then
$u(t,x) = Q (x - (\beta -c_0)t)$ is a traveling wave of $u_t = u_{xx} - \beta u_x +f(u)$.
It travels leftward (resp. rightward) if and only if $\beta <c_0$ (resp. $\beta >c_0$).
Next we consider the following problem
\begin{equation}\label{c*}
\left\{
\begin{array}{ll}
q''(z) + (c-\beta ) q'(z) + f(q)=0,\quad z \in (-\infty,0),\\
q(0)=0,\ q(-\infty)=1,\ -\mu q'(0)=c,\quad q' (z)<0\mbox{ for } z\in (-\infty, 0].
\end{array}
\right.
\end{equation}
As is shown in Lemma \ref{lem:semi-wave} (see also \cite{DuLin,DuLou,GLL2}), for any
$\beta>0$, this problem has a unique solution $(c,q)=(c^*,U^*(z))$.
So $u(t,x) = U^* (x-c^* t)$ is a solution of $u_t = u_{xx} -\beta u_x +f(u)$,
with $u(t,c^* t)=0,\ c^* = - \mu u_x(t, c^* t)$. It is called a {\it traveling semi-wave}
in \cite{DuLou} since it is only defined for $x\leqslant   c^* t$. We also write
$c^*$ as $c^*(\beta)$ to emphasize the dependence of $c^*$ on $\beta$, then we
will show in Lemma \ref{lem:semi-wave} that the equation $\beta - c_0 = c^*(\beta)$
has a unique root $\beta^* >c_0$:
\begin{equation}\label{def:beta*}
\beta^* - c_0 = c^*(\beta^*).
\end{equation}

We will see below that the traveling wave $Q(x -(\beta -c_0) t)$ and the traveling
semi-wave $U^* (x-c^* t)$ are of special importance in the study of spreading
solutions. To explain their roles intuitively, we consider the problem \eqref{p}
with initial data $u_0 (x)$ which is even and
$$
u_0(x) = \left\{
\begin{array}{ll}
 1, & x\in [0, h_0-1],\\
 \mbox{smooth and decreasing}, & x\in [h_0 -1, h_0],
 \end{array}
 \right. \qquad \mbox{ with } h_0 \gg 1.
$$
It is easily seen by the maximum principle that $u(t,\cdot)$ has
exactly one maximum point. As usual, we call the sharp decreasing part in the graph
of $u(t,\cdot)$ the {\it front}, and call the sharp increasing part on the left side
the {\it back}. Now we sketch the influence of the advection $-\beta u_x$. {\it Case 1}.
When $\beta \in (0, c_0)$, the advection influence is not strong, the solution
has enough space between the back and the front to grow up and to converge to 1. Its front
approaches a profile like $U^*(\cdot)$ and moves rightward at a speed $\approx c^*$. Its back
approaches a profile like $U^*(-\; \cdot)$
and moves leftward at a speed smaller than $ c_0 -\beta$ (see details in Theorem
\ref{thm:profile of spreading sol} below). This case is similar as the spreading
phenomena in \cite{DuLin,DuLou} for the equation with $\beta=0$. {\it Case 2}. When
$\beta \in (c_0, \beta^*)$ with $\beta^*$ being the unique root of \eqref{def:beta*},
the traveling wave $Q(x  - (\beta-c_0)t)$ travels rightward at a speed $\beta -c_0>0$.
Hence the back of the solution $u$, with a shape like $Q(x-(\beta-c_0)t)$, is pushed
by $Q(x  - (\beta-c_0)t)$ to move rightward at a speed $\approx \beta -c_0$, and so
$u\to 0$ locally uniformly. But, when the initial domain is wide enough, the solution
still have enough space to grow up between the back and the front
since the front moves rightward (at a speed $\approx c^*$) faster than the back. In
this paper we call such a phenomena as {\it virtual spreading} (see Theorem
\ref{thm:middle beta} and Lemma \ref{lem:condition for virtual spreading} below). {\it Case 3}.
When $\beta = c_0$, the traveling wave $Q(x  - (\beta-c_0)t) = Q(x)$ is indeed a
stationary solution of \eqref{p}$_1$. However, the back of the solution $u$ still
moves rightward at a speed $O(t^{-1})$ since it starts from a compactly supported
initial data $u_0$ (cf. \cite{HNRR} and see details below). Hence virtual spreading
still happens when $h_0$ is sufficiently large, since the front moves rightward
at speed $c^*$, faster than the back. {\it Case 4}. When $\beta > \beta^*$, the back
moves rightward (at a speed $\approx \beta-c_0$) faster than the front (which moves
rightward at speed $\approx c^*< \beta -c_0$). So the solution is suppressed by its
back, and then $u\to 0$ uniformly. In summary, the long time behavior of the solutions
is quite different for $\beta\in (0,c_0),\; \beta =c_0,\; \beta\in (c_0, \beta^*)$
and $\beta> \beta^*$.

This paper is organized as the following. In section 2 we present our main results.
In section 3 we give some preliminaries including the comparison principles, stationary solutions, several types of
traveling waves, zero number arguments and some upper bound estimates. In section 4
we study the influence of the advection on the long time behavior of the solutions.
In section 5, we revisit the virtual spreading phenomena and give a uniform
convergence for such solutions.

\section{Main Results}
Throughout this paper we choose initial data $u_0$ from the following set:
\begin{equation}\label{def:X}
\mathscr {X}(h_0):= \Big\{ \phi \in C^2 ([-h_0,h_0]) \mid
\phi(-h_0)= \phi (h_0)=0,\; \phi(x) \geqslant ,\not\equiv 0 \ \mbox{in } (-h_0,h_0).\Big\}
\end{equation}
where $h_0 >0$ is any given real number. By a similar argument as in \cite{DuLin,DuLou},
one can show that, for any initial data $u_0\in\mathscr{X}(h_0)$, the problem \eqref{p}
has a time-global solution $(u,g,h)$, with $u\in C^{1+\nu/2,2+\nu}((0,\infty)\times[g(t),h(t)])$
and $g,h\in C^{1+\nu/2}((0,\infty))$ for any $\nu\in(0,1)$. Moreover, it follows from the
maximum principle that, when $t>0$, the solution $u$ is positive in $(g(t),h(t))$,
$u_x(t,g(t))>0$ and $u_x(t,h(t))<0$. Hence $g'(t)<0$, $h'(t)>0$. Denote
$$
g_{\infty}:=\lim_{t\to\infty}g(t),\quad h_{\infty}:=
\lim_{t\to\infty}h(t),\quad I(t):= [g(t),h(t)] \quad \mbox{and} \quad
I_{\infty}:=(g_{\infty},h_{\infty}).
$$
In what follows, we mainly consider the solution of \eqref{p} with initial data
$u_0=\sigma\phi$ for some given $\phi\in\mathscr {X}(h_0)$ and $\sigma\geqslant0$.
We also use $(u(t,x;\sigma\phi),g(t;\sigma\phi),h(t;\sigma\phi))$ to denote such
a solution. Now we list some possible situations for the solutions of \eqref{p}.
\begin{itemize}
\item {\it spreading} : $I_\infty =\R$ and
\begin{equation}\label{eq spreading}
\lim_{t\to\infty}u(t,\cdot)=1 \mbox{ locally uniformly in $\R$};
\end{equation}

\item {\it vanishing} :  $I_\infty$ is a bounded interval and
\begin{equation}\label{eq vanishing}
\lim_{t\to\infty}\max_{g(t)\leqslant  x\leqslant  h(t)} u(t,x)=0;
\end{equation}

\item {\it virtual spreading} : $g_\infty>-\infty,\ h_\infty=+\infty$,
\begin{equation}\label{eq v s}
\lim_{t\to\infty}u(t,\cdot)=0 \mbox{ locally uniformly in } I_\infty
\end{equation}
and
\begin{equation}\label{eq v s 2}
\lim_{t\to\infty}u(t,\cdot+ct)=1 \mbox{ locally uniformly in } \R,
\quad \mbox{ for some } c >0;
\end{equation}

\item {\it virtual vanishing} : $g_\infty>-\infty,\ h_\infty=+\infty$ and
\eqref{eq vanishing} holds.
\end{itemize}

When the advection is small, we have the following conclusion on the long time
behavior of the solutions.

\begin{thm}[{\rm the case $\beta \in (0, c_0)$}]\label{thm:small beta}
Assume $0<\beta<c_0$ and $(u,g,h)$ is a time-global solution of \eqref{p} with
initial data $u_0 = \sigma \phi$ for some $\phi\in \mathscr {X}(h_0)$. Then
there exists $\sigma^* = \sigma^* (h_0, \phi) \in [0,\infty]$ such that

{\rm (i)} {\rm vanishing} happens when $\sigma \in [0,\sigma^*]$, with
$|I_\infty| = h_\infty -g_\infty \leqslant\frac{2\pi}{\sqrt{c_0 ^2-\beta^2}}$;

{\rm (ii)} {\rm spreading} happens when $\sigma> \sigma^*$.
\end{thm}

\noindent
From this theorem we see that the long time behavior of the solutions of \eqref{p} with
small advection: $\beta \in (0,c_0)$ is similar as the case without advection: $\beta =0$
(cf. \cite{DuLin, DuLou, GLL1}). The main reason is that in both cases the problem \eqref{p}
has exactly two stationary solutions: $0$ and $1$ in $\R$. The proof of this theorem,
which is given in subsection 4.2, is also similar as that for $\beta=0$.

\smallskip

Next we consider the case where the advection is not small: $\beta \geqslant  c_0$.
The most interesting phenomena appears in the problem with medium-sized advection:
$\beta\in [c_0, \beta^*)$, where $\beta^*$ is the unique root of \eqref{def:beta*}.

\begin{thm}[\rm the case $c_0<\beta<\beta^*$]\label{thm:middle beta}
  Assume $c_0<\beta<\beta^*$ and $(u,g,h)$ is a time-global solution of \eqref{p}
  with initial data $u_0 =\sigma \phi$ for some $\phi\in \mathscr{X}(h_0)$. Then
  there exists $\sigma^* = \sigma^* (h_0, \phi) \in (0,\infty]$ such that

{\rm (i)} {\rm virtual spreading} happens when $\sigma >\sigma^*$, and
\[
\lim_{t\to\infty}u(t,\cdot+ct)=1 \mbox{ locally uniformly in $\R$},
\quad \mbox{ for any } c\in (\beta -c_0, c^*),
\]
where $c^* =c^*(\beta)$ is the speed of the traveling semi-wave in \eqref{c*};

{\rm (ii)} {\rm vanishing} happens when $0<\sigma <\sigma^*$;

{\rm (iii)} in the {\rm transition} case $\sigma =\sigma^*$: $g_\infty>-\infty,\ h_\infty=+\infty$,
\[
\lim_{t\to\infty}h'(t)=\beta-c_0 \quad \mbox{and} \quad h(t)= (\beta -c_0)t+\varrho (t)
\]
with $\varrho(t) = o(t)$ and $\varrho(t)\to\infty\ (t\to \infty)$. In addition,
\begin{equation}\label{u to V*}
\lim\limits_{t\to\infty} \| u(t, \cdot ) - V^*(\cdot -(\beta-c_0)t  - \varrho(t)) \|_{L^\infty (I(t))} =0,
\end{equation}
where $V^*(z)$ is the unique solution of
\begin{equation}\label{vv}
\left\{
\begin{array}{ll}
q''(z)-c_0 q'(z)+f(q)=0 \quad \mbox{for } z\in (-\infty,0),\\
q(0)=0,\ q(-\infty)=0,\ q(z)>0 \mbox{ for } z\in (-\infty,0),\ -\mu q'(0)=\beta-c_0.
\end{array}
\right.
\end{equation}
\end{thm}

\noindent
In the next section we will see that $V^*$ has a {\it tadpole-like} shape: it has a
\lq\lq big head" and a boundary on the right side and an infinite long \lq\lq tail" on the
left side. So we call $V^*(x-(\beta -c_0)t)$ a {\it tadpole-like traveling wave} with
speed $\beta-c_0$, which exists if and only if $\beta\in(c_0,\beta^*)$ (see Lemma
\ref{lem:tadpole tw bata<beta*} below). Theorem \ref{thm:middle beta} (iii) implies that,
roughly, $u(t,x)$ converges to this traveling wave.

In Aronson and Weinberger \cite{AW}, it was shown that any positive solution of the Cauchy
problem for Fisher-KPP equation converges to 1 (i.e., {\it hair-trigger effect}). In
\cite{DuLin,DuLou}, by introducing the free boundaries, the authors proved a spreading-vanishing
dichotomy on the long time behavior of the solutions of Fisher-KPP equation. In particular,
vanishing may happen for some solutions. Now our Theorem \ref{thm:middle beta} gives the
third possibility besides the virtual spreading and vanishing, that is, with a medium-sized
advection in the equation, there may exist a transition state: the solution converges to a
{\it tadpole-like traveling wave}. This interesting phenomena is new comparing with the
results for Cauchy problems and for free boundary problems without advection.

\begin{thm}[the case $\beta =c_0$]\label{thm:beta=c0}
  Assume $\beta=c_0$ and $(u,g,h)$ is a time-global solution of \eqref{p}
  with initial data $u_0 =\sigma \phi$ for some $\phi\in \mathscr{X}(h_0)$. Then
  there exists $\sigma_*,\; \sigma^* \in (0,\infty]$ with $\sigma_* \leqslant  \sigma^*$
  such that

{\rm (i)} {\rm virtual spreading} happens when $\sigma>\sigma^*$, and
\[
\lim_{t\to\infty}u(t, \cdot + c t)=1 \mbox{ locally uniformly in $\R$}, \quad
\mbox{ for any } c\in (0,c^*),
\]
where $c^*=c^*(\beta)$ is the speed of the traveling semi-wave in \eqref{c*};

{\rm (ii)} {\rm vanishing} happens when $0<\sigma<\sigma_*$;

{\rm (iii)} {\rm virtual vanishing} happens when $\sigma\in [\sigma_*, \sigma^*]$.
\end{thm}

\noindent
The transition cases in Theorem \ref{thm:middle beta} and Theorem \ref{thm:beta=c0} are
different. In case $c_0<\beta<\beta^*$, a solution $u(t,x;\sigma\phi)$ is a transition one
only if the initial value is taken the sharp threshold value $\sigma^*\phi$. However,
in case $\beta=c_0$, we obtain transition solutions whose initial data are taken from
$\{\sigma\phi\mid \sigma\in [\sigma_*, \sigma^*]\}$. Whether or not this domain
is a singleton: $\sigma_*=\sigma^*$ is still open now. The difficulty in studying this
problem is that virtual vanishing solutions have no \lq\lq shapes", so it is not easy to
compare one to another.

The conclusions for the problem with large advection: $\beta \geqslant  \beta^*$ is rather
simple.

\begin{thm}[the case $\beta \geqslant  \beta^*$]\label{thm:large beta}
Assume $\beta\geqslant \beta^*$ and $(u,g,h)$ is a time-global solution of \eqref{p} with
initial data $u_0\in\mathscr {X}(h_0)$. Then vanishing happens.
\end{thm}

Besides the convergence/dichotomy/trichotomy results on the long time behavior of the
solutions as stated in the previous theorems, we can say more about the solutions when
(virtual) spreading happens. It turns out that, when $\beta\in [c_0, \beta^*)$,
the virtual spreading solution can be characterized by the rightward traveling semi-wave
$U^*(x-c^* t)$ and the traveling wave $Q(x -(\beta-c_0)t)$; when $\beta\in (0,c_0)$, the spreading
solution can be characterized by $U^*(x-c^* t)$ and the leftward traveling semi-wave $U^*_l (x - c^*_l t)$.
Here $(c^*_l, U^*_l)$ (with $c_l^*<0$) is the unique
solution of the following problem with $\beta\in(0,c_0)$ (see details in subsection 3.3)
\begin{equation}\label{q*l}
\left\{
\begin{array}{ll}
q''(z) + (c-\beta) q'(z) + f(q)=0,\quad z \in (0,\infty),\\
q(0)=0,\ q(\infty)=1,\ -\mu q'(0)=c,\ q' (z)>0\mbox{ for }z\in (0, \infty).
\end{array}
\right.
\end{equation}
Using these traveling waves we can give the asymptotic profiles for (virtual) spreading
solutions.

\begin{thm}\label{thm:profile of spreading sol}
Assume spreading or virtual spreading happens for a solution of \eqref{p} as
in Theorems \ref{thm:small beta}, \ref{thm:middle beta} or \ref{thm:beta=c0}.
Let $(c^*,U^*)$ be the unique solution of \eqref{c*} with $c^*>0$.
\begin{itemize}
  \item [(i)] When $\beta\in(0,c_0)$, let $(c_l^*,U_l^*)$ be the unique solution of \eqref{q*l}
with $0< -c_l^* <c^*$. Then there exist $H_\infty$, $G_\infty\in\R$ such that
\begin{equation}\label{right spreading speed}
\lim\limits_{t\to\infty}[h(t)-c^*t] = H_\infty ,\quad \lim\limits_{t\to\infty}h'(t)=c^*,
\end{equation}
\begin{equation}\label{left spreading speed}
\lim\limits_{t\to\infty}[g(t)-c_l^*t] = G_\infty ,\quad \lim\limits_{t\to\infty}g'(t)=c_l^*,
\end{equation}
and, if we extend $U^*$, $U_l^*$ to be zero outside their supports we have
\begin{equation}\label{profile convergence 1}
\lim\limits_{t\to\infty} \left\| u(t,\cdot) - U^*(\cdot -c^*t - H_\infty) \cdot
U^*_l \big(\cdot - c^*_l  t - G_\infty \big)\right\|_{L^\infty (I(t))} =0.
\end{equation}
  \item [(ii)] When $\beta\in [c_0,\beta^*)$, \eqref{right spreading speed} holds
  for some $H_\infty\in\R$. Moreover, if we extend $U^*$ to be zero outside its support,
  then
\begin{equation}\label{profile convergence 2}
\lim\limits_{t\to\infty} \left\| u(t,\cdot)-
U^*(\cdot -c^* t- H_\infty) \cdot Q \big(\cdot - (\beta-c_0)t-\theta(t) \big)\right\|_{L^\infty (I(t))}=0
\end{equation}
for some function $\theta(t)$ satisfying $\theta(t)=o(t)$ and $\theta(t)\to \infty\ (t\to \infty)$.
\end{itemize}
\end{thm}

\noindent
Assume $\beta \in [0,c_0)$ and spreading happens for a solution $(u,g,h)$ of \eqref{p}.
The asymptotic spreading speed $\lim_{t\to \infty} \frac{h(t)}{t}=c^*$ was obtained in
\cite{DuLin,DuLou} for the case $\beta =0$, and in \cite{GLL2} for the case $\beta\in (0,c_0)$.
Recently, Du, Matsuzawa and Zhou \cite{DMZ}, Kaneko and Matsuzawa \cite{KM} improved
them to analogues of \eqref{right spreading speed}, \eqref{left spreading speed} and
\eqref{profile convergence 1}. Note that our theorem includes both the case
$\beta\in (0,c_0)$ and the case $\beta\in [c_0, \beta^*)$. The proof of
\eqref{profile convergence 2} will be given in the last section, based on the fact that
$Q$ is steeper than any other entire solution (see section 5 below and \cite{DGM}).

\section{Preliminaries}
In this section we first give some comparison principles and then
present all the bounded stationary solutions and traveling wave
solutions of \eqref{p}$_1$ which will be used for comparison. In the fourth subsection we give some results
on the zero numbers of the solutions of linear equations which will play key roles in our approach.
In the last subsection we give some precise upper bound estimates for the solutions.

\subsection{The comparison principle}
In this subsection we give two types of comparison principles which will be used frequently in this paper. Similar as \cite{DuLin,DuLou}, we have

\begin{lem}
\label{lem:comp1} Assume $T\in(0,\infty)$,
$\overline g(t),\ \overline h(t)\in C^1([0,T])$, $\overline u(t,x)\in C(\overline D_T)\cap C^{1,2}(D_T)$
with $D_T=\{(t,x)\in\R^2\mid0<t\leqslant T, \overline g(t)<x<\overline h(t)\}$, and
\begin{eqnarray*}
\left\{
\begin{array}{lll}
\overline u_{t} \geqslant \overline u_{xx} -\beta \overline u_x+f(\overline u),\; & 0<t \leqslant T,\
\overline g(t)<x<\overline h(t), \\
\overline u= 0,\quad \overline g'(t)\leqslant -\mu \overline u_x,\quad &
0<t \leqslant T, \ x=\overline g(t),\\
\overline u= 0,\quad \overline h'(t)\geqslant -\mu \overline u_x,\quad
&0<t \leqslant T, \ x=\overline h(t).
\end{array} \right.
\end{eqnarray*}
 If
\[
\mbox{$[-h_0, h_0]\subset [\overline g(0), \overline h(0)]$ \quad
and \quad $u_0(x)\leqslant \overline u(0,x)$ for $x\in[-h_0,h_0]$,}
\]
and $(u,g, h)$ is a solution of \eqref{p}, then
\[
\mbox{ $g(t)\geqslant \overline g(t),\; h(t)\leqslant\overline h(t)$\quad for $t\in(0,
T]$,}
\]
\[
\mbox{$u(t,x)\leqslant \overline u(t,x)$ \quad for $t\in (0, T]$ and $x\in (g(t), h(t))$.}
\]
\end{lem}

\begin{lem}
\label{lem:comp2} Assume $T\in(0,\infty)$, $l(t),\, k(t)\in C^1([0,T])$, $w(t,x)\in C(\overline D_T)\cap C^{1,2}(D_T)$
with $D_T=\{(t,x)\in\R^2\mid0<t\leqslant T, l(t)<x<k(t)\}$, and
\begin{eqnarray*}
\left\{
\begin{array}{lll}
w_{t}\geqslant  w_{xx}-\beta w_x+ f(w),\; &0<t \leqslant T,\ l(t)<x<k(t), \\
w\geqslant u, &0<t \leqslant T, \ x= l(t),\\
w= 0,\quad k'(t)\geqslant -\mu w_x,\quad &0<t \leqslant T, \ x=k(t),
\end{array} \right.
\end{eqnarray*}
with
\begin{equation}\label{com prin 2}
\mbox{$g(t)\leqslant l(t)\leqslant h(t) $ for $t\in[0,T]$,\quad $h_0\leqslant k(0),$\quad  $u_0(x)\leqslant w(0,x)$ for $x\in[l(0),h_0]$,}
\end{equation}
where  $(u,g, h)$ is a solution of \eqref{p}. Then
\[
\mbox{ $h(t)\leqslant k(t)$ for $t\in(0, T]$,\quad $u(t,x)\leqslant w(t,x)$ for $t\in (0, T]$\ and $ l(t)<x< h(t)$.}
\]
\end{lem}

The proof of Lemma \ref{lem:comp1} is identical to that of Lemma 5.7
in \cite{DuLin}, and a minor modification of this proof yields Lemma
\ref{lem:comp2}.

\begin{remark}
\rm The function $\overline u$, or the triple $(\overline u,\overline g,\overline h)$
in Lemmas \ref{lem:comp1} and the function $w$, or the triple
$(w,l,k)$ in Lemma \ref{lem:comp2} are often called the upper solutions of \eqref{p}.
There is a symmetric version of Lemma~\ref{lem:comp2}, where the conditions on the left and
right boundaries are interchanged.
The lower solutions can be defined analogously by reversing all the inequalities except for
$g(t)\leqslant l(t)\leqslant h(t)$ in \eqref{com prin 2}.  We also have corresponding comparison
results for lower solutions in each case.
\end{remark}

\subsection{Phase plane analysis and stationary solutions}\label{subsec:phase plane}
We first use the phase plane analysis to study the following equation
\begin{equation}\label{phase sol}
q''(z) + \gamma q'(z)+f(q)=0,\quad q(z) \geqslant  0 \quad \mbox{ for } z\in J,
\end{equation}
where $J$ is some interval in $\R$.
Note that a nonnegative stationary solution $u$ of \eqref{p}$_1$ solves
\eqref{phase sol} with $\gamma = - \beta$, a nonnegative traveling wave
$u(t,x)= q(x-ct)$ of \eqref{p}$_1$ solves \eqref{phase sol} with
$\gamma = c-\beta$ and $J=\R$.

The equation \eqref{phase sol} is equivalent to the system
\begin{equation}\label{q-p}
\left\{
\begin{array}{l}
q'(z)=p,\\ p'(z)= - \gamma  p- f(q).
\end{array}
\right.
\end{equation}
A solution $(q(z), p(z))$ of this system traces out a trajectory in the $q$-$p$ phase plane (cf.
\cite{AW,DuLou,Pet}). Such a trajectory has slope
\begin{equation}\label{Pq}
\frac{\mathrm{d}p}{\mathrm{d}q}= - \gamma  -\frac{f(q)}{p}
\end{equation}
at any point where $p\neq0$.
It is easily seen that $(0,0)$ and $(1,0)$ are two singular points on the phase plane.
We are only interested in the case $\gamma < c_0 := 2\sqrt{f'(0)}$. For such a $\gamma$,
the eigenvalues of the corresponding linearizations at the singular points are
$$
\lambda_{0}^{\pm}=\frac{-\gamma \pm\sqrt{\gamma^2 -4f'(0)}}{2}\ \ (\mbox{at } (0,0))\quad
\mbox{and} \quad
\lambda_{1}^{\pm}=\frac{-\gamma \pm\sqrt{\gamma^2 -4 f'(1)}}{2}\ \ (\mbox{at } (1,0)),
$$
respectively. Since $f'(0)>0$ and $f'(1)<0$, $(1,0)$ is a saddle point, $(0,0)$ is a center
when $\gamma=0$, or a focus when $0<|\gamma|<c_0$, or a node when $\gamma \leqslant  - c_0$.
By the phase plane analysis (cf. \cite{AW,DuLou,Pet}), it is not difficult to give all
kinds of bounded, nonnegative solutions of \eqref{phase sol} for $\gamma <c_0$ (see Figure 1).

\medskip
{\rm (i)} \textbf{Constant solutions:} $q\equiv 0$ and $q\equiv 1$.

\medskip
{\rm (ii)} \textbf{Strictly decreasing solutions on the half-line in case $\gamma<c_0$:}
$q(\cdot)=U(\cdot-z_0;\gamma)$
for any $z_0\in\R$, where $U\in C^2((-\infty,0])$ is the unique solution of \eqref{phase sol}
in $(-\infty,0)$, with $U(0;\gamma)=0$, $U(-\infty;\gamma)=1$ and $U'(\cdot;\gamma)<0$ in
$(-\infty,0]$ {\rm (}see  $\Gamma_1$ and $\Gamma_5$ in Figure 1{\rm )}.
Denote
$$
P(\gamma) := -\mu U'(0;\gamma).
$$
Using the comparison principle for the ordinary
differential equation \eqref{Pq} we have $P'(\gamma)<0$ for $\gamma\in (-\infty,c_0)$,
$P(c_0 -0) =0$ and $P(-\infty)= +\infty$ (see Figure 2 (a)).

\medskip
{\rm (iii)} \textbf{Strictly increasing solutions on the half-line in case $\gamma\in(-c_0,c_0)$:}
$q(\cdot)=U_l (\cdot-z_0;\gamma)$ for any $z_0\in\R$, where $U_l \in C^2([0,\infty))$ is the unique
solution of \eqref{phase sol} in $(0,\infty)$, with $U_l (0;\gamma)=0$, $U_l (\infty;\gamma)=1$ and
$U'_l (\cdot;\gamma)>0$ in $[0,+\infty)$ {\rm (}see $\Gamma_4$ in Figure 1 (a){\rm )}.

\medskip
{\rm (iv)} \textbf{Solutions with compact supports in case $\gamma\in(-c_0,c_0)$:}
$q(\cdot)= W(\cdot - z_0;b,\gamma)$ for any $z_0\in\R$, where for each $b \in (0,P(\gamma))$,
there exists a unique $L(b,\gamma) >0$ such that $W \in C^2([-L(b,\gamma), 0])$ is the unique
solution of \eqref{phase sol} in $(-L(b,\gamma),0)$ with $W(-L(b,\gamma);b,\gamma)=W(0;b,\gamma)=0$
and $b = -\mu W'(0;b,\gamma)$ {\rm (}see $\Gamma_2$ and $\Gamma_3$ in Figure 1 (a){\rm )}.
Each point $(\gamma, b)$ in the set
$S_1:= \{(\gamma,b) \mid 0<b<P(\gamma),\; -c_0 <\gamma <c_0\}$
in Figure 2 (a) corresponds to such a compactly supported solution $W(z;b,\gamma)$.

\begin{figure}[!htbp]
\begin{center}
\includegraphics[scale=0.18]{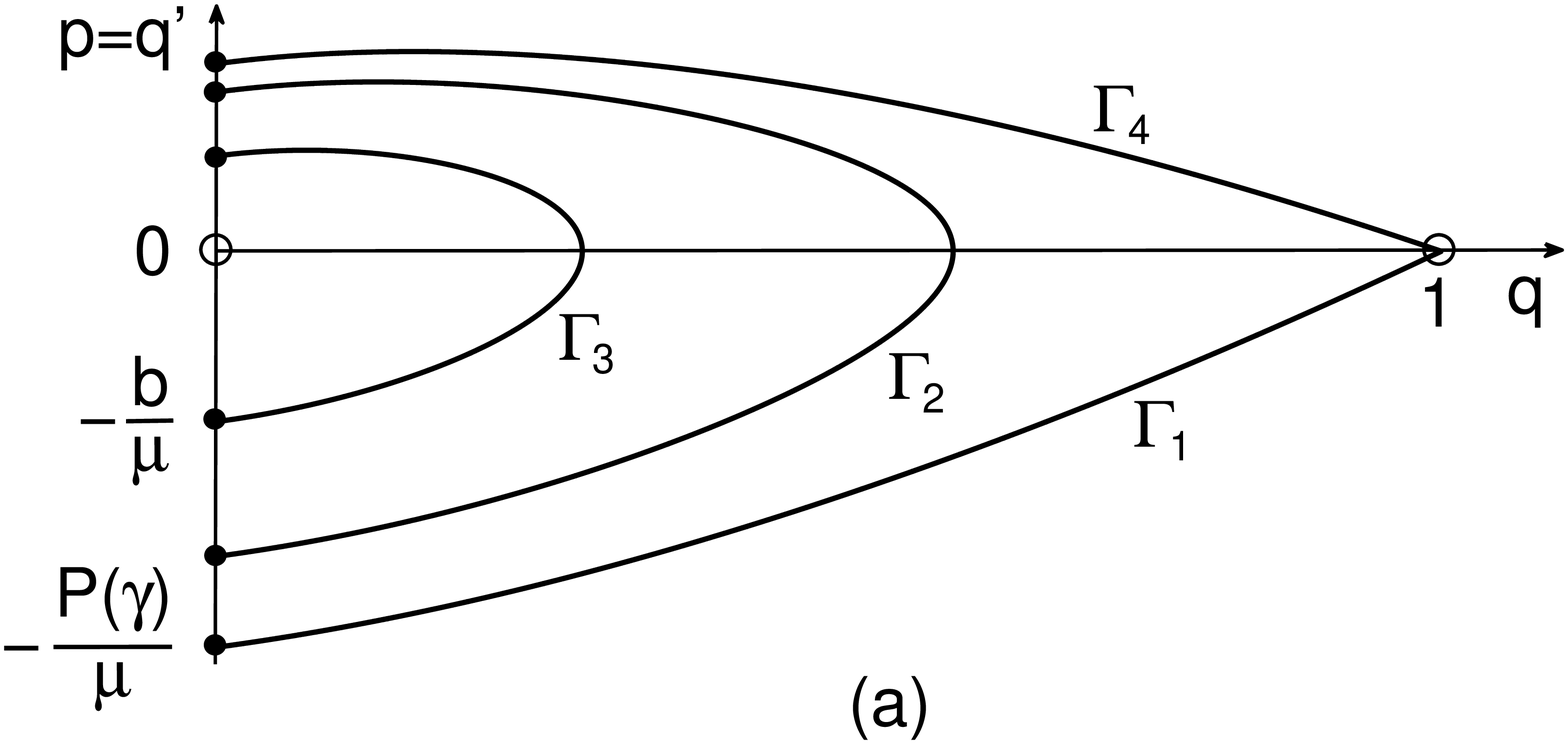}
\includegraphics[scale=0.18]{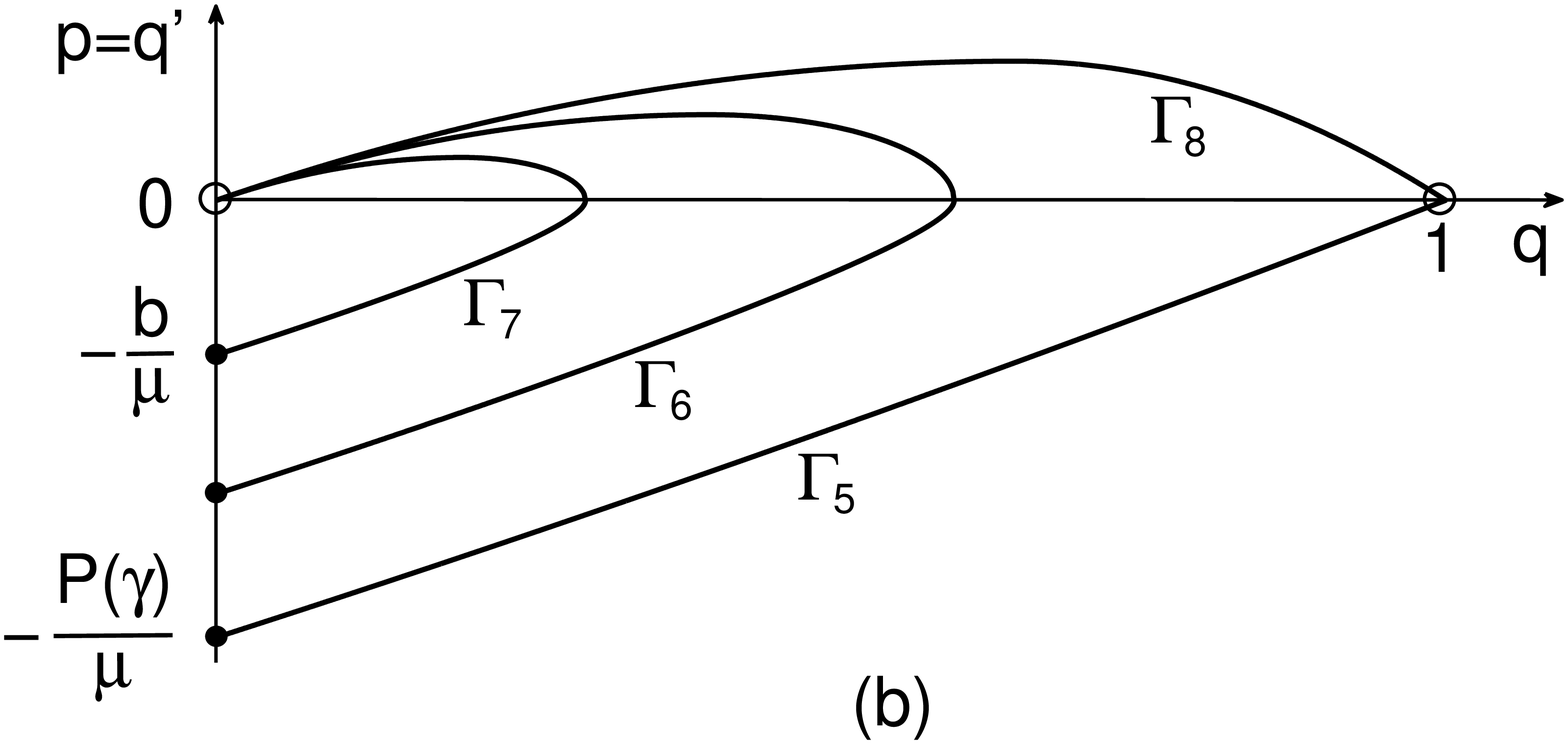}
\end{center}
\caption{\small{Trajectories corresponding to the equation $q''+\gamma q'+f(q)=0$.
(a) The case $\gamma\in(-c_0,c_0)$; (b) the case $\gamma\leqslant  -c_0$.}}
\end{figure}

\begin{figure}[!htbp]
\begin{center}
\includegraphics[scale=0.21]{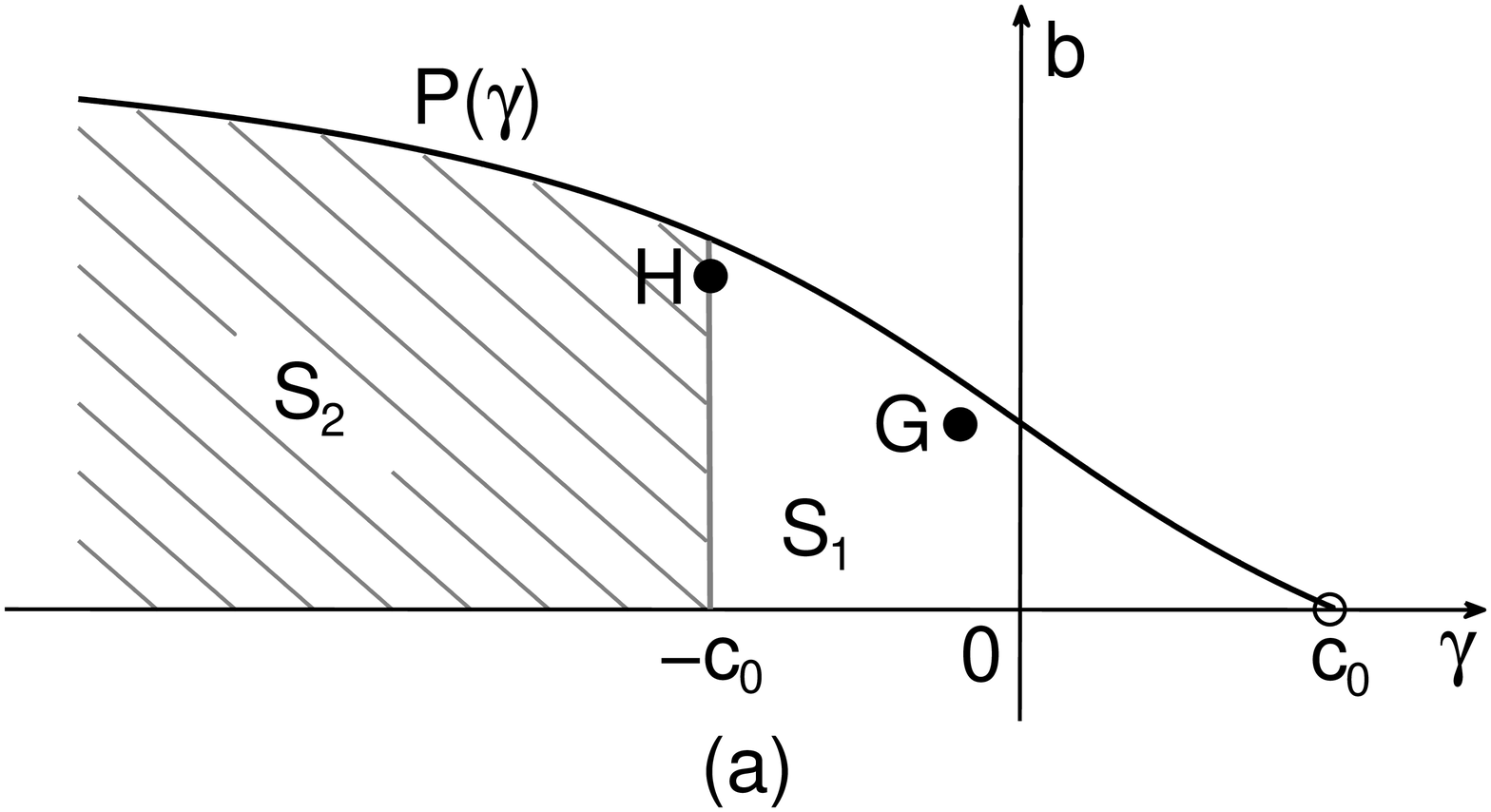}
\includegraphics[scale=0.21]{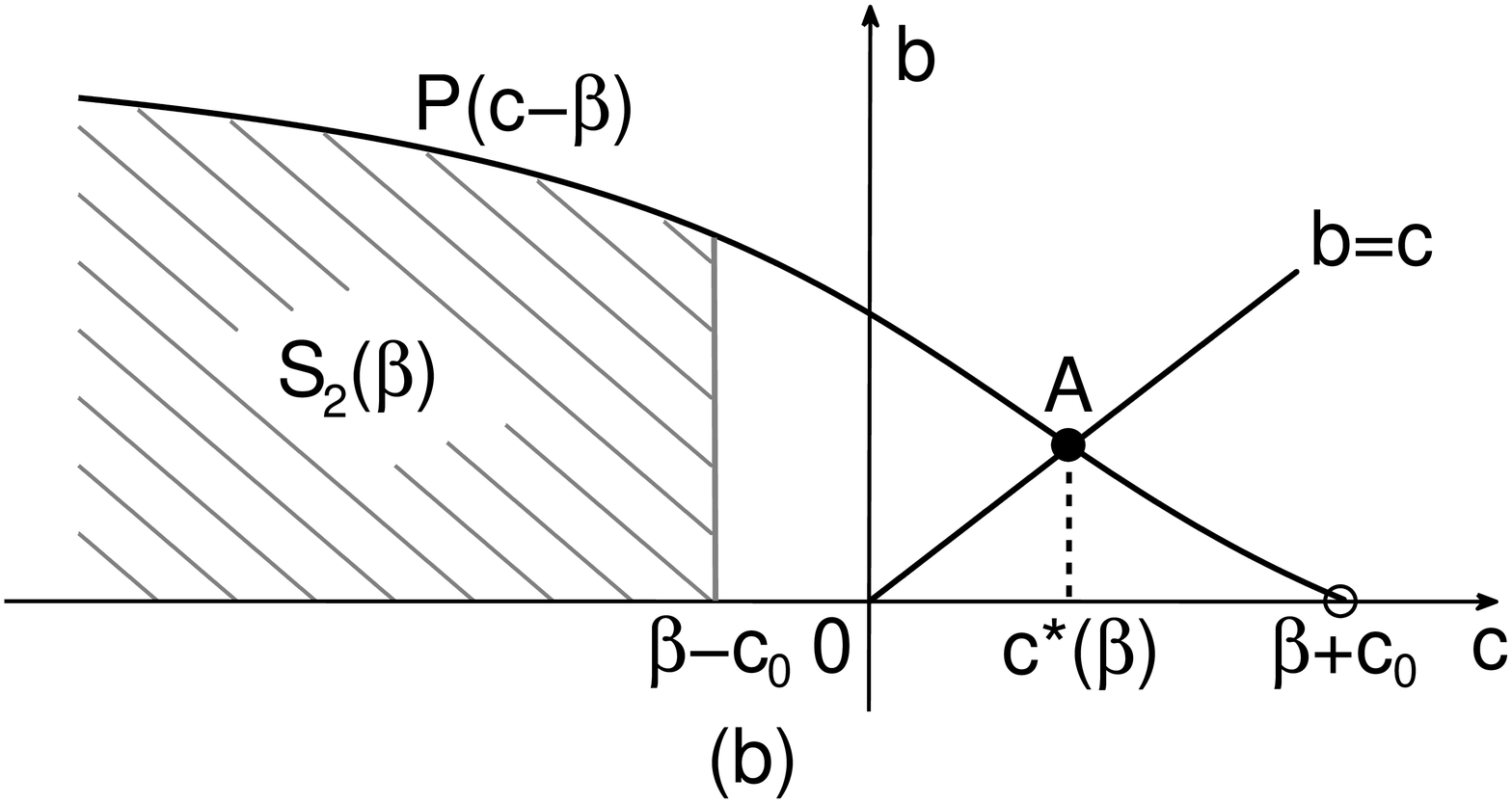}
\end{center}
\caption{\small{(a) The $\gamma$-$b$ plane about stationary solutions:
each point in $S_1$ (resp. $S_2$) corresponds to a compactly supported solution
(resp. a tadpole-like solution);
(b) the $c$-$b$ plane about traveling waves, point $A$ corresponds to a rightward traveling
semi-wave satisfying both the equation and Stefan boundary condition.}}
\end{figure}

\begin{figure}[!htbp]
\begin{center}
\includegraphics[scale=0.21]{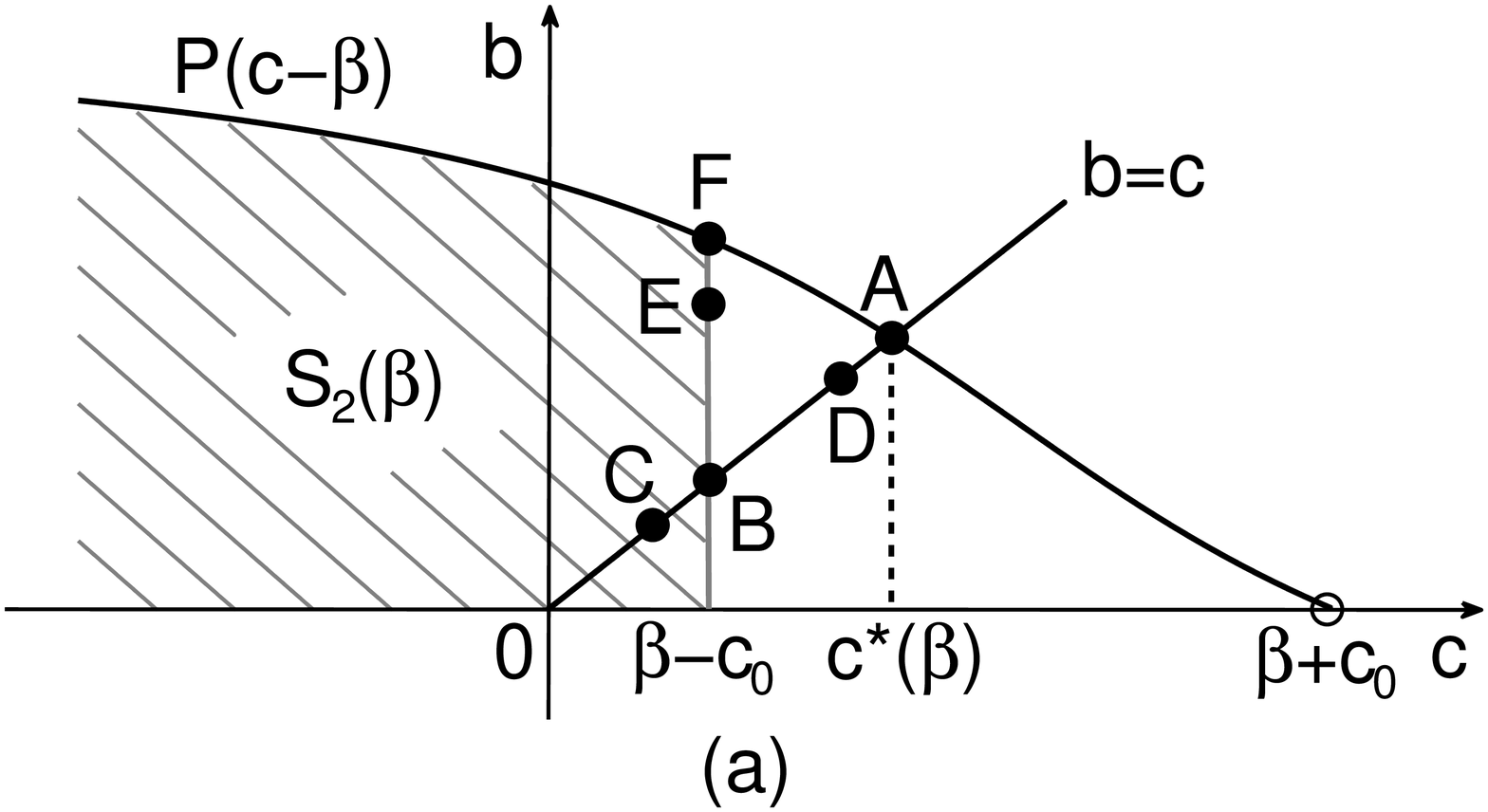}
\includegraphics[scale=0.21]{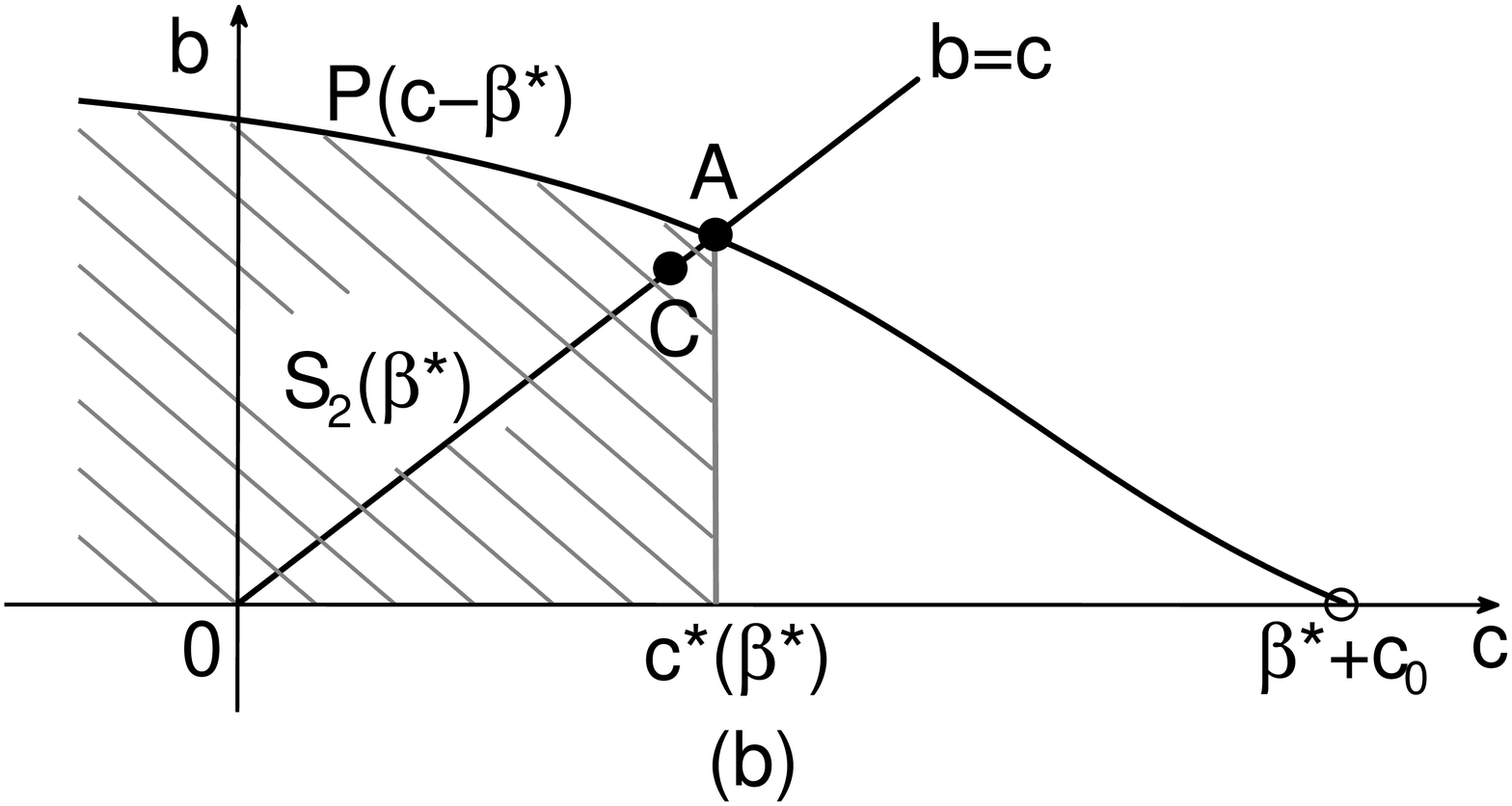}
\end{center}
\caption{\small{Points $A$ and $F$ correspond to
traveling semi-waves, points $B,\ C,\ E$ correspond to tadpole-like traveling waves, point $D$
corresponds to a compactly supported traveling wave, the waves denoted by $A,\ B,\ C,\ D$ satisfy
Stefan boundary condition. (a) The case $c_0\leqslant\beta<\beta^*$; (b) the case $\beta=\beta^*$.}}
\end{figure}

\medskip
{\rm (v)} \textbf{Tadpole-like solutions in case $\gamma \leqslant  -c_0$:}
$q(\cdot)=V(\cdot-z_0;b,\gamma)$ for any $z_0\in\R$, where for each
$b \in (0,P(\gamma))$, $V\in C^2((-\infty,0])$ is the
unique solution of \eqref{phase sol} in $(-\infty,0)$ with $V (0;b,\gamma)=0$,
$V(-\infty;b,\gamma)=0$ and $b =-\mu V'(0;b,\gamma)$ {\rm (}see $\Gamma_6$ and $\Gamma_7$ in Figure 1
(b){\rm )}. Each point $(\gamma, b)$ in the set
$S_2:= \{(\gamma,b) \mid 0<b<P(\gamma),\; \gamma \leqslant  -c_0\}$ in Figure 2 (a) corresponds to
such a tadpole-like solution $V(z;b,\gamma)$.

We call $V$ a {\it tadpole-like solution} since its graph has a big \lq\lq head"
and a boundary on the right side, and an infinite long \lq\lq tail" on the left side. Similarly,
when we construct a traveling wave with the form $V(x-ct;b,\gamma)$, we
call it a {\it tadpole-like traveling wave}.

\medskip
{\rm (vi)} \textbf{Strictly increasing solutions in $\R$ in case $\gamma\leqslant  -c_0$:}
$q(\cdot)=Q(\cdot-z_0;\gamma)$
 for any $z_0\in\R$, where $Q\in C^2(\R)$ is the unique solution of \eqref{phase sol}
in $\R$ with $Q (-\infty;\gamma)=0$, $Q (\infty ;\gamma)=1$, $Q(0;\gamma)=1/2$ and
$Q'(z;\gamma) >0$ in $\R$ {\rm (}see $\Gamma_8$ in Figure 1 (b){\rm )}.

\medskip

Each nonnegative stationary solution of \eqref{p}$_1$ is a solution of the problem
\eqref{phase sol} with $\gamma =-\beta$.
Using the above results we see that, when $\beta\in (0,c_0)$,
a bounded, nonnegative stationary solution of \eqref{p}$_1$ is either $0$, or $1$,
or a strictly decreasing solutions $U(\cdot-z_0;-\beta)\ (z_0\in \R)$ defined
on $(-\infty, z_0]$, or a compactly supported solution $W(\cdot-z_0;b,-\beta)\
(z_0\in \R)$ for some $b\in (0,P(-\beta))=(0,-\mu U'(0;-\beta))$, or a strictly increasing solution
$U_l (\cdot -z_0; -\beta)\ (z_0\in \R)$ defined on $[z_0, \infty)$. When $\beta\geqslant  c_0$,
a bounded, nonnegative stationary solution of \eqref{p}$_1$
is either $0$, or $1$, or a strictly decreasing solutions $U(\cdot-z_0;-\beta)\ (z_0\in \R)$
defined on $(-\infty, z_0]$, or a tadpole-like function $V(\cdot-z_0;b,-\beta)\
(z_0\in \R)$ for some $b\in (0,P(-\beta))=(0,-\mu U'(0;-\beta))$, or a strictly increasing solution
$Q(\cdot-z_0; -\beta)\ (z_0\in \R)$ defined in $\R$.

\subsection{Traveling waves}
If $u(t,x)=q(x-ct)$ is a traveling wave of $u_t=u_{xx} -\beta u_x +f(u)$,
then $(c,q)$ solves \eqref{phase sol} with $\gamma =c- \beta$, that is,
\begin{equation}\label{eq tw}
q''(z) + (c-\beta)q'(z) +f(q)=0.
\end{equation}
In this paper we will use several types of traveling waves which are specified
now.

\medskip

{\bf (I) Traveling wave} $Q(x-ct;c-\beta)$ for any $c\leqslant  \beta-c_0$, where
$q(z)=Q(z;c-\beta)$ satisfies \eqref{eq tw} and
\begin{equation}\label{tw R}
q(-\infty )=0,\ q(\infty)=1,\ q(0)=\frac{1}{2},\ q'(z)>0 \mbox{ for } z\in \R.
\end{equation}
The existence of such solutions has been given in the previous subsection.

\medskip
{\bf (II) Rightward traveling semi-wave} $U^* (x-c^*t)$ with $c^* \in(0,c_0+\beta)$,
where $q(z)= U^*(z) := U(z;c^*-\beta)$ satisfies \eqref{eq tw} with $c=c^*$ and
\begin{equation}\label{traveling semi-wave}
q(0)=0,\ q(-\infty)=1,\ -\mu q'(0) = c^*,\ q'(z)<0 \mbox{ for } z\in (-\infty, 0],
\end{equation}
that is, $(c^*, U^* (z) )$ is a solution of \eqref{c*} (cf. point A in Figure 2 (b) and in
Figure 3). $U^*(x-c^*t)$ is called a traveling semi-wave as in \cite{DuLou} since $U^*(z)$
is defined only on the half-line $(-\infty,0]$.

\begin{lem}\label{lem:semi-wave}
Assume $\beta>0$. Then
\begin{itemize}
\item[{\rm (i)}] there exists a unique $c^*=c^*(\beta) \in (0,c_0+\beta)$ such that
the problem \eqref{eq tw} and \eqref{traveling semi-wave} with $c=c^*$ has a solution,
which is unique  and denoted by $U^*(z)$;

\item[{\rm (ii)}] $0<\frac{\rm d}{{\rm d}\beta} c^*(\beta)<1$ for $\beta>0$;

\item[{\rm (iii)}] there exists a unique $\beta^* > c_0$ such that
\begin{equation}\label{def beta *}
c^*(\beta) -\beta +c_0>0 \ \ {\rm (} \mbox{resp. }  =0,\ <0 {\rm )}\ \mbox{ when } \ \beta <\beta^*
\ \ {\rm (} \mbox{resp. } \beta =\beta^*,\ \beta >\beta^* {\rm )}.
\end{equation}
\end{itemize}
\end{lem}
\begin{proof}
(i) For any $c< c_0 +\beta$, the problem \eqref{phase sol} with
$\gamma= c-\beta <c_0$ has a unique strictly decreasing solution $q(\cdot)=U(\cdot;c-\beta)$ in
$(-\infty,0]$, satisfying $U(0;c-\beta )=0$, $U(-\infty;c-\beta )=1$ and $U'(\cdot;c-\beta)<0$
in $(-\infty,0]$. Denote $P(c-\beta) := -\mu U'(0;c-\beta)$ as above, then $P(c-\beta)$
is strictly decreasing in $c\in (-\infty ,c_0+\beta)$,
$$
(P(c-\beta) - c) \big|_{c=0} = P(-\beta)>0\quad \mbox{and} \quad
(P(c-\beta) - c) \big|_{c=c_0+\beta-0} = -c_0 -\beta <0
$$
(see Figure 2 (b)). Hence the equation $P(c-\beta)=c$ has a unique root
$c=c^*(\beta)\in (0,c_0+\beta)$, that is,
\begin{equation}\label{eq P c}
c^*(\beta) = P( c^*(\beta) -\beta)= -\mu U'(0; c^*(\beta) -\beta).
\end{equation}

(ii) Differentiating $P( c^*(\beta) -\beta) = c^*(\beta)$ in $\beta$ and using
the fact $P'(\gamma)<0$ for $\gamma<c_0$ we have
$$
\frac{{\rm d} c^*(\beta)}{{\rm d}\beta} = \frac{-P'( c^*(\beta) -\beta)}
{1 - P'(c^*(\beta) -\beta ) } \in (0,1).
$$

(iii) Set $\beta^*:= P(-c_0)+c_0 >c_0$. Then $c=\beta^* -c_0$ is a root of
$P(c-\beta^*) =c$ in $(0,c_0+\beta^*)$. By the definition of $c^*(\beta)$ and by its uniqueness
we have $\beta^*-c_0 =c^*(\beta^*)$.
Moreover, the inequalities in (ii) shows that the function $c^*(\beta)-\beta +c_0$ is strictly
decreasing in $\beta>0$ and so it has a unique zero $\beta^*$.
This proves \eqref{def beta *}.
\end{proof}

\medskip

{\bf (III) Leftward traveling semi-wave} $U^*_l (x - c^*_lt)$ in case $\beta\in (0,c_0)$,
where $c_l^*=c^*_l (\beta) \in (\beta -c_0, 0)$, $q(z)= U^*_l (z) := U_l (z;c^*_l-\beta)$
satisfies \eqref{eq tw} with $c=c^*_l$ and
\begin{equation}\label{traveling semi-wave-left}
q(0)=0,\ q(\infty)=1,\ -\mu q'(0) = c^*_l,\ q'(z)>0 \mbox{ for } z\in [0, \infty).
\end{equation}

\noindent
For any given $\beta\in (0,c_0)$, the existence and uniqueness of such a solution
can be proved as in Lemma \ref{lem:semi-wave} (i).

\medskip

{\bf (IV) Tadpole-like traveling wave} $V(x-ct; b, c-\beta)$ in case $\beta > c_0$.
For any $c\in (0,\beta-c_0]$ and any $b\in (0,P(c-\beta))$, $V(x-ct;b,c-\beta)$ is a
tadpole-like traveling wave if the function $q(z):=V(z;b,c-\beta)$ satisfies \eqref{eq tw} and
\begin{equation}\label{tadpole tw0}
q(0)= q(-\infty)=0,\ q(z)>0 \mbox{ for } z\in (-\infty, 0) \ \mbox{and} \ -\mu q'(0) = b
\end{equation}
(cf. points $B,\ C,\ E$ in Figure 3 (a)). In particular, when $b=c$, the function
$V(z;c,c-\beta)$ is a solution of \eqref{eq tw} and
\begin{equation}\label{tadpole tw}
q(0)= q(-\infty)=0,\ q(z)>0 \mbox{ for } z\in (-\infty, 0) \ \mbox{and} \ -\mu q'(0) = c
\end{equation}
(cf. points $B,\ C$ in Figure 3 (a)). On the existence of such solutions we have the following results.

\begin{lem}\label{lem:tadpole tw bata<beta*}
Let $\beta^*$ be the constant given in Lemma \ref{lem:semi-wave}. Assume $c_0<\beta<\beta^*$. Then

\begin{itemize}
\item[\rm (i)] for any $b\in (0,P(-c_0))$, \eqref{tadpole tw0} and \eqref{eq tw} with $c=\beta-c_0$
has a unique tadpole-like solution $V(z;b,-c_0)$ (cf. points $B,\ E$ in Figure 3 (a)). Moreover,
there exists $z_b <0$ such that
\begin{equation}\label{V to Q}
V(\cdot + z_b; b,-c_0)\to Q(\cdot) \mbox{ locally uniformly in } \R,\quad \mbox{as } b\to P(-c_0);
\end{equation}

\item[\rm (ii)] $q(z) =V^*(z):= V(z; \beta -c_0,-c_0)$ is the unique tadpole-like
solution of \eqref{eq tw} and \eqref{tadpole tw} with $c=\beta-c_0$, that is, the unique solution
of \eqref{vv};

\item[\rm (iii)] for any $\delta\in (0,\beta-c_0)$,
$q(z)=V_{\delta} (z):= V(z; \beta-c_0-\delta, -c_0-\delta)$ is a tadpole-like solution of
\eqref{eq tw} and \eqref{tadpole tw} with $c=\beta-c_0-\delta$. Moreover,
$V_\delta(z)\to V^*(z)$ locally uniformly in $(-\infty, 0]$ as $\delta \to 0$ (cf. point $C$
in Figure 3 (a)).
\end{itemize}
\end{lem}
\begin{proof}
(i) Since $c_0<\beta<\beta^*$, we have $0<\beta-c_0 < c^*$ by Lemma \ref{lem:semi-wave}.
On the $c$-$b$ plane (see Figure 3 (a)), any point $E(\beta -c_0,b)$ with $b\in (0,P(-c_0))$
corresponds to a tadpole-like solution $V(z;b,-c_0)$ of \eqref{tadpole tw0} and \eqref{eq tw}
with $c=\beta -c_0$ (cf. trajectories $\Gamma_6$ or $\Gamma_7$ in Figure 1 (b). Note that such
a solution does not necessarily satisfy Stefan condition since $b$ may be not equal to $c$).
As $b\to P(\beta -c_0)$ (i.e., point $E$ moves up to $F$ in Figure 3 (a)), the trajectory
of $V(z; b,-c_0)$ approaches the union of the trajectories of $Q$ and $U^*$ (i.e.,
$\Gamma_6 \to \Gamma_5\cup \Gamma_8$). Denote $z_b:=\min \{z<0 \mid V(z;b,-c_0)=\frac12\}$.
Then the trajectory of $V(\cdot+z_b; b,-c_0)$ approaches that of $Q(\cdot)$ (since $Q(0)=\frac12$),
and so we obtain \eqref{V to Q} by continuity.

(ii) On the $c$-$b$ plane, the line $\{b=c\}$ passes through
the domain $S_2(\beta):= \{(c,b)\mid 0<b<P(c-\beta),0<c<\beta-c_0\}$ and leaves it
at a point $B(\beta-c_0, \beta-c_0)$.
This point corresponds to the desired tadpole-like solution $V^*(z):=V(z; \beta-c_0 , -c_0)$.

(iii) For any small $\delta >0$, we consider the point $C(\beta-c_0-\delta, \beta-c_0-\delta)$
on the $c$-$b$ plane. Since $C\in S_2(\beta)\cap \{b=c\}$, it corresponds to a tadpole-like
solution $V_{\delta} (z):= V(z; \beta-c_0-\delta, -c_0-\delta)$
of \eqref{eq tw} and \eqref{tadpole tw} with $c=\beta-c_0-\delta$. In particular, $V_\delta(z)$
satisfies the following initial value problem
$$
\left\{
 \begin{array}{l}
 q''(z) -(c_0 +\delta) q'(z) +f(q)=0,\ \ z<0,\\
 q(0)=0,\ \ -\mu q'(0)=\beta - c_0 -\delta.
 \end{array}
 \right.
$$
Since $V^*$ satisfies this problem with $\delta =0$ and since $V_\delta$ depends on
$\delta$ continuously, we have $V_\delta(\cdot)\to V^*(\cdot)$ as $\delta\to 0$,
uniformly in $[-M, 0]$ for any $M>0$.
This proves the lemma.
\end{proof}

In a similar way, one can prove the following lemma (cf.  Figure 3 (b)).

\begin{lem}\label{lem:tadpole tw beta=beta*}
Assume $\beta=\beta^*$. Then for any small $\delta>0$, $q(z)=V^*_{\delta} (z):=
V(z; \beta^*-c_0-\delta,  -c_0-\delta)$ is the unique  tadpole-like
solution of \eqref{eq tw} and \eqref{tadpole tw} with $\beta =\beta^*$,
$c=\beta^*-c_0-\delta$.
Moreover, $V^*_\delta(\cdot)\to U^*(\cdot)$ locally uniformly in $(-\infty, 0]$ as $\delta \to 0$,
where $U^*(z)$ is the unique solution of \eqref{eq tw} and \eqref{traveling semi-wave}
with $\beta=\beta^*$, $c^*=c^*(\beta^*) = \beta^* -c_0$.
\end{lem}

\medskip

{\bf (V) Compactly supported traveling wave} $W(x-ct;c,c-\beta)$
in case $\beta\in [c_0,\beta^*)$, where for any $c\in(\beta-c_0,c^*(\beta))$,
$q(z)=W(z;c, c-\beta)$ satisfies \eqref{eq tw} and
\begin{equation}\label{compact tw}
q(0)= q(-L(c,c-\beta))=0,\quad q(z)>0 \mbox{ in } (-L(c,c-\beta), 0) \quad \mbox{and} \quad -\mu q'(0) = c.
\end{equation}

\begin{lem}\label{lem:compact tw}
Assume $c_0 \leqslant  \beta < \beta^*$. For any $\delta\in(0,c^*(\beta)-\beta+c_0)$,
$q(z)=W_{\delta}(z):= W(z;\beta -c_0 +\delta, -c_0 +\delta )$ is the unique solution
of the problem \eqref{eq tw} and \eqref{compact tw} with $c=\beta -c_0 +\delta$.
Moreover, $L_\delta\to\infty$ and $D_\delta\to1$ as $\delta\to c^*(\beta)-\beta+c_0$,
where $L_\delta:=L(\beta-c_0+\delta,-c_0+\delta)$ denotes the width of the support of
$W_\delta(\cdot)$, and $D_\delta$ denotes its height.
\end{lem}
\begin{proof}
We only prove the case $\beta\in(c_0,\beta^*)$, the proof for the case $\beta=c_0$ is
similar.

When $c_0<\beta<\beta^*$, we have $c^*(\beta)-\beta+c_0>0$ by Lemma \ref{lem:semi-wave}.
On the $c$-$b$ plane (see Figure 3 (a)), the line $b=c$ leaves the domain $S_2(\beta)$
and enters $S_1(\beta):=\{(c,b)\mid 0<b<P(c-\beta),\ \beta-c_0<c<\beta+c_0\}$ at a point
$B(\beta-c_0,\beta-c_0)$, then it passes through $S_1(\beta)$ and leaves it finally at
$A(c^*(\beta),c^*(\beta))$. For any $\delta\in(0,c^*(\beta)-\beta+c_0)$, the point
$(c,b)=(\beta-c_0+\delta,\beta-c_0+\delta)$ is on the line segment $AB$ (cf. point $D$
in Figure 3 (a)), it corresponds to a trajectory like $\Gamma_2$ on the $q$-$p$ phase
plane, and so it defines a compactly supported function
$W_\delta(z):=W(z;\beta-c_0+\delta,-c_0+\delta)$. As $\delta\to c^*(\beta)-\beta+c_0$,
the point $(\beta-c_0+\delta,\beta-c_0+\delta)$
approaches point $A$ in Figure 3 (a), this implies that its corresponding trajectory
approaches the union of the trajectories of $U^*(\cdot)$ and $U_l(\cdot;c^*(\beta)-\beta)$ (i.e.,
$\Gamma_2\to\Gamma_1\cup\Gamma_4$ in Figure 1 (a)). Therefore, the corresponding function
$W_\delta$ satisfies $\max\limits_{-L_\delta\leqslant z\leqslant0} W_\delta(z)=D_\delta\to1$,
and the width $L_\delta$ of its support tends to $\infty$ as $\delta\to c^*(\beta)-\beta+c_0$.
\end{proof}

\subsection{Zero number arguments}
In what follows, we use $\mathcal{Z}_I[w(\cdot)]$ to denote the number of zeros of
a continuous function $w(\cdot)$ defined in $I\subset\R$.
The following lemma is an easy consequence of the proofs of Theorems C and D in Angenent \cite{A}.

\begin{lem}\label{angenent}
Let $u:[0,T]\times [0, 1]\to \R$ be a bounded classical solution of
\begin{equation}
\label{linear}
u_t=a(t,x)u_{xx}+b(t,x)u_x+c(t,x)u
\end{equation}
with boundary conditions
\[
u(t,0)=l_0(t), \;u(t, 1)=l_1(t),
\]
 where $l_0, l_1\in C^1([0,T])$, and each function is either identically zero
 or never zero for $t\in [0,T]$. In the special case where $l_0(t),\ l_1(t)\equiv 0$
 we assume further that $u(t,\cdot) \not\equiv 0$ for each $t\in [0,T]$.
 Suppose also that
\[
a, 1/a, a_t, a_x, a_{xx}, b, b_t, b_x, c \in L^\infty, \mbox{ and } u(0,\cdot)\not\equiv 0 \mbox{ when $l_0=l_1\equiv 0$}.
\]
 Then for each $t\in (0, T]$, $\mathcal{Z}_{[0,1]} [u(t,\cdot)]<\infty$.
Moreover, $\mathcal{Z}_{[0,1]} [u(t,\cdot)]$ is nonincreasing in $t$ for
$t\in (0, T]$, and if for some $t_0\in(0, T]$ the function $u(t_0,\cdot)$ has
a degenerate zero $x_0\in [0,1]$, then
$\mathcal{Z}_{[0,1]} [u(t_1, \cdot)] >\mathcal{Z}_{[0,1]} [u(t_2, \cdot)] $ for all
$t_1, t_2\in (0,T]$ satisfying $t_1<t_0<t_2$.
\end{lem}

For convenience of applications in this paper we give a variant of Lemma \ref{angenent}.

\begin{lem}\label{zero-number}
 Let $\xi_1(t)<\xi_2(t)$ be two continuous functions for $t\in (t_0, t_1)$. If $u(t,x)$
 is a continuous function for $t\in (t_0, t_1)$ and $x\in J(t):= [\xi_1(t),\xi_2(t)]$,
 and satisfies \eqref{linear} in the classical sense for such $(t,x)$, with
\[
u(t,\xi_1(t))\not=0,\; u(t, \xi_2(t))\not=0 \mbox{ for } t\in (t_0, t_1),
\]
then for each $t\in (t_0, t_1)$, $\mathcal{Z}_{J(t)} [u(t,\cdot)] <\infty$.
Moreover $\mathcal{Z}_{J(t)}[u(t,\cdot)]$ is nonincreasing in $t$ for $t\in (t_0, t_1)$,
and if for some $s\in (t_0, t_1)$ the function $u(s,\cdot)$ has a degenerate zero $x_0\in
J(s)$, then $\mathcal{Z}_{J(s_1)} [u(s_1,\cdot)]>
\mathcal{Z}_{J(s_2)} [u(s_2,\cdot)]$ for all $s_1, s_2$ satisfying $t_0<s_1<s<s_2<t_1$.
\end{lem}
\begin{proof}
For any given $t^*\in (t_0, t_1)$, we can find $\epsilon>0$ and $\delta>0$ small such that
$u(t,x)\not=0$ for $t\in T_{t^*}:=[t^*-\delta, t^*+\delta]\subset (t_0, t_1)$
and $x\in [\xi_1(t),\xi_1(t^*)+\epsilon]\cup [ \xi_2(t^*)-\epsilon, \xi_2(t)]$. Hence we may
apply Lemma \ref{angenent} with $[0,T]\times [0,1]$ replaced by
$T_{t^*}\times [\xi_1(t^*)+\epsilon,\xi_2(t^*)-\epsilon]$ to see that the
conclusions for $\mathcal{Z}_{J(t)} [u(t,\cdot)]$ hold for $t\in T_{t^*}$.
Since any compact subinterval of $(t_0, t_1)$ can be covered by finitely many such $T_{t^*}$,
we see that $\mathcal{Z}_{J(t)}[u(t,\cdot)]$ has the required properties over any compact
subinterval of $(t_0, t_1)$. It follows that $\mathcal{Z}_{J(t)}[u(t,\cdot)]$ has the required
properties for $t\in (t_0, t_1)$.
\end{proof}

In our approach we will compare the solution $u$ of \eqref{p} with traveling semi-wave
$U^*$ or tadpole-like traveling wave $V^*$ or compactly supported traveling wave $W_\delta$
by studying the number of their intersection points. Now we give some preliminary results.

We use $\Psi(x-ct-C)$ (for some $c>0$ and some $C\in\R$) to represent one of the
traveling waves $U^*(x-c^*t-C)$, $V^*(x-(\beta-c_0)t-C)$ and $W_\delta(x-(\beta-c_0+\delta)t-C)$.
Denote the support of $\Psi(x-ct-C)$ by $[k_1(t),k_2(t)]$, where $k_2(t)=ct+C$, $k_1(t)=-\infty$
in case $\Psi=U^*$ or $\Psi=V^*$, and $k_1(t)=k_2(t)-L_\delta$ in case $\Psi=W_\delta$. Denote
$$
r(t):=\min\{h(t),k_2(t)\},\quad R(t):=\max\{h(t),k_2(t)\},\quad l(t):=\max\{g(t),k_1(t)\}
$$
and
$$
\eta(t,x):=u(t,x)-\Psi(x-ct-C),\quad x\in J(t):=[l(t),r(t)],\ t\in(t_1,t_2)
$$
(here we only consider the case where $J(t)\neq\emptyset $ for each $t\in(t_1,t_2)$, otherwise, $u$ and
$\Psi$ has no common domain and so there is no need to compare them). We notice that $\eta$
satisfies
$$
\eta_t=\eta_{xx}-\beta\eta_x+c(t,x)\eta\quad \mbox{for }x\in(l(t),r(t)),\ t\in(t_1,t_2)
$$
with $c(t,x):=[f(u(t,x))-f(\Psi(x-ct-C))]/\eta(t,x)$ when $\eta(t,x)\neq0$, and $c(t,x)=0$
otherwise. Using Lemmas \ref{angenent} and \ref{zero-number} one can
obtain the following result on the number of zeros of $\eta(t,\cdot)$.

\begin{lem}\label{lem:zeros between u and Psi}
For any given $C\in\R$, let $r(t),\ R(t),\ l(t)$ and $\eta$ be defined as above.
Then
\begin{itemize}
  \item [(i)] $\mathcal{Z}_{J(t)}[\eta(t,\cdot)]$ is finite and nonincreasing in
  $t\in(t_1,t_2)$;
  \item [(ii)] if $t_0\in(t_1,t_2)$ such that $r(t_0)=R(t_0)$, or $\eta(t_0,\cdot)$ has
  a degenerate zero in the interior of $J(t_0)$, then
  $\mathcal{Z}_{J(\tau_1)}[\eta(\tau_1,\cdot)]>\mathcal{Z}_{J(\tau_2)}[\eta(\tau_2,\cdot)]$
  for any $t_1<\tau_1<t_0<\tau_2<t_2$.
\end{itemize}
\end{lem}

\noindent
{\it Sketch of the proof}.\ The proof is essentially identical to that of Lemma 2.3 in
\cite{DLZ}. We give a sketch here for the readers' convenience.

Note that when $h(t)=k_2(t) = r(t)=R(t)$, $x=r(t)$ becomes a degenerate zero of
$\eta$ on the boundary since both $u$ and $\Psi$ satisfy Stefan condition on their right
boundaries. We claim that $h(t)\equiv k_2(t)$ in an interval $(t_3,t_4)\subset(t_1,t_2)$
is impossible. For, otherwise, we can consider $\zeta=\eta e^{-\frac{\beta}{2} x}$ instead of
$\eta$, which satisfies an equation without advection and so can be extended outside $r(t)$
as an odd function with respect to $x=r(t)$. For this extended function $x=r(t)$ is an interior degenerate zero
in time interval $(t_3,t_4)$, this contradicts Lemma \ref{angenent}.
On the other hand, $g(t)=k_1(t)$ at most once, and only possible in case $\Psi=W_\delta$. Therefore, the set of
the times when $r(t)=R(t)$ or $g(t)=k_1(t)$ is a nowhere dense set, and for other times we have
$\mathcal{Z}_{J(t)}[\eta(t,\cdot)]<\infty$ by Lemma \ref{angenent}.

Assume $r(t_0)=R(t_0)$ and $r(t)<R(t)$
for $t\in[t_0-\epsilon,t_0)$. Assume further that $\eta(t,\cdot)$ has nondegenerate
zeros $\{z_i(t)\}_{i=1}^m$ with
$$l(t)<z_1(t)<z_2(t)<\cdots<z_m(t)<r(t),\quad t\in[t_0-\epsilon,t_0).$$
Then by \cite[Theorem 2]{F} one can prove that $\lim\limits_{t\to t_0}z_i(t)$ exist
(denoted by $\bar{z}_i$) and $\{\bar{z}_i\}_{i=1}^m$ are the only zeros of $\eta(t_0,\cdot)$.
Moreover, by the maximum principle we have $\bar{z}_m=r(t_0)$, that is, the largest zero
$z_m(t)$ tends to the right boundary. Then using the maximum principle again in the domain
$\{(t,x)\mid t_0<t<t_0+\epsilon_1,r(t)-\epsilon_1<z<r(t)\}$ for some small $\epsilon_1$,
we can show that the boundary zero $r(t_0)$ disappear immediately after time $t_0$. In
summary,
$$
\mathcal{Z}_{J(\tau_1)}[\eta(\tau_1,\cdot)]=m\geqslant\mathcal{Z}_{J(t_0)}[\eta(t_0,\cdot)]>
\mathcal{Z}_{J(\tau_2)}[\eta(\tau_2,\cdot)]
$$
for $t_0-\epsilon< \tau_1 <t_0 < \tau_2 <t_0+\epsilon_1$.
{\hfill $\Box$}

\medskip

As can be expected, the presence of the advection makes the maximum points prefer
to move rightward. Indeed we can show that the local maximum points concentrate near the
right boundary under certain conditions.

Using zero number properties Lemma \ref{zero-number} to $u_x$, we see that $u_x(t,\cdot)$
has only nondegenerate zeros for all large $t$. Hence, $u(t,\cdot)$ has
fixed number of (nondegenerate) local maximum points for large $t$.

\begin{lem}\label{lem:max at right}
Assume, for some $T\geqslant 0$, $u(t,\cdot)$ has exactly $N$ ($N$ is a positive integer) local
maximum points $\{\xi_i(t)\}_{i=1}^{N}$ for all $t\geqslant  T$, with
$$g(t)<\xi_1(t)<\xi_2(t)<\cdots<\xi_N(t)<h(t).$$
If $N\geqslant 2$, then
\begin{equation}\label{xi1(t)>=eta(t)}
  \xi_1(t)\geqslant \beta\cdot(t-T)+C\quad \mbox{for }T\leqslant  t<T_\infty,
\end{equation}
for some $C\in\R$, where
\begin{equation}\label{monotonicity outside}
 T_\infty= \left\{
\begin{array}{ll}
     \inf \mathcal{T},\ \ & \mbox{ if } \mathcal{T}:=\{t\geqslant  T\mid\beta\cdot(t-T)+C=h(t)\}\neq\emptyset,\\
      \infty, \ \ & \mbox{ if } \mathcal{T}=\emptyset.
\end{array}
\right.
\end{equation}
\end{lem}
\begin{proof}
Choose $C=\frac{1}{2}[g(T)+\xi_1(T)]$ and define $\rho(t):=\beta(t-T)+C$, then
$\rho(T)=C<\xi_1(T)$. Hence $T_1:=\inf\{s\geqslant T\mid\rho(s)=\xi_1(s)\}>T$.

If $T_1=T_\infty$, then \eqref{xi1(t)>=eta(t)} holds. Now we assume $T<T_1<T_\infty$.
By the definitions of $\xi_1(t)$ and $T_1$ we have
\begin{equation}\label{monotonicity outside for [T0,T1)}
  u_x(t,x)>0\quad \mbox{for }x\in[g(t),\rho(t)],\ T\leqslant  t<T_1,
\end{equation}
and
\begin{equation}\label{u_x(T1)=0}
u_x(T_1,\rho(T_1))=u_x(T_1,\xi_1(T_1))=0.
\end{equation}
Define
$$\zeta(t,x):=u(t,x)-u(t,2\rho(t)-x)\quad \mbox{for }x\in[l(t),\rho(t)],\ t\geqslant  T,$$
with $l(t):=\max\{g(t),2\rho(t)-\xi_N(t)\}$. A direct calculation shows that
$$\zeta_t=\zeta_{xx}-\beta\zeta_x+c\zeta\quad \mbox{for }x\in[l(t),\rho(t)],\ t\in[T,T_1),$$
where $c$ is a bounded function. Since
\begin{equation*}
\left\{
\begin{array}{ll}
 \zeta(T,x)<0 &\mbox{ for }x\in[l(T),\rho(T)],\\
 \zeta(t,\rho(t))=0 &\mbox{ for }t\in[T,T_1],\\
 \zeta(t,l(t))<0&\mbox{ for } t\in[T,T_1].
\end{array}
\right.
\end{equation*}
The last inequality follows from the following analysis. By the monotonicity of $u_x$
and the fact that $\xi_N(t)$ is the rightmost local maximum point, we see that, if $\zeta(t_1, l(t_1))=0$ for some
$t_1\in(T,T_1]$ (denote $t_1$ the first of such times), then $\zeta(t_1, l(t_1)+\varepsilon )>0$
for some small $\varepsilon >0$. This, however, is impossible
since $\zeta(t_1, x)<0$ for all $x\in (l(t_1), \rho(t_1))$ by the
maximum principle and by the fact that $\zeta(t,l(t))<0$ for $t\in[T,t_1)$.

Now we use the Hopf lemma for $\zeta$ in the domain
$\{(t,x)\mid l(t)\leqslant  x\leqslant  \rho(t),\ T\leqslant  t\leqslant  T_1\}$ to derive
$$\zeta(T_1,\rho(T_1))=0\quad \mbox{and}\quad \zeta_x(T_1,\rho (T_1))>0.$$
The latter, however, contradicts \eqref{u_x(T1)=0}. This proves $T_1=T_\infty$.
\end{proof}

\subsection{Upper bound estimate}
In order to study the convergence of $u$ we give some precise upper bound estimates
for the solutions. In this paper we always write
\begin{equation}\label{def A}
A:= 2 \max\big\{ 1, \|u_0\|_{L^\infty ([-h_0, h_0])}\big\}.
\end{equation}

\medskip

1. {\it Bound of $u$ near the free boundary $x=h(t)$}. For any $\delta \in (0,-f'(1))$, set
$$
\bar{g}(t):= g(t),\quad \bar{h}(t):= c^* t - M  e^{-\delta t} +H \ \mbox{ for some } M,\ H>0,
$$
and
$$
\bar{u}(t,x) := (1+Ae^{-\delta t}) U^* (x-\bar{h}(t) ) \ \mbox{ for } x\leqslant \bar{h}(t),\ t>0,
$$
where $U^*$ is the rightward traveling semi-wave (the solution of \eqref{c*}). A direct calculation as in the proof of
\cite[Lemma 3.2]{DMZ} shows that $(\bar{u},\bar{g},\bar{h})$
is an upper solution of \eqref{p} provided $M,H>0$ are large. Hence we have
\begin{equation}\label{right bound}
u(t,x) \leqslant  \bar{u}(t, x) \mbox{ for } x\in [g(t), h(t)] \mbox{ and } t>0,\quad h(t) \leqslant  \bar{h}(t)
\mbox{ for } t>0.
\end{equation}

\bigskip
2. {\it Bound of $u$ in case $\beta \geqslant  c_0$}. We define a function
$f_A(s)\in C^2([0,\infty),\R)$ such that
\begin{equation}\label{fA}
f_A (s) \left\{
   \begin{array}{ll}
    = f'(0) s, & 0\leqslant  s \leqslant  1,\\
    >0 , & 1<s<A,\\
    <0 , & s> A,
\end{array}
\right.
\quad f'_A (A)<0,\quad  f(s)\leqslant  f_A (s)\leqslant  f'(0) s\ \mbox{for}\ s\geqslant  0.
\end{equation}
Denote by $Q_A(z)$ the unique solution of \eqref{c0} with $c=c_0$, with $f$ replaced by $f_A$
and $q(+\infty)=1$ replaced by $q(+\infty)=A$. Then by the comparison principle we have
\begin{equation}\label{left bound beta large}
u(t, x) \leqslant  Q_A (x- (\beta -c_0)t + x_0)\quad \mbox{for } x\in [g(t),h(t)],\; t>0,
\end{equation}
provided $x_0 >0$ is large enough.

Using $f_A$ we consider the Cauchy problem:
$$
\left\{
\begin{array}{l}
 (u_1)_t = (u_1)_{xx} - \beta (u_1)_x + f_A ( u_1), \quad   x\in \R,\; t>0,\\
 u_1 (0,x) = \tilde{u}_0 (x)
  := \left\{
     \begin{array}{ll}
      u_0(x), &  x\in [-h_0, h_0],\\
      0,       &  |x|>h_0.
     \end{array}
     \right.
\end{array}
\right.
$$
Since $f_A \geqslant  f$, by the comparison principle we have
\begin{equation}\label{u < u1}
u(t,x)\leqslant  u_1(t,x),\quad x\in [g(t), h(t)],\; t>0.
\end{equation}

Set $y:= x- \beta t$ and $u_2 (t,y):= u_1 (t,x)= u_1 (t,y+\beta t )$.
Then $u_2 $ is a solution of
$$
\left\{
\begin{array}{l}
 (u_2)_t = (u_2)_{yy} + f_A ( u_2), \quad   y\in \R,\; t>0,\\
 u_2 (0,y) = \tilde{u}_0 (y)\in [0,A], \quad  y\in \R.
\end{array}
\right.
$$
Set $\tilde{u}(t,y) := \frac{1}{A}u_2(t,-y) $, then $\tilde{u}(t,y)$ is the solution of
$$
\left\{
\begin{array}{l}
 \tilde{u}_t  = \tilde{u}_{yy}+ \frac{1}{A} f_A \big(A \tilde{u} \big), \quad   y\in \R, \; t>0,\\
 \tilde{u}(0,y) = \frac{1}{A}\tilde{u}_0 (-y) \in [0,1], \quad y\in \R.
\end{array}
\right.
$$
By \eqref{u < u1} and by the definitions of $u_2$ and $\tilde{u}$ we have
\begin{equation}\label{u < tildeu}
u(t,x) \leqslant  A \tilde{u}(t, \beta t -x),\quad x\in [g(t),h(t)],\; t>0.
\end{equation}

On the other hand, by Proposition 2.3  in \cite{HNRR} and its proof,
there exist $C_1,C_2,t_0$ depending on $u_0$ such that
$$
\tilde{u} \Big(t,c_0 t - \frac{3}{c_0} \ln \big(1+\frac{t}{t_0}\big) +y \Big) \leqslant  C_1 Z (t,y ), \quad y\geqslant  h_0,\; t>0,
$$
where
\begin{equation}\label{def Z}
Z(t,y):= \frac{1}{\sqrt{t_0}}y e^{- \frac{c_0}{2} y} \big[ C_2e^{\frac{-y^2}{4(t+t_0)}} +h(t,y) \big],\quad y\in \R, t>0,
\end{equation}
with $h(t,y)$ satisfying
$$
\limsup\limits_{t\to \infty}  \sup\limits_{0\leqslant  y \leqslant  \sqrt{t+1}} |h(t,y)| \leqslant   \frac{C_2}{2}.
$$

In particular, there exists $C_3 >0$ such that
$$
\tilde{u} \Big(t,c_0 t - \frac{3}{c_0} \ln \big(1+\frac{t}{t_0}\big) +y \Big) \leqslant
C_3 e^{- \frac{5c_0}{12} y},\quad  y\in [0, \sqrt{t+1}].
$$
Combining with \eqref{u < tildeu} we have
\begin{equation}\label{left 1}
u (t,x) \leqslant  C_4 e^{-\frac{5c_0}{12} (Y(t)-x)}\quad \mbox{for } Y(t)-\sqrt{t+1} \leqslant  x\leqslant
\min \big\{Y(t), h(t)\big\},\; t\gg1,
\end{equation}
where $C_4>0$ is a constant and
\begin{equation}\label{def Y}
Y(t):= (\beta-c_0) t + \frac{3}{c_0} \ln \big(1+\frac{t}{t_0}\big),\quad t>0.
\end{equation}

\begin{cor}\label{cor:-h0 h(t)}
{\rm (i)} Assume $\beta =c_0$. Then there exists $C$ depending on $u_0,c_0,h_0$
such that
$$
u(t, -h_0)\leqslant  C t^{-\frac{5}{4}}\quad  \mbox{for } t>0 \mbox{ large};
$$
{\rm (ii)} Assume $\beta\in(c_0,\beta^*]$ and $h(t) =(\beta-c_0) t +O(1)$. Then there exists
$C$ depending on $u_0, c_0, h_0$ such that
$$
u (t,x ) \leqslant  C t^{-\frac{5}{4}}\quad  \mbox{for } x\in \Big[h(t)-\frac{\pi}{2}, h(t)\Big] \mbox{ and } t>0 \mbox{ large}.
$$
\end{cor}
\begin{proof}
We only prove (ii) since (i) can be proved similarly. In Case (ii),
$$
Y(t)=(\beta-c_0) t  + \frac{3}{c_0} \ln \big(1+\frac{t}{t_0}\big),
$$
and so, for any $x\in [h(t)-\frac{\pi}{2}, h(t) ]$, we have
$$
Y(t) -x = \frac{3}{c_0} \ln \big(1+\frac{t}{t_0}\big) + O(1)  \in [0, \sqrt{t+1}],\quad
\mbox{when } t\gg 1.
$$
Using \eqref{left 1} we have
\begin{equation}\label{u h0<}
u (t, x)\leqslant  C t^{-\frac{5}{4}}\quad \mbox{when } t\gg 1,
\end{equation}
where $C>0$ are some constants depending on $u_0,c_0,h_0$.
\end{proof}

\begin{remark}
\rm For any given $m\in (0,1)$, denote
$$
\chi(t):= \min\{x\in [g(t),h(t)] \mid u(t,x)=m\} \mbox{ and }
\widetilde{\chi}(t):= \min\{y\in \R \mid \tilde{u}(t,y)=m/A\}.
$$
Then by \cite[Theorem 1.1]{HNRR} we have
$$
c_0 t -\frac{3}{c_0} \ln t -C \leqslant \widetilde{\chi}(t)
\leqslant c_0 t -\frac{3}{c_0} \ln t +C ,\quad t\gg1,
$$
for some $C>0$. Hence by \eqref{u < tildeu} we have
\begin{equation}\label{chi(t)>=}
\chi(t)\geqslant \beta t -\widetilde{\chi}(t) \geqslant (\beta -c_0)t +\frac{3}{c_0} \ln t -C,\quad t\gg1.
\end{equation}
\end{remark}

\section{Influence of $\beta$ on the long time behavior of solutions}
In this section we consider the influence of $\beta$ on the long time behavior of the solutions. In subsection 1 we give a locally uniformly convergence result. In
subsection 2 we consider the small advection $\beta \in (0,c_0)$ and prove Theorem
\ref{thm:small beta}. In subsection 3 we first prove the boundedness of $g_\infty$
for $\beta \geqslant  c_0$, the boundedness of $h_\infty$ for $\beta \geqslant  \beta^*$,
and then prove Theorem \ref{thm:large beta} for large advection $\beta \geqslant  \beta^*$.
In subsection 4, we consider \eqref{p} with medium-sized advection $\beta\in [c_0, \beta^*)$
and prove Theorems \ref{thm:middle beta} and \ref{thm:beta=c0}. The argument for the case $\beta\in [c_0, \beta^*)$
are longer and much more complicated than the cases with small or large advection.

\subsection{Convergence result}
First we give a locally uniformly convergence result for $\beta \geqslant  c_0$.

\begin{lem}\label{lem:beta>c0 u to 0}
Assume $\beta\geqslant  c_0$. Then $u(t,\cdot)$ converges as $t\to\infty$ to $0$ locally
uniformly in $I_\infty$.
\end{lem}
\begin{proof}
When $\beta >c_0$, the conclusion follows easily from \eqref{left bound beta large} since the
upper solution $Q_A(x-(\beta -c_0)t +x_0)$ is a rightward traveling wave with positive speed
$\beta-c_0$ and $Q_A(z)\to 0$ as $z\to -\infty$.

Assume $\beta =c_0$. Then for any $[a,b]\subset I_\infty$, when $t$ is sufficiently large we
have
$$
\frac{3}{c_0} \ln \big( 1+\frac{t}{t_0} \big) -\sqrt{t+1} < x< \frac{3}{c_0} \ln \big( 1+\frac{t}{t_0} \big) \quad \mbox{ for any } x\in [a,b],
$$
and so by \eqref{left 1}, for any $x\in [a,b]$,
$$
u(t,x) \leqslant  C_4 e^{-\frac{5c_0}{12}  \big[ \frac{3}{c_0} \ln \big( 1+\frac{t}{t_0} \big) -b\big] } \to 0,\quad \mbox{as } t\to \infty.
$$
This proved the lemma.
\end{proof}

\begin{thm}\label{thm:convergence}
Let $(u,g, h)$ be a time-global solution of \eqref{p}. Then as $t\to\infty$, $u(t,\cdot)$ converges to $0$
or to $1$ locally uniformly in $I_\infty$ when $\beta\in(0,c_0)$; $u(t,\cdot)$ converges
to $0$ locally uniformly in $I_\infty$ when $\beta\geqslant c_0$.

Moreover, $\lim_{t\to\infty}\|u(t,\cdot)\|_{L^{\infty}([g(t),h(t)])}=0$
if $I_{\infty}$ is a bounded interval.
\end{thm}
\begin{proof}
Using a similar argument as proving \cite[Theorem 1.1]{DM}, \cite[Theorem 1.1]{DuLou},
\cite[Theorem 1.1]{LL1}, one can show that $u(t,\cdot)$ converges, as $t\to\infty$, to a
stationary solution, that is, a solution $v$ of $v_{xx}-\beta v_x+f(v)=0$, locally uniformly
in $x\in I_\infty$. Moreover, one can show by Hopf lemma that $v=0$ when $g_\infty>-\infty$
or $h_\infty <\infty$. In other word, the limit $v$ can not be a non-trivial solution with
endpoint. Therefore, when $\beta\in (0,c_0)$, the only possible choice for the
$\omega$-limit of $u$ in the topology of $L^\infty_{\rm loc} (I_\infty)$ is $0$ or $1$;
when $\beta\geqslant  c_0$, the conclusion follows from Lemma \ref{lem:beta>c0 u to 0}.
Finally, when $I_\infty$ is bounded, the uniform convergence for $u$ is also
proved in the same way as that in \cite{DuLou,LL1}.
\end{proof}

\subsection{Problem with small advection: $0<\beta<c_0$}
In a similar way as proving \cite[Lemma 2.2]{GLL1} and \cite[Theorem 3.2, Corollary 4.5]{DuLou}
one can show the following conditions for spreading and for vanishing.

\begin{lem}\label{lem:<c0}
Assume $\beta\in(0,c_0)$. Let $(u,g,h)$ be a solution of \eqref{p}.

{\rm (i)} If $h_0<H^*:= \frac{\pi}{\sqrt{{c^2_0}-\beta^2}}$ and if $\|u_0\|_{L^\infty([-h_0,h_0])}$
is sufficiently small, then vanishing happens;

{\rm (ii)} if $h_0\geqslant  H^*$, then spreading happens.
\end{lem}

This lemma implies that, when $\beta\in (0,c_0)$, $2H^*$ is a critical width of the
interval $I(t)$. Spreading happens if and only if $|I_\infty| >2H^*$. This extends
the results in \cite{DuLin, DuLou} for $\beta =0$, where it was shown that the
critical width $2H^*= \frac{2\pi}{c_0}=\frac{\pi}{\sqrt{f'(0)}}$.

\medskip

\noindent
{\bf Proof of Theorem \ref{thm:small beta}:}\ By Lemma \ref{lem:<c0} (ii) we see that
spreading happens if $|I_\infty|>2H^*$. By the definition of spreading, this implies
that $I_\infty=\R$. Hence both the case $g_\infty>-\infty$, $h_\infty=\infty$ and the
case $g_\infty=-\infty$, $h_\infty<\infty$ are impossible. $I_\infty$ is either a
bounded interval with width $|I_\infty|\leqslant2H^*$ or the whole line $\R$. Using
Theorem \ref{thm:convergence} and Lemma \ref{lem:<c0} again, we can get the
spreading-vanishing dichotomy result for the long-time behavior of the solutions of
\eqref{p}. Then the sharp threshold of the initial data $\sigma\phi$ can be proved in a
similar way as in \cite[Theorem 5.2]{DuLou}.
{\hfill $\Box$}

\subsection{Boundedness of $g_\infty$ and $h_\infty$, the proof of Theorem \ref{thm:large beta}}

Whether $g_\infty$ and $h_\infty$ are bounded or not is also a part of the conclusions
in the long time behavior of $(u,g,h)$. In this subsection we show that $g_\infty>-\infty$
if $\beta \geqslant  c_0$, and $h_\infty<\infty$ if $\beta\geqslant\beta^*$.

We will prove these conclusions by using Corollary \ref{cor:-h0 h(t)}. For this purpose
we need the monotonicity of $u_x$. When $\beta =0$, Du and Lou \cite{DuLou} proved the
monotonicity of $u_x$ in $[h_0, h(t)]$ and in $[g(t),-h_0]$. When $\beta>0$
we find that this is true only on the left side: $x\in[g(t),-h_0]$.

\begin{lem}\label{lem:center}
Assume $(u,g, h)$ is a solution of \eqref{p}. Then
\begin{equation}\label{g+h}
g(t)+h(t)>-2h_0 \quad\mbox{for all } t>0,
\end{equation}
\begin{equation}
\label{rough-symmetry}
 u_x(t,x)>0\quad \mbox{for all }x\in [g(t), -h_0],\ t>0.
\end{equation}
\end{lem}
\begin{proof}
It is easily seen by the continuity that, when $t>0$ is sufficiently small,
$g(t)+h(t)>-2h_0$ and $u_x(t,x)>0$ for $g(t)\leqslant  x\leqslant  -h_0$. Define
$$
T_1:=\sup\big\{s\mid g(t)+h(t)>-2h_0\ \mbox{for all}\ t\in(0,s)\big\},
$$
$$
T_2:=\sup\big\{s\mid u_x(t,x)>0\ \mbox{for any}\ x\in[g(t),-h_0],\ t\in(0,s)\big\}.
$$
We prove that $T_1=T_2=+\infty$. Otherwise, either $T_1<T_2\leqslant  +\infty$, or
$T_2\leqslant  T_1\leqslant  +\infty$ and $T_2<+\infty$.

1. If $T_1<T_2\leqslant +\infty$, then
$$
g(t)+h(t)>-2h_0 \mbox{ for } t\in[0,T_1)\quad \mbox{and}\quad  g(T_1)+h(T_1)=-2h_0.
$$
Hence
\begin{equation}\label{T1}
   g'(T_1)+h'(T_1) \leqslant   0.
\end{equation}
Set $G_{T_1}:=\{(t,x)|t\in(0,T_1],x\in(g(t),-h_0)\}$ and
$$
w(t,x):=u(t,x)-u(t,-2h_0-x) \mbox{ in } \overline{G}_{T_1}.
$$
Since $-h_0\leqslant -2h_0-x\leqslant -2h_0-g(t)\leqslant  h(t)$ when $(t,x)\in\overline{G}_{T_1}$,
$w$ is well-defined over $\overline{G}_{T_1}$ and it satisfies
\begin{eqnarray*}
  w_t-w_{xx}-\beta w_x -c(t,x)w=-2\beta u_x(t,x) \leqslant 0 \mbox{ for } (t,x)\in G_{T_1},
\end{eqnarray*}
where $c$ is a bounded function, and
$$w(t,-h_0)=0,\ w(t,g(t))\leqslant0\ \mbox{for}\ t\in(0,T_1].$$
Moreover,
$$
w(T_1,g(T_1))=u(T_1,g(T_1))-u(T_1,-2h_0-g(T_1))=u(T_1,g(T_1))-u(T_1,h(T_1))=0.
$$
Then by the strong maximum principle and the Hopf lemma, we have
$$w(t,x)<0\ \mbox{for}\ (t,x)\in G_{T_1},\ \mbox{and}\ w_x(T_1,g(T_1))<0.$$
Thus
$$
g'(T_1)+h'(T_1)=-\mu[u_x(T_1,g(T_1))+u_x(T_1,h(T_1))]=-\mu w_x(T_1,g(T_1))>0.
$$
This contradicts \eqref{T1}.

2. If $T_2\leqslant  T_1\leqslant \infty$ and $T_2<+\infty$, then
$$
u_x(t,x)>0,\ t\in(0,T_2),\ x\in[g(t),-h_0].
$$
By the definition of $T_2$, there exists $y\in (g(T_2),-h_0]$ such that $u_x(T_2,y)= 0$.
Denote $x_0$ the minimum of such $y$. By the continuity and the monotonicity of $g(t)$,
there exists $T_0\in[0,T_2)$ such that
$x_0=g(T_0)$. Let $$G_{T_2}:=\{(t,x)|t\in(T_0,T_2],\ x\in(g(t),x_0)\},$$
$$
z(t,x):=u(t,x)-u(t,2x_0-x)\ \mbox{for}\ (t,x)\in\overline{G}_{T_2}.
$$
Using the maximum principle for $z(t,x)$ in $G_{T_2}$ as above we conclude that
$z_x (T_2, x_0)>0$. This contradicts the definition of $x_0$.

Combining the above two steps we obtain $T_1=T_2=+\infty$.
\end{proof}

A direct consequence of Lemma \ref{lem:center} is $g_{\infty}+h_{\infty}\geqslant -2h_0$.
So we have

\begin{cor}\label{3case}
There are only three possible situations for $I_\infty = (g_\infty, h_\infty)$:
{\rm (i)} $I_\infty =\R$; {\rm (ii)} $I_\infty$ is a finite interval;
and {\rm (iii)} $I_\infty =(g_\infty, \infty)$ with $g_\infty>-\infty$.
\end{cor}

\noindent
Indeed, (i) and (ii) are possible when $\beta\in (0,c_0)$ (see Theorem  \ref{thm:small beta}),
(ii) and (iii) are possible when $\beta \geqslant  c_0$ (see Theorems \ref{thm:middle beta}, \ref{thm:beta=c0} and
\ref{thm:large beta}).

By the monotonicity of $u(t,\cdot)$ in $[g(t),-h_0]$, we can prove the boundedness
of $g_\infty$.

\begin{prop}\label{prop:g_infty>-infty}
Assume $\beta\geqslant  c_0$ and $(u,g,h)$ is a solution of \eqref{p}. Then
$g_\infty>-\infty$.
\end{prop}
\begin{proof}
First we consider the case $\beta>c_0$. Let $f_A(s)$ be defined as in \eqref{fA}
and let $Q_A(z)$ be the unique solution of \eqref{c0} with $c=c_0$, with $f$
replaced by $f_A$ and $q(+\infty)=1$ replaced by $q(+\infty)=A$, as in \eqref{left bound beta large}. Since
\begin{equation}\label{QA(-infty)}
  Q_A(z)\sim-Cze^{\frac{c_0}{2}z}\mbox{ as }z\to-\infty
\end{equation}
for some $C>0$ (cf.  \cite{AW,HNRR}), there exist $T_1>0$, $C_1>0$ such that,
when $t\geqslant  T_1$,
\begin{eqnarray*}
  Q_A(-h_0-(\beta-c_0)t+x_0)&\leqslant & -2C(-h_0-(\beta-c_0)t+x_0)e^{\frac{c_0}{2}(-h_0-(\beta-c_0)t+x_0)} \\
  &\leqslant &C_1 te^{-\frac{c_0}{2}(\beta-c_0)t}\leqslant  C_1 e^{-\frac{c_0}{4}(\beta-c_0)t},
\end{eqnarray*}
where $x_0>0$ is large such that \eqref{left bound beta large} holds. By Lemma
\ref{lem:center} and \eqref{left bound beta large} we have
\begin{equation}\label{u(t,x)<=C1e^{-t/4}}
  u(t,x)\leqslant  u(t,-h_0)\leqslant  Q_A(-h_0-(\beta-c_0)t+x_0)\leqslant  C_1e^{-\frac{c_0}{4}(\beta-c_0)t}
\end{equation}
for $x\in[g(t),-h_0]$ and $t\geqslant  T_1$. Set $\delta:=\min\{1,\frac{c_0}{4}(\beta-c_0)\}$,
$\epsilon_1:=C_1e^{-\frac{\beta\pi+c_0(\beta-c_0)T_1}{4}}$,
$$k(t):=-g(T_1)+\frac{\pi}{2}+\frac{\mu\epsilon_1}{\delta}(1-e^{-\delta t})\mbox{ for }t\geqslant 0$$
and
$$w(t,x):=\epsilon_1 e^{-\delta t}e^{\frac{\beta}{2}(x+k(t))}\sin(x+k(t))\quad
\mbox{ for }-k(t)\leqslant  x\leqslant -k(t)+\frac{\pi}{2},\ t\geqslant 0.$$
A direct calculation shows that
\begin{eqnarray*}
  w_t-w_{xx}+\beta w_x-f(w)&=&w\Big[1-\delta+\frac{\beta}{2}k'(t)+\frac{\beta^2}{4}\Big]+k'(t)\hat{w}-f(w) \\
  &\geqslant& w\Big[1-\delta+\frac{\beta^2}{4}-f'(0)\Big]\geqslant(1-\delta)w\geqslant0,
\end{eqnarray*}
for $-k(t)\leqslant x\leqslant -k(t)+\frac{\pi}{2},\ t>0$, where
$\hat{w}=\epsilon_1 e^{-\delta t}e^{\frac{\beta}{2}(x+k(t))}\cos(x+k(t))$,
$$-k'(t)=-\epsilon_1\mu e^{-\delta t}=-\mu w_x(t,-k(t)),\quad t>0,$$
and
$$
w\Big(t,-k(t)+\frac{\pi}{2}\Big)=\epsilon_1 e^{\frac{\beta\pi}{4}}e^{-\delta t}
\geqslant  C_1 e^{-\frac{c_0(\beta-c_0)}{4}(t+T_1)}\geqslant  u(t+T_1,x)
$$
for $g(t+T_1)\leqslant  x\leqslant -h_0$, $t>0$. Hence for $t\geqslant 0$, either
$g(t+T_1)\geqslant -k(t)+\frac{\pi}{2}\geqslant  g(T_1)-\frac{\mu\epsilon_1}{\delta}$,
or $u(t+T_1,\cdot)$ and $w(t,\cdot)$ have common domain. In the latter case, by comparing them on their
common domain we have
$$g(t+T_1)\geqslant -k(t)\geqslant  g(T_1)-\frac{\pi}{2}-\frac{\mu\epsilon_1}{\delta}>-\infty.$$
This proves $g_\infty>-\infty$.

Next we consider the case $\beta=c_0$. By Corollary \ref{cor:-h0 h(t)} and Lemma
\ref{lem:center}, there exist $T_2>\frac{5}{4}$ and $C>0$ such that
$$u(t,x)\leqslant  u(t,-h_0)\leqslant  Ct^{-\frac{5}{4}}\mbox{ for }g(t)\leqslant  x\leqslant -h_0,\ t\geqslant  T_2.$$
Set $\epsilon_2:=Ce^{-\frac{\beta\pi}{4}}$ and define
$$k_2(t):=-g(T_2)+\frac{\pi}{2}+4\mu\epsilon_2[T_2^{-\frac{1}{4}}-(t+T_2)^{-\frac{1}{4}}]\mbox{ for }t>0,$$
$$w_2(t,x):=\epsilon_2(t+T_2)^{-\frac{5}{4}}e^{\frac{\beta}{2}(x+k_2(t))}\sin(x+k_2(t))
\mbox{ for }-k_2(t)\leqslant  x\leqslant -k_2(t)+\frac{\pi}{2},\ t\geqslant 0.$$
A similar discussion as above shows that $(w_2,-k_2,-k_2+\frac{\pi}{2})$ is an upper
solution, and so
$$g(t+T_2)\geqslant -k_2(t)\geqslant  g(T_2)-\frac{\pi}{2}-4\mu\epsilon_2 T_2^{-\frac{1}{4}}>-\infty.$$
This proves the proposition.
\end{proof}

\medskip
Next we prove the boundedness of $h_\infty$ when $\beta \geqslant  \beta^*$.

\begin{prop}\label{>=b^*-h_infty}
Assume $\beta\geqslant \beta^*$ and $(u,g,h)$ is a solution of \eqref{p}. Then
$h_\infty<\infty$.
\end{prop}
\begin{proof}
1. First we consider the case $\beta>\beta^*$. In this case we have
$c^*(\beta)<\beta-c_0$. Denote $\nu:=\beta-c_0-c^*(\beta)>0$. By \eqref{left bound beta large},
\eqref{right bound} and \eqref{QA(-infty)}, there exist $T_1>0$, $C_1>0$ such that,
for $x\in(g(t),h(t))$ and $t\geqslant  T_1$, we have
\begin{eqnarray*}
  u(t,x)&\leqslant & Q_A(x-(\beta-c_0)t+x_0)\leqslant  Q_A(h(t)-(\beta-c_0)t+x_0) \\
  &\leqslant & Q_A(-\nu t+H+x_0)\leqslant -2C(-\nu t+ H +x_0)e^{\frac{c_0}{2}(-\nu t+ H +x_0)} \\
  &\leqslant & C_1e^{-\frac{c_0\nu}{4}t}.
\end{eqnarray*}
Set $\delta:=\frac{1}{2}\min\{1,\frac{c_0\nu}{4}\}$ and choose $T_2>T_1$ such that
$$
\epsilon_3 :=C_1 e^{\frac{\beta\pi-c_0\nu T_2}{4}} < \frac{2}{\beta \mu}.
$$
Define
$$
k_3(t):=h(T_2)+\frac{\pi}{2}+\frac{\mu\epsilon_3}{\delta}(1-e^{-\delta t})\mbox{ for } t\geqslant 0
$$
and
$$
w_3(t,x):=\epsilon_3 e^{-\delta t}e^{\frac{\beta}{2}(x-k_3(t))}\cos \Big( x-k_3(t)+\frac{\pi}{2} \Big)
\mbox{ for }k_3(t)-\frac{\pi}{2}\leqslant  x\leqslant  k_3(t),\ t\geqslant 0.
$$
A direct calculation as in the proof of Proposition \ref{prop:g_infty>-infty} shows
that $(w_3,k_3(t)-\frac{\pi}{2},k_3(t))$ is an upper solution and
$$h(t+T_2)\leqslant  k_3(t)\leqslant  h(T_2)+\frac{\pi}{2}+\frac{\mu\epsilon_3}{\delta}<\infty.$$

\medskip

2. Next we consider the case $\beta=\beta^*$. We first show that, for some
large $T_3$,
\begin{equation}\label{u(t,x)<U*(x-c*t+h0)}
  u(t,x)<U^*(x-c^*(\beta^*)t+h_0)\mbox{ for }x\in[g(t),h(t)],\ t\geqslant  T_3,
\end{equation}
where $U^*(x-c^*(\beta^*)t+h_0)$ is the rightward traveling semi-wave with endpoint at $c^*(\beta^*) t-h_0$.
At time $t=0$, $u(t,x)$ and $U^*(x+h_0)$ intersect at $x=-h_0$. Then for small time
$t>0$, they intersect at exact one point.

We claim that the case $c^*(\beta^*) t-h_0<h(t)$ for
all $t\geqslant 0$ is impossible. Otherwise, combining with \eqref{right bound} we have
$h(t)=c^* (\beta^*) t+O(1)$, and so by Corollary \ref{cor:-h0 h(t)} there exist $T_4>0$ and $C>0$
such that
$$
u(t,x)\leqslant  Ct^{-\frac{5}{4}}\mbox{ for }h(t)-\frac{\pi}{2}\leqslant  x\leqslant  h(t),\ t\geqslant  T_4.
$$
Set $\epsilon_4:=Ce^{\frac{\beta\pi}{4}}$, $T_5:=\max\{1,T_4, \frac{5}{4}+\frac{\beta\mu\epsilon_4}{2}\}$
and define
$$
k_4(t):=h(T_5)+\frac{\pi}{2}+4\mu\epsilon_4[T_5^{-\frac{1}{4}}-(t+T_5)^{-\frac{1}{4}}],\quad t\geqslant 0,
$$
$$
w_4(t,x):=\epsilon_4(t+T_5)^{-\frac{5}{4}}e^{\frac{\beta}{2}(x-k_4(t))}\cos\Big(x-k_4(t)+\frac{\pi}{2}\Big),
\quad k_4(t)-\frac{\pi}{2}\leqslant  x\leqslant  k_4(t),\ t\geqslant 0.
$$
A direct calculation shows that $(w_4,k_4(t)-\frac{\pi}{2},k_4(t))$ is an upper solution,
and so
$$h(t+T_5)\leqslant  k_4(t)\leqslant  h(T_5)+\frac{\pi}{2}+4\mu\epsilon_4 T_5^{-\frac{1}{4}}<\infty,$$
contradicts our assumption $c^*(\beta^*) t-h_0<h(t)$ for all $t$.

Therefore, there exists $T_6>0$ such that $h(T_6)=c^*(\beta^*) T_6-h_0$ and by Lemma
\ref{lem:zeros between u and Psi}, the unique intersection point between $u$ and
$U^*$ disappears after $T_6$. This implies \eqref{u(t,x)<U*(x-c*t+h0)} holds when
$x\in[g(t),h(t)]$ and $t>T_3$ for any $T_3 >T_6$.

On the other hand, by Lemma \ref{lem:tadpole tw beta=beta*},
$V^*_{\delta_1}(z):=V(z;\beta^*-c_0-\delta_1,-c_0-\delta_1)$ approaches $U^*(z)$ locally
uniformly in $(-\infty,0]$ as $\delta_1\to0$. Hence there exists $\delta_1>0$ sufficiently
small such that $V^*_{\delta_1} (z)$ is close to $U^*(z)$ and so
$$
u(T_3,x)<V^*_{\delta_1}(x-c^*(\beta^*)T_3+h_0)\mbox{ for }x\in[g(t),h(t)].
$$
By comparison $u(t+T_3,x)\leqslant V^*_{\delta_1}(x-(\beta^*-c_0-\delta_1)t-c^*(\beta^*)T_3+h_0)$
and so $h(t+T_3)$ is blocked by the right endpoint $(\beta^* -c_0-\delta_1)t+c^*(\beta^*)T_3-h_0$ of
$V^*_{\delta_1}$:
$$
h(t+T_3)\leqslant (\beta^*-c_0-\delta_1)t+c^*(\beta^*)T_3-h_0,\quad t\geqslant0.
$$
Using \eqref{left bound beta large} we see that, for any $x\in[g(t),h(t)]$ and sufficiently
large $t$,
$$u(t,x)\leqslant  Q_A(x-(\beta^*-c_0)t+x_0)\leqslant  Q_A(-\delta_1t+x_1)\leqslant  C_2 e^{-\frac{c_0\delta_1}{4}t},$$
for some $x_1\in\R$ and $C_2>0$. The rest proof is similar as that in Step 1.

This proves the proposition.
\end{proof}

\noindent
 {\bf Proof of Theorem \ref{thm:large beta}:}\ The conclusions follow from Proposition
 \ref{prop:g_infty>-infty}, \ref{>=b^*-h_infty} and Theorem \ref{thm:convergence} immediately.
{\hfill $\Box$}

\subsection{Problem with medium-sized advection: $c_0\leqslant \beta<\beta^*$}
In this subsection we consider the case $\beta\in [c_0,\beta^*)$. In this case,
the long time behavior of the solutions is complicated and more interesting.
Besides vanishing, we find some new phnomena: virtual spreading, virtual
vanishing and convergence to the tadpole-like traveling wave.

In the first part, we give some sufficient conditions for vanishing; in the second part
we give a necessary and sufficient condition for virtual spreading; in the third part we
study the limits of $h'(t)$ and $u(t,x)$ when vanishing and virtual spreading do not happen; in the last
part we finish the proofs of Theorems \ref{thm:middle beta} and \ref{thm:beta=c0}.

\subsubsection{Vanishing phenomena}
When $\beta\geqslant  c_0$, we have $g_\infty > -\infty$ by Proposition \ref{prop:g_infty>-infty},
which implies that $u\to 0$ locally uniformly. We now show that the convergence can be
a uniform one when the initial data $u_0$ is sufficiently small.

\begin{lem}\label{lem:condition for vanishing}
Assume $\beta\geqslant  c_0$ and $(u,g,h)$ is the solution of \eqref{p}. If
$\|u_0\|_{L^{\infty}([-h_0,h_0])}$ is sufficiently small, then vanishing happens.
\end{lem}

\begin{proof}
Choose $\delta>0$ small such that
$$
\frac{\pi^2}{h_0^2(1+\delta)^2}\geqslant 4\delta+\beta h_0\delta^2.
$$
Set $k(t):=h_0(1+\delta-\frac{\delta}{2}e^{-\delta t})$,  $\epsilon:= \frac{h_0^2\delta^2}{\pi\mu}(1+\frac{\delta}{2})$
and
\begin{equation}\label{vanishing-w}
  w(t,x):=\epsilon e^{-\delta t}e^{\frac{\beta}{2}(x-k(t))}\cos \frac{\pi x}{2k(t)}\ \ \mbox{ for }-k(t)\leqslant  x\leqslant  k(t),\ t>0.
\end{equation}
A direct calculation shows that, for $x\in(-k(t),k(t))$ and $t>0$,
$$
w_t-w_{xx}+\beta w_x-f(w)\geqslant \frac{1}{4}\Big( \frac{\pi^2}{h_0^2(1+\delta)^2}-
4\delta-\beta h_0\delta^2 \Big) w\geqslant 0.
$$
On the other hand, for any $t>0$,
$$
\mu w_x(t,-k(t))\leqslant -\mu w_x(t,k(t))=\frac{\pi\mu \epsilon}{2k(t)}e^{-\delta t}
\leqslant \frac{\pi\mu \epsilon}{2h_0(1+\frac{\delta}{2})}e^{-\delta t} =\frac{\delta^2}{2}h_0e^{-\delta t}=k'(t).
$$
Hence $(w,-k,k)$ is an upper solution of \eqref{p}.
Clearly $k(0)=h_0(1+\frac{\delta}{2})>h_0$ and $w(t, \pm k(t))=0$ for $t\geqslant 0$.
If $\|u_0\|_{L^\infty([-h_0,h_0])}$ is small such that
$$
\|u_0\|_{L^\infty([-h_0,h_0])}\leqslant  \epsilon e^{-\frac{\beta}{2}h_0(2+\frac{\delta}{2})}\cos \frac{\pi}{2+\delta}=w(0,-h_0),
$$
then $u_0(x)\leqslant  w(0,x)$ for $x\in[-h_0,h_0]$. By the comparison principle, we have
$$
-h_0(1+\delta)\leqslant  -k(t)\leqslant  g(t)<h(t)\leqslant  k(t)\leqslant  h_0(1+\delta),
$$
$$
\|u(t,\cdot)\|_{L^\infty([g(t),h(t)])} \leqslant  \|w(t,\cdot)\|_{L^\infty([-k(t),k(t)])} \leqslant  \epsilon e^{-\delta t}
\to 0\quad \mbox{as } t\to \infty.
$$
This proves the lemma.
\end{proof}

For any given $h_0>0$ and $\phi\in\mathscr {X}(h_0)$, we write the solution $(u,g,h)$ also
as $(u(t,x;\sigma\phi)$, $g(t;\sigma\phi)$,$h(t;\sigma\phi))$ to emphasize the dependence on the
initial data $u_0=\sigma\phi$. Set
\begin{equation}\label{def E0}
  E_0:=\{\sigma\geqslant 0 \mid \mbox{vanishing happens for } u(\cdot,\cdot;\sigma\phi)\},\quad \sigma_*:=\sup E_0.
\end{equation}
Lemma \ref{lem:condition for vanishing} implies that $\sigma\in E_0$ for all small $\sigma>0$.
By the comparison principle we have $[0,\sigma_*)\subset E_0$.
In case $\sigma_*=\infty$ (this happens in particular when $\liminf_{s\to \infty} \frac{-f(s)}{s} \gg 1$
and $\beta =0$, see \cite[Proposition 5.4]{DuLou}), there is nothing left to prove.
Hence we only consider the case $\sigma_*\in(0,\infty)$.

\begin{thm}\label{thm:sigma_*}
Assume $c_0\leqslant \beta<\beta^*$. For any $\phi\in\mathscr {X}(h_0)$, let $E_0$ and $\sigma_*$
be defined as in \eqref{def E0}. If $\sigma_*\in(0,\infty)$, then $E_0 = [0,\sigma_*)$. If
$\sigma\geqslant  \sigma_*$, then $g(\infty;\sigma\phi)>-\infty$, $h(\infty;\sigma\phi)=\infty$,
and $u(t,\cdot;\sigma\phi)\to 0$ locally uniformly in $(g(\infty;\sigma\phi), \infty)$.
\end{thm}
\begin{proof}
For any positive $\sigma_0\in E_0$, since $u(t,\cdot;\sigma_0\phi)\to 0$ uniformly,
we can find a large $T_0>0$ such that $u(T_0,x;\sigma_0\phi)< w(0,x)$, where $w(t,x)$
is defined as in \eqref{vanishing-w},
with $h_0$ replaced by $H:= \max\{h(\infty;\sigma_0\phi), -g(\infty;\sigma_0\phi)\}<\infty$.
By continuity, there exists $\epsilon>0$ such that $u(T_0,x;\sigma\phi)< w(0,x)$ for every
$\sigma\in[\sigma_0,\sigma_0+\epsilon)$. As in the proof of Lemma \ref{lem:condition for vanishing},
we conclude that vanishing happens for $u(t,x;\sigma\phi)$, that is, $\sigma \in E_0$.
Therefore, $E_0\backslash \{0\}$ is an open set, and so $E_0=[0,\sigma_*)$.
The rest of the conclusions follow from Theorem \ref{thm:convergence} and Proposition
\ref{prop:g_infty>-infty}.
\end{proof}

We finish this part by giving another sufficient condition for vanishing.

\begin{lem}\label{lem:condition 2 for vanishing}
Assume $c_0 < \beta<\beta^*$. Let $(u,g,h)$ be a solution of \eqref{p}. Then
vanishing happens if there exist $t_1\geqslant 0$, $x_1\in\R$ such that
$$h(t_1)\leqslant  x_1,\ u(t_1,x)\leqslant  V^*(x-x_1)\mbox{ for }x\in[g(t_1),h(t_1)],$$
where $V^*(z)$ is the tadpole-like solution as in Lemma \ref{lem:tadpole tw bata<beta*} (ii).
\end{lem}

\begin{proof}
Since $V^*(x-(\beta-c_0)t-x_1)$ is a solution of \eqref{p}$_1$ satisfying Stefan
free boundary condition at $x=k(t):=(\beta-c_0)t+x_1$:
$$k'(t)=\beta-c_0=-\mu(V^*)'(0)$$
by Lemma \ref{lem:tadpole tw bata<beta*} (ii). By the comparison principle we have
$$h(t_1+1)<k(1),\ u(t_1+1,x)<V^*(x-\beta+c_0-x_1)\mbox{ for }x\in[g(t_1+1),h(t_1+1)].$$
By Lemma \ref{lem:tadpole tw bata<beta*} (iii),
$V_\delta(z):=V(z;\beta-c_0-\delta,-c_0-\delta)\to V^*(z)$
locally uniformly in $(-\infty,0]$ as $\delta\to0$. Hence for sufficiently small
$\delta>0$, we have
$$u(t_1+1,x)<V_\delta(x-\beta+c_0-x_1)\mbox{ for }x\in[g(t_1+1),h(t_1+1)].$$
Since $V_\delta(x-(\beta-c_0-\delta)t-\beta+c_0-x_1)$ is a solution of \eqref{p}$_1$
satisfying Stefan free boundary condition at $x=k_\delta(t):=(\beta-c_0-\delta)t+\beta-c_0+x_1$,
by comparison we have
$$h(t+t_1+1)\leqslant  k_\delta(t)=(\beta-c_0-\delta)t+\beta-c_0+x_1.$$
A similar argument as in step 1 of the proof of Proposition \ref{>=b^*-h_infty} shows
that $h_\infty<\infty$, this implies that vanishing happens by Theorem \ref{thm:convergence}.
\end{proof}

\subsubsection{A necessary and sufficient condition for virtual spreading}

\begin{lem}\label{lem:condition for virtual spreading}
Assume $c_0\leqslant\beta<\beta^*$. Let $(u,g,h)$ be a solution of \eqref{p}. Then
virtual spreading happens if and only if, for any $\delta\in(0,c^*(\beta)-\beta+c_0)$,
there exist $t_1$ and $x_1$ such that
\begin{equation}\label{condition for v s}
u(t_1,x)\geqslant  W_\delta(x-x_1)\mbox{ for }x\in[x_1-L_\delta,x_1],
\end{equation}
where $W_\delta(z)$, $L_\delta:=L(\beta-c_0+\delta,-c_0+\delta)$ are the notation in Lemma \ref{lem:compact tw}.
\end{lem}

\begin{proof}
The inequality \eqref{condition for v s} follows from the definition of virtual spreading
immediately. We only need to show that \eqref{condition for v s} is a sufficient condition for
virtual spreading.

Since $W_\delta(x-(\beta-c_0+\delta)t-x_1)$ satisfies \eqref{p}$_1$ and Stefan free
boundary condition at $x=r(t):=(\beta-c_0+\delta)t+x_1$. Comparing $u$ and $W_\delta$ we have
\begin{equation}\label{u(t+t2,x)>W_delta}
  u(t+t_1,x)>W_\delta(x-(\beta-c_0+\delta)t-x_1)\mbox{ for }x\in[r(t)-L_\delta,r(t)].
\end{equation}
In particular, this is true at $t=1$. Since $W_\delta(z)$ depends on $\delta$ continuously,
we have
$$u(t_1+1,x)>W_{\delta+\epsilon}(x-(\beta-c_0+\delta+\epsilon)-x_1)
\mbox{ for }x\in[r(1)+\epsilon-L_{\delta+\epsilon},r(1)+\epsilon],$$
for any $\epsilon\in(0,\epsilon_0]$ provided $\epsilon_0>0$ is small. Using the comparison principle
again between $u(t+t_1+1,x)$ and $W_{\delta+\epsilon}(x-(\beta-c_0+\delta+\epsilon)(t+1)-x_1)$,
we have
$$h(t+t_1+1)\geqslant (\beta-c_0+\delta+\epsilon)(t+1)+x_1,\ t>0.$$
This implies that
\begin{equation}\label{H}
  H(t):=h(t+t_1+1)-(\beta-c_0+\delta)t\geqslant \epsilon(t+1)+x_1\to\infty\mbox{ as }t\to\infty.
\end{equation}
Set
\begin{equation}\label{G}
  G(t):=g(t+t_1+1)-(\beta-c_0+\delta)t
\end{equation}
and
\begin{equation}\label{w-G-H}
  w(t,x):=u(t+t_1+1,x+(\beta-c_0+\delta)t)\mbox{ for }G(t)\leqslant  x\leqslant  H(t),\ t\geqslant 0.
\end{equation}
Then $G(t)\to-\infty$ as $t\to\infty$ by Proposition \ref{prop:g_infty>-infty}, $w$
satisfies
\begin{equation}\label{w(t,x)>W_delta}
  w(t,x)>W_\delta(x-(\beta-c_0+\delta)-x_1)\mbox{ for }x\in[x_1+\beta-c_0+\delta-L_\delta,x_1+\beta-c_0+\delta],\ t\geqslant0
\end{equation}
by \eqref{u(t+t2,x)>W_delta} and
\begin{equation}\label{shift-w}
\left\{
\begin{array}{ll}
 w_t=w_{xx}-(c_0-\delta)w_{x}+f(w),\quad t>0,\ G(t)< x<H(t),\\
 w(t,G(t))=0,\ \ G'(t)=-\mu w_x(t, G(t))-(\beta-c_0+\delta),\quad t>0,\\
 w(t,H(t))=0,\ \ H'(t)=-\mu w_x (t, H(t))-(\beta-c_0+\delta) ,\quad t>0,\\
G(0)=g(t_1+1),\ H(0)= h(t_1+1),\ \ w(0,x) =u(t_1+1,x) \mbox{ for } G(0)\leqslant  x \leqslant  H(0).
\end{array}
\right.
\end{equation}
In a similar way as proving Theorem \ref{thm:convergence} (cf. the proof of
\cite[Theorem 1.1]{DuLou}), one can show that $w(t,\cdot)$ converges to a stationary
solution of \eqref{shift-w}$_1$ locally uniformly in $\R$. By \eqref{w(t,x)>W_delta}, such a stationary
solution must be 1. This means spreading happens for
$w$ and so virtual spreading happens for $u$.
This proves the lemma.
\end{proof}

\subsubsection{The limits of $h$ and $u$ when vanishing and virtual spreading do not happen}
In this part we always assume $c_0\leqslant  \beta<\beta^*$ and $\sigma_*\in (0,\infty)$
for given $\phi\in \mathscr{X}(h_0)$, where $\sigma_*$ is defined by \eqref{def E0}. We consider the limits of $h(t;\sigma\phi)$,
$h'(t;\sigma\phi)$ and $u(t,\cdot+h(t);\sigma\phi)$ when vanishing does not happen, that is, when
$\sigma\geqslant\sigma_*$.

\begin{lem}\label{lem:h(t)-(beta-c0)t=+infty}
Assume $c_0\leqslant \beta<\beta^*$. If vanishing does not happen for a solution $u$ of
\eqref{p}, then
\begin{equation}\label{h(t)-(beta-c0)t=+infty}
\lim\limits_{t\to\infty}[h(t)-(\beta-c_0)t]=+\infty.
\end{equation}
\end{lem}
\begin{proof}
When $\beta=c_0$, we have $g_\infty>-\infty$ by Proposition \ref{prop:g_infty>-infty}.
If $h_\infty<\infty$, then vanishing happens for $u$ by Theorem \ref{thm:convergence},
contradicts our assumption. Therefore, \eqref{h(t)-(beta-c0)t=+infty} holds when $\beta=c_0$.

We now consider the case $c_0<\beta<\beta^*$. First we prove that
$$h(t)>(\beta-c_0)t-h_0\mbox{ for any }t>0.$$
Set
$$\eta_1(t,x):=u(t,x)-V^*(x-(\beta-c_0)t+h_0)\mbox{ for }x\in J_1 (t),\ t>0,$$
where
$$
J_1(t):=[g(t),\min\{h(t),(\beta-c_0)t-h_0\}]\mbox{ for }t>0.
$$
It is easily seen that, for $0<t\ll1$,
\begin{equation}\label{Z_J(t)[eta1]}
  (\beta-c_0)t-h_0<h(t)\mbox{ and }\mathcal{Z}_{J_1(t)}[\eta_1(t,\cdot)]=1.
\end{equation}
We claim that this is true for all $t>0$. Otherwise, there exists $T_1>0$ such that
$(\beta-c_0)t-h_0<h(t)$ for $0<t<T_1$ and $(\beta-c_0)T_1-h_0=h(T_1)$. By Lemma
\ref{lem:zeros between u and Psi} we have $\mathcal{Z}_{J_1(t)}[\eta_1(t,\cdot)]=1$ for
$0<t<T_1$ and $\mathcal{Z}_{J_1(t)}[\eta_1(t,\cdot)]=0$ for $t>T_1$. Therefore,
$$u(t,x)<V^*(x-(\beta-c_0)t+h_0)\mbox{ for }x\in I(t),\ T_1<t\ll T_1+1.$$
This implies that vanishing happens for $u$ by Lemma \ref{lem:condition 2 for vanishing},
contradicts our assumption.

Next we prove that, for any large $M>0$, $h(t)>(\beta-c_0)t+M$ when $t$ is large.
Without loss of generality we assume
$$u_0'(-h_0)>0,\ u_0'(h_0)<0\mbox{ and }u_0(x)>0\mbox{ for }x\in(-h_0,h_0).$$
(Otherwise one can replace $u_0(x)$ by $u(1,x)$ to proceed the following analysis.)
So there exists $X>h_0$ large such that $u_0(x)$ intersects $V^*(x-M)$ at exactly
two points for any $M\geqslant X$. Set
$$\eta_2(t,x):=u(t,x)-V^*(x-(\beta-c_0)t-M)\mbox{ for }x\in J_2(t),\ t>0,$$
where
$$J_2(t):=[g(t),\min\{h(t),(\beta-c_0)t+M\}].$$
Then $\mathcal{Z}_{J_2(t)}[\eta_2(t,\cdot)]=2$ for $0<t\ll1$. Denote by $\xi_1(t)$
and $\xi_2(t)$ with $\xi_1(t)<\xi_2(t)$ the two zeros of $\eta_2(t,\cdot)$. Then we
have the following situations about the relations among $\xi_1(t)$, $\xi_2(t)$, $h(t)$
and $(\beta -c_0)t +M $.

\medskip

\emph{Case 1}. $h(t)<(\beta-c_0)t+M$ for all $t>0$. In this case, combining with
\eqref{Z_J(t)[eta1]} we have $h(t)=(\beta-c_0)t+O(1)$. Using a similar argument as
in step 2 of the proof of Proposition \ref{>=b^*-h_infty} we can derive $h_\infty<\infty$.
This implies that vanishing happens for $u$, contradicts our assumption.

\medskip

\emph{Case 2}. There exists $T_2>0$ such that $h(t)<(\beta-c_0)t+M$ for $0<t<T_2$
and $h(T_2)=(\beta-c_0)T_2+M$. This includes several subcases.

\medskip

\emph{Subcase 2-1}. $\xi_1(t)$ meets $\xi_2(t)$ at time $t=T_3<T_2$. In this case,
$\xi_1(T_3)=\xi_2(T_3)$ is a degenerate zero of $\eta_2(T_3,\cdot)$ and so
$\mathcal{Z}_{I(t)}[\eta_2(t,\cdot)]=0$ for $T_3<t\ll T_3+1$. This indicates that
\begin{equation}\label{u(t+1,x)<V*}
  u(t,x)<V^*(x-(\beta-c_0)t-M),\quad x\in[g(t),h(t)],\ T_3<t\ll T_3+1,
\end{equation}
and so vanishing happens by Lemma \ref{lem:condition 2 for vanishing}, contradicts
our assumption.

\medskip

\emph{Subcase 2-2}. $\xi_1(t)<\xi_2(t)<h(t)$ for $0<t\leqslant T_2$. This means a new
intersection point $(h(T_2),0)$ between $u$ and $V^*$ emerges on the boundary. This is
impossible by Lemma \ref{lem:zeros between u and Psi}.

\medskip

\emph{Subcase 2-3}. $\xi_1(t)<\xi_2(t)<h(t)$ for $0<t<T_2$ and $\xi_1(T_2)=\xi_2(T_2)=h(T_2)$.
This means the two intersection points between $u$ and $V^*$ move rightward to $(h(t),0)$
at time $T_2$. By Lemma \ref{lem:zeros between u and Psi}, this is the unique zero of
$\eta_2(T_2,\cdot)$ and it will disappear after time $T_2$. Hence \eqref{u(t+1,x)<V*} holds for
$t>T_2$. Then vanishing happens, a contradiction.

\medskip

\emph{Subcase 2-4}. $\xi_1(t)<\xi_2(t)<h(t)$ for $0<t<T_2$ and
$\xi_1(T_2)<\xi_2(T_2)=h(T_2)=(\beta-c_0)T_2+M$. By Lemma \ref{lem:zeros between u and Psi},
$\mathcal{Z}_{J_2(t)}[\eta_2(t,\cdot)]=1<2$ for $T_2<t\ll T_2+1$, where
$J_2(t):=[g(t),(\beta-c_0)t+M]$. Using the maximum principle for $\eta_2(t,x)$ in the domain
$$\Omega:=\{(t,x)\mid\xi_1(t)<x<\xi_2(t),\ 0<t\leqslant T_2\}$$
and using Hopf lemma at $(t,x)=(T_2,h(T_2))=(T_2,\xi_2(T_2))$ we have $(\eta_2)_x(T_2,h(T_2))<0$,
that is,
\begin{equation}\label{u_x(T2,h(T2))<(V*)'(0)}
  u_x(T_2,h(T_2))<(V^*)'(0)=-\frac{\beta-c_0}{\mu},
\end{equation}
and so
\begin{equation}\label{h'(T2)>beta-c0}
 h'(T_2)=-\mu u_x(T_2,h(T_2))>\beta-c_0.
\end{equation}
We claim that
\begin{equation}\label{(beta-c0)t+M<h(t)}
  (\beta-c_0)t+M<h(t)\mbox{ for all }t>T_2
\end{equation}
and so $\mathcal{Z}_{J_2(t)}[\eta_2(t,\cdot)]=1$ for all $t>T_2$. Indeed, if
$(\beta-c_0)t+M$ catches up $h(t)$ again at time $t=T_4$, then the unique intersection
point $(\xi_1(t),u(t,\xi_1(t)))$ (for $t\in[T_2,T_4)$) moves to $(h(t),0)$ at time
$T_4$ and then it disappear after time $T_4$ by Lemma \ref{lem:zeros between u and Psi}.
This implies that \eqref{u(t+1,x)<V*} holds for $t>T_4$ and so vanishing happens,
a contradiction. \eqref{(beta-c0)t+M<h(t)} is true for any $M>0$ and so
\eqref{h(t)-(beta-c0)t=+infty} holds.
\end{proof}

\begin{lem}\label{lem:h'(t)=beta-c0}
Assume $c_0\leqslant \beta<\beta^*$. If vanishing and virtual spreading do not happen for
the solution $u$ of \eqref{p}, then
\begin{equation}\label{h'(t)=beta-c0}
 \lim\limits_{t\to\infty}h'(t)=\beta-c_0.
\end{equation}
\end{lem}
\begin{proof}
We divide the proof into several steps.

\medskip

\emph{Step 1}. We first prove $h'(t)>\beta-c_0$ for all large $t$. This is clear
when $\beta=c_0$. We now assume $c_0<\beta<\beta^*$.

For readers' convenience, we first sketch the idea of our proof. We put a tadpole-like traveling wave $V^*(x-(\beta -c_0)t-C)$
whose right endpoint $r(t):= (\beta -c_0)t+C$ lies right to $h_0$. As $t$ increasing,
both $h(t)$ and $r(t)$ move rightward, but $h(t)$ moves faster by Lemma \ref{lem:h(t)-(beta-c0)t=+infty}.
Hence $h(t)$ catches up $r(t)$ at some time $T$. We will show that at this moment $u>V^*$ near $x=h(T)$
and so $h'(T)\geqslant
\beta -c_0$ (in fact, strict inequality holds by Hopf lemma). Since the shift $C$ of $V^*$ can be chosen
continuously we indeed obtain $h'(t)>\beta -c_0$  for all large time $t$.

Now we give the details of the proof. As in the proof of the previous lemma,
there exists $X>h_0$ such that $u_0(x)$ intersects $V^*(x-M)$ at exactly two points for any
$M\geqslant X$.

By \eqref{h(t)-(beta-c0)t=+infty}, there exists $T_X>0$ such that $h(t)-(\beta-c_0)t>X$
for all $t\geqslant T_X$. For any $a>h(T_X)$ denote $T_a$ the unique time such that
$h(T_a)=a$. Set $X_a:=h(T_a)-(\beta-c_0)T_a\ (>X)$. We study the intersection points
between $u(t,\cdot)$ and $V^*(x-(\beta-c_0)t-X_a)$. As in the proof of the previous lemma,
only subcase 2-4 is possible: there exists $T^*>0$ such that
$$\xi_1(t)<\xi_2(t)<h(t)\mbox{ for }0<t<T^*\quad \mbox{and}\quad \xi_1(T^*)<\xi_2(T^*)=h(T^*)=(\beta-c_0)T^*+X_a,$$
and as proving \eqref{(beta-c0)t+M<h(t)} we have
$$
(\beta-c_0)t+X_a<h(t)\quad \mbox{for all }t>T^*.
$$
Therefore $T^*$ is nothing but $T_a$. By \eqref{h'(T2)>beta-c0} we have
$$h'(T_a)=-\mu u_x(T_a,h(T_a))=-\mu u_x(T_a,a)>\beta-c_0.$$
Since $a>h(T_X)$ is arbitrary, $T_a$ is continuous and strictly increasing in $a$,
we indeed have
$$h'(t)>\beta-c_0\quad \mbox{for all }t>T_X.$$

\medskip

\emph{Step 2}. We prove
\begin{equation}\label{h(t)-(beta-c_0+delta)t=-infty}
  \lim\limits_{t\to\infty}[h(t)-(\beta-c_0+\delta)t]=-\infty\mbox{ for all }\delta\in(0,c^*-\beta+c_0).
\end{equation}
For any $\delta\in(0,c^*-\beta+c_0)$, we choose $\delta_1\in(0,\delta)$
and consider the compactly supported traveling wave $W_{\delta_1}(x-c_1t-M)$,
where $c_1=\beta-c_0+\delta_1$, $M>0$ is a large real number such that $u_0(x)$
has no intersection point with $W_{\delta_1}(x-M)$. Clearly \eqref{h(t)-(beta-c_0+delta)t=-infty}
is proved if we have $h(t)<c_1t+M$ for all $t>0$. If, otherwise, there exists
some $T_1>0$ such that
$$h(t)<c_1t+M\mbox{ for }t\in[0,T_1),\ h(T_1)=c_1T_1+M,$$
then there exists $T_2\in(0,T_1)$ such that $h(t)$ catches up the left boundary
$l_1(t):=c_1t+M-L_{\delta_1}$ of the support of $W_\delta(x-c_1t-M)$ at time $T_2$
and never lags behind it again. So in the time interval $(T_2,T_1)$.
$$\mathcal{Z}_{J_1(t)}[\zeta_1(t,\cdot)]=1\mbox{ for }t\in[T_2,T_1],$$
where $J_1(t):=[l_1(t),h(t)]$ and
$$\zeta_1(t,x):=u(t,x)-W_{\delta_1}(x-c_1t-M)\mbox{ for } x\in J_1(t),\ t\in[T_2,T_1].$$
By Lemma \ref{lem:zeros between u and Psi}, the unique zero $\zeta_1(t,\cdot)$ moves to
$(h(t),0)$ at time $t=T_1$ and it disappears after $T_1$. Hence
$$u(T_1,x)\geqslant  W_{\delta_1}(x-c_1T_1-M)\mbox{ for }x\in[l_1(T_1),c_1T_1+M]=[l(T_1),h(T_1)].$$
This implies that virtual spreading happens for $u$ by Lemma \ref{lem:condition for virtual spreading},
contradicts our assumption.

\medskip

\emph{Step 3}. Based on Step 2 we prove
\begin{equation}\label{h'(t)<beta-c_0+delta}
  h'(t)<\beta-c_0+\delta\mbox{ for large }t,
\end{equation}
for any $\delta\in(0,c^*-\beta+c_0)$. Fix such a $\delta$, we
consider $u(t,x)$ and $W_\delta(x-(\beta-c_0+\delta)t+h_0)$. It is easily seen
that these two functions intersect at exactly one point in their common domain
$J_2(t):=[g(t),r(t)]$ for small $t>0$, where $r(t):=(\beta-c_0+\delta)t-h_0$. By
Step 2, there exists $T_3>0$ such that
$$r(t)<h(t)\mbox{ for }t\in[0,T_3),\ r(T_3)=h(T_3).$$
If the left boundary $l_2(t):=r(t)-L_\delta$ of the support of $W_\delta(x-r(t))$
lags behind $g(t)$ till $t=T_3$: $l_2(t)<g(t)$ for $t\in[0,T_3)$, then
$$u(T_3,x)\leqslant  W_\delta(x-r(T_3))\mbox{ for }x\in[g(T_3),h(T_3)].$$
Using Hopf lemma at $h(T_3)$ we have
\begin{equation}\label{h'(T3)<beta-c0+delta}
  h'(T_3)=-\mu u_x(T_3,h(T_3))<-\mu W_\delta'(0)=\beta-c_0+\delta.
\end{equation}
If there exists $T_4\in(0,T_3)$ such that
$$l_2(t)<g(t)\mbox{ for }t\in[0,T_4),\ l_2(T_4)=g(T_4).$$
Then either $W_\delta(x-r(T_4))\leqslant u(T_4,x)$ in $[l_2(T_4),r(T_4)]$ or
$\mathcal{Z}_{J_2(T_4)}[u(T_4,\cdot)-W_\delta(\cdot-r(T_4))]=2$ by the zero number arguments.
In the former case, virtual spreading happens for $u$ by Lemma
\ref{lem:condition for virtual spreading}, contradicts our assumption. In the latter
case, we have
$$\mathcal{Z}_{[l_2(t),r(t)]}[u(t,\cdot)-W_\delta(\cdot-r(t))]=2\mbox{ for }T_4\leqslant t\ll T_4+1.$$
In a similar way as in the proof of the previous lemma we see that the only possibility is
that $r(t)$ catches up $h(t)$ at $t=T_3$, and the other intersection point between
$u(T_3,\cdot)$ and $W_\delta(\cdot-r(T_3))$ stays on the left. Hence we have \eqref{h'(T3)<beta-c0+delta} again
at time $t=T_3$.
Using a similar idea as in step 1 of the current proof, we obtain \eqref{h'(t)<beta-c_0+delta}
for all large time $t$.

\medskip

\emph{Step 4}. Combining Step 1 with Step 3 we have
$$\beta-c_0<h'(t)<\beta-c_0+\delta\mbox{ for large }t.$$
Since $\delta>0$ can be arbitrarily small, we proves \eqref{h'(t)=beta-c0}.
\end{proof}

\begin{lem}\label{lem:max at right beta}
Under the assumption of Lemma \ref{lem:h'(t)=beta-c0},
$u(t,\cdot)$ has exactly one local maximum point  for large $t$.
\end{lem}
\begin{proof}
Using zero number argument Lemma \ref{zero-number} to $u_x(t,\cdot)$ we see that $u(t,\cdot)$ has exactly
$N$ local maximum points for large $t$, where $N$ is a positive integer. If $N\geqslant2$, then by
Lemma \ref{lem:max at right} the leftmost maximum point $\xi_1(t)$ moves right at a speed not less
than $\beta$. On the other hand, \eqref{h'(t)=beta-c0} indicates $h(t)$ moves right at a speed $\beta-c_0$.
Therefore, after some time, $\xi_1(t)$ reaches $h(t)$, this is a contradiction.
\end{proof}

\begin{thm}\label{thm:V-convergence}
Assume that vanishing and virtual spreading do not happen for the solution $u$ of \eqref{p}.
\begin{itemize}
  \item[\rm (i)] If $c_0<\beta<\beta^*$, then
  \begin{equation}\label{u converges to V}
    \lim\limits_{t\to\infty}\left\|u(t,\cdot)-V^*(\cdot-h(t))\right\|_{L^\infty(I(t))}=0;
  \end{equation}
  \item[\rm (ii)] If $\beta=c_0$, then
  \begin{equation}\label{u converges to 0}
    \lim\limits_{t\to\infty}\left\|u(t,\cdot)\right\|_{L^\infty(I(t))}=0;
  \end{equation}
\end{itemize}
\end{thm}
\begin{proof}
1. We first prove the locally uniform convergence near $h(t)$. Set
$w(t,x):=u(t,x+h(t))$ and $G(t):=g(t)-h(t)$ for $t\geqslant 0$. Then
\begin{equation}\label{w-G-H-problem}
\left\{
\begin{array}{ll}
 w_t = w_{xx}-(\beta-h'(t))w_{x} +f(w), &  t>0,\ G(t)< x<0,\\
 w(t,G(t))=0,\ G'(t)=-\mu w_x(t,G(t))+\mu w_x(t,0) , & t>0,\\
 w(t,0)=0,\ h'(t)=-\mu w_x(t,0), & t>0,\\
G(0)=-2h_0,\ \ w(0,x) =u_0 (x+h_0),& -2h_0\leqslant  x \leqslant 0.
\end{array}
\right.
\end{equation}
It is easy to know that $G_{\infty}:=\lim_{t\to\infty}G(t)=-\infty$. Since
$w\in C^{1+\nu/2,2+\nu}([1,\infty)\times[G(t),0])$, $h\in C^{1+\nu/2}([1,\infty))$
for any $\nu\in(0,1)$ and $h'(t)\to \beta -c_0$ by Lemma \ref{lem:h'(t)=beta-c0}, there exists a sequence $\{t_n\}_{n=1}^\infty$ satisfying
$t_n\to\infty$ as $n\to\infty$ such that
$$w(t+t_n,x)\to v(t,x)\mbox{ as }n\to\infty\mbox{ locally uniformly in }(t,x)\in\R\times(-\infty,0],$$
and $v$ is a solution of
\begin{equation*}
\left\{
\begin{array}{ll}
 v_t = v_{xx}- c_0v_{x} +f(v), &  t\in\R,\ x<0,\\
 v(t,0)=0,\ v_x(t,0)=-\frac{\beta-c_0}{\mu}, & t\in\R.
\end{array}
\right.
\end{equation*}

In case $\beta\in(c_0,\beta^*)$, we show that $v(t,x)\equiv V^*(x)$ for all $t\in\R$. If this is
not true, then there exists $(t_0,x_0)\in\R\times(-\infty,0)$ such that $v(t_0,x_0)\neq V^*(x_0)$.
Then for sufficiently small $\epsilon>0$, when $t\in(0,\epsilon)$ we have $v(t_0+t,x_0)\neq V^*(x_0)$.
Using zero number result Lemma \ref{angenent} for $\eta(t,x):=v(t_0+t,x)-V^*(x)$ in
$(t,x)\in[0,\epsilon]\times[x_0,0]$, we see that $\mathcal{Z}_{[x_0,0]}[\eta(t,\cdot)]<\infty$
for $t\in(0,\epsilon)$, and it decreases strictly once it has a degenerate point in $[x_0,0]$.
This contradicts the fact that $x=0$ is a degenerate zero of $\eta(t,\cdot)$ for all $t\in (0,\epsilon)$.
Therefore, $v(t,x)\equiv V^*(x)$, and so $w(t+t_n,x)\to V^*(x)$ as $n\to \infty$ locally uniformly
in $(t,x)\in \R\times (-\infty, 0]$. By the uniqueness of $V^*(x)$ we actually proves
$u(t,\cdot+h(t))=w(t,\cdot) \to V^*(\cdot)$ as $t\to \infty$ uniformly in $[-M, 0]$ for any $M>0$.

In case $\beta =c_0$, a similar discussion as above shows that $v(t,x)\equiv 0$ and so
$u(t,\cdot+h(t)) \to 0$ as $t\to \infty$ uniformly in $[-M, 0]$ for any $M>0$.

\medskip

2. We prove the uniform convergence in $I(t)$ in case $c_0<\beta<\beta^*$. For any small
$\epsilon>0$, there exists a large $M>0$ such that
$$
V^*(x)\leqslant  V^*(-M)\leqslant \frac{\epsilon}{3}\mbox{ for }x\leqslant -M.
$$
Taking $T>0$ sufficiently large, by Step 1 we have
\begin{equation}\label{u(+h(t))-V*<epsilon/3}
  G(t)<-M,\quad \|u(t,\cdot+h(t))-V^*(\cdot)\|_{L^\infty([-M,0])}<\frac{\epsilon}{3}\mbox{ for }t\geqslant  T.
\end{equation}
Hence, the function $u(t,\cdot+h(t))$ has a maximum point in $[-M,0]$. It is the unique maximum
point by Lemma \ref{lem:max at right beta}. Hence $u(t,\cdot+h(t))$ is increasing
in $ [G(t),-M]$, and so
$$0\leqslant  u(t,x+h(t))\leqslant  u(t,h(t)-M)\leqslant  V^*(-M)+\frac{\epsilon}{3}\leqslant \frac{2\epsilon}{3}
\quad \mbox{for } x\in [G(t),-M],\ t\geqslant T.
$$
This implies that
$$\|u(t,\cdot+h(t))-V^*(\cdot)\|_{L^\infty([G(t),-M])}\leqslant \epsilon\mbox{ for }t\geqslant  T.$$
Combining with \eqref{u(+h(t))-V*<epsilon/3} we proves \eqref{u converges to V}.

\medskip


3. We now prove \eqref{u converges to 0} in case $\beta=c_0$.
By Lemma \ref{lem:max at right beta}, $u(t,\cdot)$ has exactly one maximum point $\xi(t)$
when $t$ is large, say, when $t\geqslant T$ for some $T>0$. There are three cases:

{\it Case 1}. $u(t,\xi(t))\to 0$ as $t\to \infty$;

{\it Case 2}.  $u(t,\xi(t))\to 1$ as $t\to \infty$;

{\it Case 3}. There exist $d\in (0,1)$ and a sequence $\{t_n\}_{n=1}^\infty \subset [T, \infty)$
with $t_n \to \infty$ such that $u(t_n,\xi(t_n))=d$ for $n=1,2,\cdots$.

The limit in \eqref{u converges to 0} follows from Case 1 immediately. We now derive contradictions for
Case 2 and Case 3.

{\it Case 2}. By Lemma \ref{lem:compact tw}, there exists $\delta_1  \in (0, c^*(\beta))$ such that
the equation in \eqref{p} has a compactly supported traveling wave $W_{\delta_1}  (x-{\delta_1} t)$ with
\begin{equation}\label{com tw}
W_{\delta_1} (0)= W_{\delta_1} (-L_{\delta_1}) =0,\quad D_{\delta_1} := \max\limits_{-L_{\delta_1} \leqslant z\leqslant 0} W_{\delta_1} (z)
=\frac12 \ \ \mbox{  and  }\ \  {\delta_1} =-\mu W'_{\delta_1} (0).
\end{equation}
By Lemmas \ref{lem:beta>c0 u to 0} and \ref{lem:center}, $u(t,\cdot)\to 0$ as $t\to \infty$ uniformly in $[g(t), 2L_{\delta_1}]$,
by the result in step 1 above, $u(t,\cdot)\to 0$ as $t\to \infty$ uniformly in $[h(t)-2L_{\delta_1}, h(t)]$. Hence
we may assume that, for some $T_1 >T$,
$$
2L_{\delta_1} <\xi(t)<h(t)-2L_{\delta_1} \mbox{ and } u(t,\xi(t))>D_{\delta_1} =\frac12 \mbox{ for all } t\geqslant T_1.
$$
Now we consider the traveling wave $w_1(t,x) :=W_{\delta_1} (x-{\delta_1} t +{\delta_1} T_1 -g_\infty)$. Clearly,
when $t=T_1$ it has no contact point with $u(T_1, x)$. Since it moves rightward with speed ${\delta_1} >0$ and
since $h'(t)\to 0$, the right endpoint $r_1(t):= {\delta_1} t -{\delta_1} T_1 +g_\infty$ of $w_1$ reaches $x=h(t)$
after some time. Before that, $r_1(t)$ first meets $g(t)$ at time $T_2 >T_1$, and then its left endpoint
$l_1(t) := r_1 (t)-L_{\delta_1}$ meets $g(t)$
at time $T_3 >T_2$. By the zero number argument, for $t\in [T_2, T_3)$ we have $\mathcal{Z}_{[g(t),r_1(t)]}
[w_1 (t,\cdot) - u(t,\cdot)] =1$, and for $T_3 <t\ll T_3 +1$, either
\begin{equation}\label{w1 < u}
w_1 (t,x) < u(t,x)\quad \mbox{for } x\in [l_1 (t), r_1(t)],
\end{equation}
or, $\mathcal{Z}_{[l_1(t) , r_1(t)]} [w_1 (t,\cdot) - u(t,\cdot)] =2$. In the latter case, the two contact
points between $w_1$ and $u$ can not remain and move across $x=\xi(t)$ where $u(t,\xi(t))>\frac12 \geqslant w_1(t,\xi(t))$. Therefore,
before $w_1(t,x)$ moves into the interval $[h(t)-L_{\delta_1}, h(t)]$, the two contact points disappear at some time
$T_4 >T_3$, and so \eqref{w1 < u} holds for $t=T_4$. Once \eqref{w1 < u} holds at some time, it holds for all larger time since
$w_1 $ is a lower solution of \eqref{p}.  This leads to virtual
spreading for $u$ by Lemma \ref{lem:condition for virtual spreading}, a contradiction.


{\it Case 3}. As above we select a compactly supported traveling wave $W_{{\delta_2}} (x-{\delta_2} t)$
for some ${\delta_2}\in (0,c^*(\beta))$ such that
\begin{equation}\label{com tw2}
W_{{\delta_2}} (0)= W_{{\delta_2}} (-L_{{\delta_2}}) =0,\quad D_{{\delta_2}} :=
\max\limits_{-L_{{\delta_2}} \leqslant z\leqslant 0} W_{{\delta_2}} (z)= W_{{\delta_2}} ( -\tilde{z} )= d
\ \ \mbox{  and  }\ \  {\delta_2} =-\mu W'_{{\delta_2}} (0),
\end{equation}
where $-\tilde{z}\in (-L_{{\delta_2}}, 0)$ is the maximum point of $W_{{\delta_2}}(z)$.
By the locally uniform convergence in the above step 1 and in Lemma \ref{lem:beta>c0 u to 0}, there exists
$n_0 $ such that
\begin{equation}\label{case 3}
2L_{{\delta_2}} <\xi(t_n)< h(t_n)-2L_{{\delta_2}} \mbox{ for all } n\geqslant n_0.
\end{equation}
Since $\xi(t_n) -{\delta_2} t_n < h(t_n) -{\delta_2}t_n \to -\infty $ as $n\to \infty$, there exists
$n_1 >n_0$ such that
$$
C:= \xi(t_{n_1}) -{\delta_2} t_{n_1} + {\delta_2} t_{n_0} +\tilde{z} \leqslant g_\infty.
$$
Now we consider the traveling wave $w_2(t,x):=W_{{\delta_2}} (x-{\delta_2} t +{\delta_2} t_{n_0} -C)$
for $t\geqslant t_{n_0}$. Since $w_2(t_{n_0},x) =W_{{\delta_2}} (x-C)$,
$w_2 (t_{n_0},\cdot)$ has no contact point with $u(t_{n_0},x)$. Since $w_2$ moves rightward with speed ${\delta_2} >0$
and since $h'(t)\to 0$, the right endpoint $r_2 (t):= {\delta_2} t - {\delta_2} t_{n_0} +C$ of $w_2$
reaches $x=h(t)$ after some time. Before that, $r_2 (t)$ first meets $g(t)$ at some time $T_5 > t_{n_0}$, and
then the left endpoint $l_2(t):= r_2 (t)- L_{{\delta_2}}$ of $w_2$ meets $g(t)$ at some time $T_6 >T_5$.
We remark that $T_6 <t_{n_1}$. In fact, by \eqref{case 3} we have
$$
r_2 (t_{n_1}) ={\delta_2} t_{n_1} -{\delta_2} t_{n_0} +C =\xi(t_{n_1}) +\tilde{z} >2L_{{\delta_2}}
>  g(T_6) +L_{{\delta_2}} = r_2 (T_6).
$$
Now, for $t\in [T_5, T_6)$, using the zero number argument we have
$\mathcal{Z}_{[g(t),r_2 (t)]} [w_2 (t,\cdot) - u(t,\cdot)] =1$. For $T_6 <t\ll T_6 +1$, we have either
\begin{equation}\label{w2 < u}
w_2 (t,x) \leqslant u(t,x)\quad \mbox{for } x\in [l_2 (t), r_2 (t)],
\end{equation}
or, $\mathcal{Z}_{[l_2 (t), r_2 (t)]} [w_2 (t,\cdot) - u(t,\cdot)] =2$.
\eqref{w2 < u} can not be true, since it implies virtual spreading for $u$ by Lemma \ref{lem:condition for virtual spreading}.
In case $w_2(t,\cdot)-u(t,\cdot)$ has two zeros for $T_6 <t\ll T_6 +1$, by the zero number argument,
the two zeros unite to be one degenerate zero $\xi(t_{n_1})$ at time $t_{n_1}$
(note that $\xi (t_{n_1})$ is the maximum point of both $w_2(t_{n_1}, \cdot)$ and $u(t_{n_1}, \cdot)$).
So after $t_{n_1}$, $w_2$ and $u$ have no
contact points. This implies that $w_2 (t,x) <u(t,x)$ ($w_2 > u$ is impossible since the support of $u$ is wider
than that of $w_2$). This again leads to virtual spreading for $u$ by Lemma \ref{lem:condition for virtual spreading}, a contradiction.

This proves Theorem \ref{thm:V-convergence}.
\end{proof}

\begin{remark}\label{rem:h to beta-c0}
\rm
By Lemma \ref{lem:h'(t)=beta-c0} we have $h(t)=(\beta-c_0)t+\varrho(t)$ for some $\varrho(t)=o(t)$.
Hence the uniform convergence in \eqref{u converges to V}  can be rewritten as \eqref{u to V*}.
\end{remark}

\subsubsection{Proofs of Theorems \ref{thm:middle beta} and \ref{thm:beta=c0}}
In the last of this subsection we prove Theorem \ref{thm:middle beta} and Theorem \ref{thm:beta=c0}.
Remember we use $(u(t,x;\sigma\phi), g(t;\sigma\phi), h(t;\sigma\phi))$ to denote the solution of \eqref{p}
with initial data $u_0 =\sigma \phi$ for some given $\phi\in \mathscr{X}(h_0)$.
Define $E_0$ and $\sigma_*$ as in \eqref{def E0}, and when $c_0 \leqslant  \beta <\beta^*$, denote
$$
E_1:=\{\sigma> 0 \mid \mbox{virtual spreading happens for } (u, g, h)\},\quad \sigma^*:=\inf E_1.
$$
By the comparison principle we have $[\sigma,\infty)\subset E_1$ if $\sigma\in E_1$. Thus
$(\sigma^*,\infty)\subset E_1$.

\medskip

\noindent
 {\bf Proof of Theorem \ref{thm:middle beta}:}\ If $\sigma_*=\infty$, then there is
nothing left to prove. We assume $\sigma_*\in(0,\infty)$ in the following.

We first prove $\sigma_*=\sigma^*$. Otherwise, $\sigma_*<\sigma^*$, and so there exist
$\sigma_1,\ \sigma_2\in(\sigma_*,\sigma^*)$ with $\sigma_1<\sigma_2$. By the strong
comparison principle we have
$$g(t;\sigma_1 \phi)>g(t;\sigma_2 \phi),\quad h(t;\sigma_1 \phi)<h(t;\sigma_2 \phi)$$
and
$$
u(t,x;\sigma_1\phi)<u(t,x;\sigma_2\phi) \ \mbox{ for } x\in I^{\sigma_1} (t):=[g(t;\sigma_1\phi),h(t;\sigma_1\phi)],\ t>0.
$$
Since these inequalities are strict at $t=1$, there exists $\epsilon>0$ small such that
$$
u(1,x;\sigma_1\phi)<u(1,x-\epsilon;\sigma_2\phi)\mbox{ for }x\in I^{\sigma_1}(1).
$$
By the comparison principle again we have
$$
u(t,x;\sigma_1\phi)<u(t,x-\epsilon;\sigma_2\phi) \mbox{ for }x\in I^{\sigma_1}(t),\ t\geqslant 1.
$$
And so
\begin{equation}\label{h1 < h2}
u(t,x+h(t;\sigma_1\phi);\sigma_1\phi)<u(t,x+h(t;\sigma_1\phi)-\epsilon;\sigma_2\phi)
\mbox{ for }x\in[g(t;\sigma_1\phi)-h(t;\sigma_1\phi),0],\ t\geqslant 1.
\end{equation}
By Theorem \ref{thm:V-convergence} (i), both $u(t,x+h(t;\sigma_1\phi);\sigma_1\phi)$ and
$u(t,x+h(t;\sigma_2\phi);\sigma_2\phi)$ converge to the tadpole-like function $V^*(x)$
uniformly. Taking limits as $t\to \infty$ in \eqref{h1 < h2} we deduce a contradiction by
$h(t;\sigma_1\phi)-\epsilon-h(t;\sigma_2\phi)\leqslant -\epsilon$. This proves $\sigma_*=\sigma^*$.

It is easily shown as in the proof of Theorem \ref{thm:sigma_*} that $E_0\setminus \{0\}$
is open, and $E_1$ is open by Lemma \ref{lem:condition for virtual spreading},
so neither vanishing nor virtual spreading happens for
$(u(t,x;\sigma\phi)$, $g(t;\sigma\phi)$,$h(t;;\sigma\phi))$ with $\sigma=\sigma^*$. Thus
$u(t,x; \sigma^*\phi)$ is a transition solution and it converges to $V^*$ as in
Theorem \ref{thm:V-convergence} and Remark \ref{rem:h to beta-c0}.

Other conclusions in Theorem \ref{thm:middle beta} follow from the previous lemmas and theorems.
{\hfill $\Box$}

\medskip

\noindent
 {\bf Proof of Theorem \ref{thm:beta=c0}:}\ If $\sigma_*=\infty$, then there is nothing
left to prove. If $\sigma_*\in(0,\infty)$ and $\sigma^*=\infty$, then vanishing happens for
$u(t,x;\sigma\phi)$ with $\sigma<\sigma_*$, and virtual vanishing happens for $u(t,x;\sigma\phi)$ with
$\sigma\geqslant \sigma_*$. Finally we consider the case $0<\sigma_*\leqslant \sigma^*<\infty$. We show
that $E_1$ is an open set. Indeed, if $\sigma_1\in E_1$, then for any $\delta\in(0,c^*(\beta))$ there exists $T_1>0$,
$x_1\in\R$ such that
$$
u(T_1,x;\sigma_1\phi)>W_\delta(x-x_1)\mbox{ for }x\in[x_1-L_\delta,x_1],
$$
since $u(T_1,\cdot;\sigma_1\phi)$ depends on $\sigma_1$ continuously,
there exists $\epsilon>0$ such that
$$
u(T_1,x;\sigma\phi)>W_\delta(x-x_1)\mbox{ for }x\in[x_1-L_\delta,x_1],
$$
for any $\sigma\in[\sigma_1-\epsilon,\sigma_1+\epsilon]$. By
Lemma \ref{lem:condition for virtual spreading}, virtual spreading happens for
$(u(t,x;\sigma\phi)$, $g(t;\sigma\phi)$,$h(t;\sigma\phi))$. Hence $E_1$ is an open set, and
so $E_1=(\sigma^*,\infty)$.

This proves the theorem.
{\hfill $\Box$}

\section{Uniform convergence when (virtual) spreading happens}
In the main results Theorems \ref{thm:small beta}, \ref{thm:middle beta} and \ref{thm:beta=c0},
we observe (virtual) spreading phenomena, which is the case where the solution converges to $1$
locally uniformly in a fixed or moving coordinate frame. In this section we consider the asymptotic
profiles for such solutions in the whole domain.

Throughout this section we assume $0<\beta<\beta^*$.

\subsection{Locally uniform convergence of the front}
We first describe the asymptotic profile near the front $x=h(t)$. In a similar way as
\cite{DMZ,KM,LL2} one can show that
\begin{prop}\label{prop:asymptotic profile near x=h(t)}
Assume $0<\beta<\beta^*$. If (virtual) spreading happens for a solution of \eqref{p}, then
there exists $H_\infty\in\R$ such that
\begin{equation}\label{h(t)-c*t converges to H}
  \lim\limits_{t\to\infty}[h(t)-c^*t]=H_\infty,\
  \lim\limits_{t\to\infty}h'(t)=c^*,
\end{equation}
\begin{equation}\label{u(t,x+h(t)) converges to U*(x)}
 \lim\limits_{t\to\infty}u(t,\cdot+h(t))=U^*(\cdot)\mbox{ locally uniformly in }(-\infty,0].
\end{equation}
\end{prop}

For small advection: $0<\beta<c_0$, one can give a uniform convergence for the solution
$(u,g,h)$ of \eqref{p} as in \cite{DMZ,KM,LL2}.

\begin{prop}\label{prop:profile when beta is small}
Assume $0<\beta<c_0$. If spreading happens for a solution $(u,g,h)$ of \eqref{p}, then
there exist $G_\infty,\ H_\infty\in\R$ such that \eqref{h(t)-c*t converges to H} holds
and
$$\lim\limits_{t\to\infty}[g(t)-c_l^*t] = G_\infty ,\
\lim\limits_{t\to\infty}g'(t)=c_l^*,$$
$$\lim\limits_{t\to\infty}\|u(t,\cdot)-U^*(\cdot-c^*t-H_\infty)\cdot U_l^*(\cdot-c_l^*t-G_\infty)\|_{L^{\infty}([g(t),h(t)])}=0,
$$
if we extend $U^*$ and $U^*_l$ to be zero outside their supports.
\end{prop}

\subsection{Locally uniform convergence of the back}
In this subsection we show that, when $c_0 \leqslant \beta<\beta^*$, the back of a virtual spreading
solution $u$ converges to a traveling wave $Q$ locally uniformly. We will use the following definition:

\begin{defn}[\cite{DGM}]\label{def:steeper}
Let $u_1$, $u_2$ be two entire solutions of $u_t=u_{xx}-\beta u_x+f(u)$ satisfying
$u_{1x}(t,x)>0$ and $u_{2x}(t,x)>0$ for all $x\in\R, t\in\R$.
We say that $u_1$ is \textbf{steeper than}
$u_2$ if for any $t_1$, $t_2$ and $x_1$ in $\R$ such that
$u_1(t_1,x_1)=u_2(t_2, x_1)$, we have either
$$
u_1(\cdot+t_1,\cdot)\equiv u_2(\cdot+t_2,\cdot)\mbox{ or }
(u_1)_x(t_1,x_1)>(u_2)_x(t_2,x_1).
$$
\end{defn}
As above, $u_1$ and $u_2$ are called entire solutions since they are defined for all $t\in\R$.
The above property implies that the graph of the solution $u_1$ (at any chosen
time moment $t_1$) and that of the solution $u_2$ (at any chosen time moment $t_2$)
can intersect at most once unless they are identical, and that if they intersect
at a single point, then $u_1 -u_2$ is negative on the left-hand side of the intersection point,
while positive on the right-hand side.

\begin{thm}\label{thm:left limit}
Assume $\beta \in [c_0, \beta^*)$. If virtual spreading happens for a solution $(u,g,h)$ of \eqref{p},
then there exists a continuous function $\theta(t)$ with $\theta(t)=o(t)$ and $\theta(t) \to\infty\ (t\to \infty)$
such that for any $M>0$,
\begin{equation}\label{left limit}
  \lim\limits_{t\to\infty}\|u(t,\cdot)-Q(\cdot-(\beta-c_0)t -\theta(t))\|_{L^\infty([g(t),(\beta-c_0)t +\theta(t)+M])}=0.
\end{equation}
\end{thm}

\begin{proof}
The proof is long and is divided into several steps.
We will use $C$ and $T$ to denote positive constants which may be different case by case.

\medskip

\emph{Step 1}. {\it A rough estimate for the speed of the back}. For any $\delta_1 \in
(0,c^*(\beta)-\beta +c_0)$, we consider the compactly supported traveling wave $W_{\delta_1}(x-c_1 t)$ with
$c_1=\beta-c_0+\delta_1 $, where $W_{\delta_1}(z):=W(z;c_1,-c_0+\delta_1 )$ is the solution of \eqref{eq tw}
and \eqref{compact tw} with $c=c_1$, whose support is $[-L_{\delta_1},0]=[-L(c_1,-c_0+\delta_1),0]$.
As in Lemma \ref{lem:compact tw} we denote $D_{\delta_1}:=\max\limits_{-L_{\delta_1}\leqslant z\leqslant0}
W_{\delta_1}(z)$. Write
\begin{equation}\label{def:m}
m:=  \frac12 \inf\{D_{\delta} \mid 0< \delta <c^*(\beta) -\beta +c_0\}.
\end{equation}
Then $m \in (0,1)$ by the phase plane analysis.

By our assumption, virtual spreading happens: $u(t,\cdot+c t)\to1$ locally uniformly in
$\R$ for some $c>0$. Hence for any given $\delta_1 \in (0,c^*-\beta +c_0)$ there exist a large
$T_0$ and $r\in\R$ such that
$$
u(T_0, x)\geqslant W_{\delta_1}(x-r)\quad \mbox{for }x\in[r-L_{\delta_1},r].
$$
By comparison we have
$$
u(t+T_0, x)> W_{\delta_1}(x-c_1t-r)\quad \mbox{for }x\in[r+c_1t-L_{\delta_1},r+c_1t],\ t>0.
$$
Therefore $\chi(t):=\min \{x\in[g(t),h(t)]\big|u(t,x)=m \}$ satisfies
\begin{equation}\label{chi(t)<=c1t+r+chi0}
  \chi(t+T_0)  < c_1 t +r\quad \mbox{for }t>0.
\end{equation}
Combining with \eqref{chi(t)>=} we have
\begin{equation}\label{<=chi(t)<=}
  (\beta-c_0)t+\frac{3}{c_0}\ln t-C\leqslant\chi(t)\leqslant(\beta-c_0+\delta_1)t+C,\quad t\gg1.
\end{equation}

\medskip

\emph{Step 2}. {\it Truncation of the solution}. Instead of $u$ we will consider its truncation
on $[g(t),\xi(t)]$ for some $\xi(t)\in(g(t),h(t))$.

Let $\epsilon \in (0,\frac12 (1-m))$ be any given small constant.
We define $\xi(t)$ as a position near $h(t)$
where $u$ takes value $1-\epsilon$. More precisely, by the
definition of the rightward traveling semi-wave $U^*(x-c^*(\beta)t)$, there exists $M_1 = M_1(\epsilon)>0$
sufficiently large such that $U^*(-M_1)>1-\frac{\epsilon}{2}$.
By \eqref{u(t,x+h(t)) converges to U*(x)}, $u(t,h(t)-M_1)> 1- \epsilon$ for sufficiently large $t$,
and so there exists
$\xi(t)\in[h(t)-M_1, h(t)]$ such that
\begin{equation}\label{u(t,xi(t))=1-epsilon}
  u(t,\xi(t))=1-\epsilon,\quad \mbox{for  large } t.
\end{equation}
By \eqref{<=chi(t)<=} and \eqref{h(t)-c*t converges to H} we have
\begin{equation}\label{xi-chi to infty}
\xi(t)-\chi(t)\geqslant h(t) -M_1 -\chi(t) \geqslant (c^* -\beta +c_0 -\delta_1) t +O(1)\to\infty ,\quad \mbox{as } t\to\infty.
\end{equation}
Since the convergence in \eqref{u(t,x+h(t)) converges to U*(x)} holds in fact
in $C^2_{\rm loc}((-\infty,0])$ topology by parabolic estimate, it follows from
$U^*_x(x)<0$ that
\begin{equation}\label{ux<0}
u_x(t,x)<0 \mbox{ for } x\in [h(t)-2M_1, h(t)] \mbox{ and } t\gg 1.
\end{equation}
So the leftmost local maximum point $\xi_1(t)$ of $u(t,\cdot)$ satisfies
$\xi_1(t)< h(t) -2M_1 < \xi(t)$ for $t\gg 1$.

We now show that
\begin{equation}\label{u>1-epsilon}
u(t,x)\geqslant 1-\epsilon \mbox{ for } x\in [\xi_1(t),\xi(t)],\ t\gg 1.
\end{equation}
In case $u(t,\cdot)$ has exactly one local maximum point $\xi_1(t)$ for large $t$, \eqref{u>1-epsilon}
holds since $u(t,\cdot)$ is decreasing in $[\xi_1(t),\xi(t)]$ and $u(t,\xi(t))=1-\epsilon$.
We now consider the case that $u(t,\cdot)$ has exactly $N\ (\geqslant2)$ local maximum points
$\{\xi_i(t)\}_{i=1}^N$ with $g(t)<\xi_1(t)<\cdots <\xi_N(t)<h(t)$ for large $t$.
We remark that this case is possible only if $\beta<c^*(\beta)$. In fact, when
$\beta\geqslant c^*(\beta)$, by \eqref{xi1(t)>=eta(t)} and \eqref{h(t)-c*t converges to H} we have
$$
0<h(t)-\xi_1(t)<(c^*(\beta)-\beta)t+C\leqslant C,\quad t\gg1.
$$
This contradicts the locally uniform convergence \eqref{u(t,x+h(t)) converges to U*(x)} and the fact
that $U^*$ is a strictly decreasing function. So, in the following, we assume that
\begin{equation}\label{beta<c*}
\beta<c^*(\beta)\quad \mbox{and}\quad \xi_1(t)\geqslant \beta t -C \mbox{ for some } C>0.
\end{equation}

Choose a small $\delta \in (0, c_0)$ and consider the solution $q(z):= W(z;b,-\delta)$ of \eqref{phase sol} with
$\gamma=-\delta$, where $b\in(0,P(-\delta))$. This solution corresponds to a point $G\in S_1$
as in Figure 2 (a), and its trajectory is a curve like
$\Gamma_2$ in Figure 1 (a). When $b\to P(-\delta)$, the trajectory $\Gamma_2\to\Gamma_1$
in Figure 1 (a). As in subsection 3.2, for the above given $\epsilon>0$ and
$\delta\in(0,c_0)$, there exists $b=b(\epsilon,\delta)\in(0,P(-\delta))$ such that
$\mathscr{W}(z):=W(z;b(\epsilon,\delta),-\delta)$ and $\mathcal{L}=L(b(\epsilon,\delta),-\delta)$ satisfy
$$
\mathscr{W}(0)=\mathscr{W}(-\mathcal{L})=0,\quad \mathscr{W}(z)>0\mbox{ for }z\in(-\mathcal{L},0),
$$
$$
\mathscr{W}(\hat{z})=\max\limits_{-\mathcal{L}\leqslant z\leqslant0}
\mathscr{W}(z)= 1- \epsilon \quad\mbox{for some }\hat{z}\in(-\mathcal{L},0),
$$

We prove \eqref{u>1-epsilon} by contradiction. By \eqref{ux<0} we only need to prove \eqref{u>1-epsilon} for
$x\in [\xi_1(t), h(t)-2M_1]$. Assume that there exist a time sequence $\{t_n\}_{n=1}^\infty$
and a sequence $\{y_n\}_{n=1}^\infty$ with $t_n\to \infty$ and $y_n \in [\xi_1(t_n), h(t_n)-2M_1]$ such that
\begin{equation}\label{un<1-epsilon}
u(t_n, y_n) <1-\epsilon \mbox{ for all } n.
\end{equation}
For each $n$, we define a continuous function of $\tau$ by
$$
\psi_n (\tau) := h(\tau) + (\beta -\delta) (t_n -\tau) - y_n -M_1 -\hat{z}.
$$
It is easily seen that
$$
\psi_n (t_n) \geqslant h(t_n) -(h(t_n)-2M_1) -M_1 - \hat{z}>0.
$$
For $\rho \in (0,1)$, by $y_n \geqslant \xi_1(t_n) \geqslant \beta t_n -C$ we have
\begin{eqnarray*}
\psi_n (\rho t_n) & = & h(\rho t_n) + (\beta -\delta)(1-\rho) t_n - y_n +O(1)\\
\ & \leqslant & [c^* \rho - (\beta-\delta) \rho -\delta] t_n +O(1) \to -\infty\mbox{ as } n\to \infty,
\end{eqnarray*}
provided $\rho >0$ is sufficiently small. Hence for such a $\rho $ and for any large $n$,
there exists $\tau_n \in (\rho t_n, t_n)$ such that $\psi_n (\tau_n ) =0$.

For any large $n$, by \eqref{u(t,x+h(t)) converges to U*(x)} we have
$$
u(\tau_n, x) \geqslant U^* (x- h(\tau_n)) -\frac{\epsilon}{2} \geqslant \mathscr{W}(x-h(\tau_n)+M_1),\quad
x\in [h(\tau_n) -M_1 -\mathcal{L}, h(\tau_n)-M_1].
$$
Set $r_n(t):= (\beta-\delta)t + h(\tau_n)-M_1$. Since $\mathscr{W}(x-r_n(t))$ is a compactly
supported traveling wave of \eqref{p}$_1$, and its right endpoint
$$
r_n(t)  = (\beta -\delta)t + c^* \tau_n +H_\infty -M_1 + o(1)
< h(t+\tau_n) =c^* t +c^* \tau_n + H_\infty +o(1)
$$
by \eqref{h(t)-c*t converges to H} and \eqref{beta<c*}, provided $n$ is sufficiently large. Hence
$\mathscr{W}(x- r_n(t)) $ is a lower solution of \eqref{p} and by the comparison principle we have
$$
u(t+\tau_n, x) \geqslant \mathscr{W} (x-r_n(t) ) \ \mbox{ for } x\in [r_n(t)- \mathcal{L} , r_n(t)],\ t>0.
$$
In particular, at $t= t_n -\tau_n >0$ and $x=y_n$, by $\psi_n (\tau_n)=0$ we have
$$
1-\epsilon > u(t_n,y_n) \geqslant \mathscr{W} (-\hat{z}) = 1-\epsilon,
$$
a contradiction. This proves \eqref{u>1-epsilon}.

In what follows, we write $\hat{u}(t,x):=u(t,x)\big|_{x\in[g(t),\xi(t)]}$ as a truncation of $u$.

\medskip

\emph{Step 3}. {\it Truncation of tadpole-like traveling waves}. For any $b\in(0,P(-c_0))$ recall that
$V(z;b,-c_0)$ is a tadpole-like solution of \eqref{phase sol} (cf. point $H$ in Figure 2 (a)).
We choose $b=b(\epsilon)$ near $P(-c_0)$ such that
$$
\max\limits_{z\leqslant 0}V(z;b(\epsilon),-c_0) =V(\bar{z};b(\epsilon),-c_0)=1-2\epsilon,
$$
for some $\bar{z}<0$. In a similar way as above, we write
$$
\widehat{V}(x-(\beta-c_0)t):=V(x-(\beta-c_0)t;b(\epsilon),-c_0)\big|_{x\in(-\infty,\bar{z}+(\beta-c_0)t]}
$$
as a truncation $V$.

\medskip

\emph{Step 4}. {\it Comparison between $\hat{u}$ and $\widehat{V}$}. To study the asymptotic profile
of $\hat{u}$, we compare $\hat{u}$ with a family of the shifts of $\widehat{V}$. Without loss of
generality, we may assume $u(t,x)$ satisfies all the properties in steps 1-3 from time $t=0$. Since
$\hat{u}_x(0,g(0))>0$, one can choose $X>0$ large such that $\hat{u}(0,x)$ and $\widehat{V}(x-\hat{x})$
(for any $\hat{x}\geqslant  X$) intersect at exactly one point $\hat{y}$, and
$$
\hat{u}(0,x)<\widehat{V}(x-\hat{x})\mbox{ for }x\in[g(0),\hat{y}),\
\hat{u}(0,x)>\widehat{V}(x-\hat{x})\mbox{ for }x\in(\hat{y},\min\{\xi(0),\bar{z}+\hat{x}\}).
$$
Since the back of $u(t,\cdot)$ moves rightward faster than $(\beta-c_0)t+\frac{3}{c_0}\ln t-C$
by \eqref{chi(t)>=}, it will exceed $\widehat{V} (x-(\beta-c_0)t-\hat{x})$ at some time
$\hat{T}>0$, that is, their intersection point $(y(t),\hat{u}(t,y(t)))$ starting from
$(\hat{y}, \hat{u}(0,\hat{y}))$ exists only in time interval $[0,\hat{T}]$.

For each $\hat{x}\geqslant  X$, both $\hat{u}(t,x)$ and $\widehat{V}(x-(\beta-c_0)t-\hat{x})$ are solutions
of \eqref{p}$_1$. We now compare them and show that for $t\in [0,\hat{T})$,
\begin{equation}\label{hat{u} < and > widehat{V}}
\left\{
\begin{array}{l}
\mbox{there exists }y(t)\in(g(t),\xi(t))\cap(-\infty,\eta(t)]\mbox{ such that }\\
  \hat{u}(t,x)<\widehat{V}(x-(\beta-c_0)t-\hat{x})\mbox{ for }x\in[g(t),y(t)), \\
  \hat{u}(t,x)>\widehat{V}(x-(\beta-c_0)t-\hat{x})\mbox{ for }x\in(y(t),\min\{\xi(t),\eta(t)\}\big].
\end{array}\right.
\end{equation}
where $\eta(t):= (\beta -c_0)t +\hat{x} +\bar{z}$.
By the comparison principle, this is true provide we exclude the following two
possibilities:
\medskip

(A) the right endpoint $(\xi(t),1-\epsilon)$ of $\hat{u}(t,\cdot)$ touches $\widehat{V}$ at
some time $t\in (0,\hat{T})$;

\medskip

(B) the right endpoint $(\eta(t), 1- 2\epsilon )$ of $\widehat{V} (x-(\beta-c_0)t-\hat{x})$ touches
$\hat{u}$ at some time $t\in (0,\hat{T})$.

\medskip

(A) of course is impossible because $\hat{u}(t,\cdot)$ takes value $1-\epsilon$ at $x=\xi(t)$,
bigger than $\max \widehat{V}$. (B) is impossible when $\eta(t) \in [\xi_1(t), \xi(t)]$ since
in this case $\hat{u}(t,\eta(t))\geqslant 1-\epsilon >\max \widehat{V}$ by \eqref{u>1-epsilon}. When $\eta(t)<\xi_1(t)$,
$\widehat{V}_x(x-(\beta-c_0)t-\hat{x})\big|_{x=\eta(t)}=\widehat{V}_x(\bar{z})=0$ and $\hat{u}_x(t,\eta(t))>0$.
Hence $(\eta(t),1-2\epsilon )$ can not be a new emerging intersection point between $\hat{u}$ and
$\widehat{V}$. This excludes the possibility of (B).

\medskip

\emph{Step 5}. {\it Slope of the back of $u(t,\cdot)$}.
By \eqref{hat{u} < and > widehat{V}} and by the Hopf lemma, at the unique
intersection point $(y(t),\hat{u}(t,y(t)))$ between $\hat{u}$ and $\widehat{V}$ we have
$$
\hat{u}(t,y(t))=\widehat{V} (y(t)-(\beta-c_0)t-\hat{x})\quad\mbox{and}\quad\hat{u}_x(t,y(t))>\widehat{V} _x(y(t)-(\beta-c_0)t-\hat{x}).
$$

Denote $y^0$ the unique root
of $\widehat{V}(z)=m$ in $(-\infty,\bar{z})$. For any given large $t$, we take
$\hat{x}= \chi(t)-(\beta -c_0)t -y^0$, then the function $\hat{u}(t,x)$
and $\widehat{V}(x-(\beta-c_0)t-\hat{x})$ intersect exactly at $x=\chi(t)$:
$$
\hat{u}(t,\chi (t))=m= \widehat{V}(y^0) = \widehat{V}(\chi (t)-(\beta-c_0)t-\hat{x}).
$$
By the Hopf lemma we have
\begin{equation}\label{u steeper than V}
\hat{u}_x(t, \chi (t))>\widehat{V}_x(y^0).
\end{equation}

\medskip

\emph{Step 6}. {\it Convergence of the back of $u$ and the slope of the limit function}.
For any increasing sequence $\{t_n\}_{n=0}^{\infty}$ with $t_n\to\infty \ (n\to \infty)$, we set
$x_n:= \chi (t_n)$ and define
$$
\hat{u}_n(t,x):=\hat{u}(t+t_n,x+x_n)\mbox{ for }g(t+t_n)-x_n\leqslant  x\leqslant \xi(t+t_n)-x_n,\ t>-t_n.
$$
Clearly, $\hat{u}_n(0,0)=\hat{u}(t_n,x_n)=m$ for $n\in\mathbb{N}$. For any given
$t\in\R$, $g(t+t_n)-x_n\to-\infty$ as $n\to\infty$ and by \eqref{h(t)-c*t converges to H} and \eqref{xi-chi to infty} we have
\begin{eqnarray*}
\xi(t+t_n)-x_n & = &\xi(t+t_n)-\xi(t_n)+ [\xi(t_n)-\chi(t_n)]\\
& \geqslant & h(t+t_n) - M_1 -h(t_n) + [\xi(t_n)-\chi(t_n)]\\
& = & c^* t +O(1) + [\xi(t_n)-\chi(t_n)] \to\infty\mbox{ as }n\to\infty.
\end{eqnarray*}

Since $\hat{u}_n(t,x)$ is bounded in $L^\infty$ norm, by parabolic estimate, it is also
bounded in $C^{1+\nu/2,2+\nu}$ $([-M,M]\times[-M,M])$ norm for any $M>0$ and any $\nu\in(0,1)$.
By Cantor's diagonal argument, there exists a subsequence $\{n_j\}$ of $\{n\}$ such that
$$\lim\limits_{j\to\infty}\hat{u}_{n_j}(t,x)=w(t,x)\mbox{ in }C^{1,2}_{\rm loc}(\R^2)\mbox{ topology},$$
where $w\in C^{1,2}(\R^2)$ is an entire solution of \eqref{p}$_1$ with $w(0,0)=m$.
By \eqref{u steeper than V} we have
\begin{equation}\label{w steeper than V}
w_x(0,0)= \lim\limits_{j\to\infty}(\hat{u}_{n_j})_x (0,0) =
\lim\limits_{j\to\infty} \hat{u}_x (t_{n_j}, x_{n_j})
=\lim\limits_{j\to\infty} \hat{u}_x (t_{n_j}, \chi (t_{n_j})) \geqslant \widehat{V}_x(y^0).
\end{equation}

For the solution $Q(z)$ of \eqref{eq tw}-\eqref{tw R} with $c=\beta-c_0$, there exists a unique $y^*\in\R$ such that
$Q(y^*)= m$. By the phase plane analysis (Lemma \ref{lem:tadpole tw bata<beta*} (i)),
$V(\cdot+y^0 ;b(\epsilon),-c_0)= \widehat{V}(\cdot+y^0) \to Q(\cdot+y^*)$ in $C_{\rm loc}^2(\R)$ topology, as
$\epsilon\to0$, or equivalently, as $b(\epsilon)\to P(-c_0)$. Taking limit as $\epsilon\to0$
in \eqref{w steeper than V} we have
\begin{equation}\label{w steeper than Q}
w_x(0,0)\geqslant Q'(y^*).
\end{equation}

On the other hand, both $u_1(t,x):= Q(x-(\beta-c_0)t+y^*)$ and $u_2(t,x):= w(t,x)$ are entire
solutions of \eqref{p}$_1$. By \cite[Lemma 2.8]{DGM}, $u_1$ is steeper than $u_2$
in the sense of Definition \ref{def:steeper}. In particular, taking $t_1 =t_2 =x_1 =0$ in
Definition \ref{def:steeper} we have $u_1(0,0)=Q(y^*)=m=w(0,0)=u_2(0,0)$. Hence
$$
w(t,x)\equiv Q(x-(\beta-c_0)t+y^*)\quad\mbox{for all }t,\ x\in\R
$$
by Definition \ref{def:steeper} and  the inequality \eqref{w steeper than Q}.

Therefore, $\lim\limits_{j\to\infty}\hat{u}_{n_j}(t,x)= Q(x-(\beta-c_0)t+y^*)$
in $C^{1,2}_{\rm loc}(\R^2)$ topology.
By the uniqueness of the limit function $Q$ we have
$$
\lim\limits_{n\to\infty}u_n(t,x)=\lim\limits_{n\to\infty}u(t+t_n,x+\chi(t_n)) =Q(x-(\beta-c_0)t+y^*)\mbox{ in }C^{1,2}_{\rm loc}(\R^2)\mbox{ topology}.
$$
Since $\{t_n\}$ is an arbitrarily chosen sequence we obtain
$$
  \lim\limits_{\tau\to\infty}u(t+\tau,x+\chi(\tau))=Q(x-(\beta-c_0)t +y^*)\mbox{ in }C^{1,2}_{\rm loc}(\R^2)\mbox{ topology}.
$$
Taking $t=0$ we have
\begin{equation}\label{u to Q}
  \lim\limits_{\tau\to\infty}u(\tau,x+\chi(\tau))=Q(x+y^*)\mbox{ in }C^2_{\rm loc}(\R)\mbox{ topology}.
\end{equation}
Define
$$
\theta(\tau):= \chi(\tau) - (\beta-c_0) \tau -y^*,
$$
which is a continuous function of $\tau$.
Then by \eqref{<=chi(t)<=} we have
$$
\frac{3}{c_0}\ln \tau -C \leqslant \theta (\tau) \leqslant \delta_1 \tau +C,\quad \tau \gg 1.
$$
Since this is true for any small $\delta_1 >0$ (see Step 1) we have $\theta(\tau)= o(\tau)\ (\tau\to \infty)$. Thus
by \eqref{u to Q} we have
\begin{equation}\label{u to Q 2}
 \lim\limits_{\tau\to\infty} [u(\tau,x) - Q(x-(\beta-c_0)\tau-\theta(\tau))]=0
\end{equation}
uniformly in $[(\beta-c_0)\tau +\theta(\tau)-M,  (\beta-c_0)\tau + \theta(\tau)+M]$ for any $M>0$.

\medskip
{\it Step 7. Complement of the proof}.  For any $\varepsilon_0>0$, there exists $M>0$ such that
$Q(z)\leqslant Q(-M)\leqslant \varepsilon_0$ for $z<-M$. For this $M$, we choose $T>0$ large such that
when $t>T$ we have
$$
u(t, (\beta-c_0)t +\theta(t)-M )\leqslant 2 Q(-M ) \leqslant 2\varepsilon_0
$$
by \eqref{u to Q 2}. $u(t,\cdot)$  is increasing in $[g(t), (\beta -c_0)t+\theta(t)-M]$
by Lemma \ref{lem:max at right}, hence, when $t>T$ and $x\in [g(t), (\beta -c_0)t+\theta(t)-M]$
we have
$$
|u(t,x)-Q(x-(\beta-c_0)t-\theta(t))| \leqslant u(t,(\beta-c_0)t +\theta(t)-M  )+ Q(-M) \leqslant 3\varepsilon_0.
$$
Combining with \eqref{u to Q 2} we obtain the conclusion \eqref{left limit}.
\end{proof}

\subsection{Uniform convergence}
In this subsection, we complete the proof of Theorem \ref{thm:profile of spreading sol}.

\medskip


\noindent
 {\bf Proof of Theorem \ref{thm:profile of spreading sol}:}\ \eqref{right spreading speed},
 \eqref{left spreading speed} and \eqref{profile convergence 1} are proved in Propositions \ref{prop:asymptotic profile near x=h(t)}
and \ref{prop:profile when beta is small}. We only need to prove \eqref{profile convergence 2} for $c_0\leqslant \beta<\beta^*$.
For any given small $\varepsilon >0$, we will prove
\begin{equation}\label{convergence 2 0}
|u(t,x)-U^*(x-c^* t -H_\infty) \cdot Q(x-(\beta -c_0)t-\theta(t))| < C\varepsilon \quad \mbox{ for } x\in I(t)
\mbox{ and large } t,
\end{equation}
where $C>0$ is a constant independent of $t$ and $x$.

By Lemma \ref{lem:compact tw}, for any $\delta \in (0, c^*(\beta)-\beta +c_0)$ the problem
\eqref{p} has a compactly supported traveling wave $W_\delta (x-(\beta-c_0+\delta)t)$, where
$W_\delta (z)$ (with support $[-L_\delta, 0]$) is the unique solution of
the problem \eqref{eq tw} and \eqref{compact tw}, whose maximum and maximum point
are denoted by $D_\delta$ and $-z_\delta$, respectively.  Moreover, for the above given
$\varepsilon >0$, Lemma \ref{lem:compact tw} also indicates that,
there exists $\delta_\varepsilon \in (0, c^*(\beta)-\beta +c_0)$ such that
$$
D_\delta = W_\delta (-z_\delta) \in (1-\varepsilon, 1) \mbox{ when } \delta\in
(\delta_\varepsilon, c^*(\beta) -\beta +c_0).
$$
We select $\delta_1, \delta_0, \delta_2 \in (\delta_\varepsilon, c^*(\beta)-\beta +c_0 )$ with
$\delta_1 >\delta_0 >\delta_2$ and fix them. For $i=1$ and $2$, denote
$\kappa_i = (1-D_{\delta_i})/(3\varepsilon)$, then $\kappa_i \in (0, \frac13)$.

By the definitions of $U^*(z)$ and $Q(z)$, there exists $M(\varepsilon) >0$ such that when $M>M (\varepsilon)$,
\begin{equation}\label{U* Q =1}
1-\varepsilon \leqslant U^*(z) \leqslant 1\mbox{ for } z\leqslant -M,\qquad
1-\varepsilon \leqslant Q(z) \leqslant 1 \mbox{ for } z\geqslant M
\end{equation}
and there exists $M(\delta_1, \delta_2)>M(\varepsilon)$ such that when $M>M(\delta_1, \delta_2)$,
\begin{equation}\label{UQ > Ddelta}
Q(z) > D_{\delta_1} +\kappa_1 \varepsilon  \mbox{ for } z\in  [M-L_{\delta_1}, M],\qquad
U^*(z) > D_{\delta_2} +\kappa_2 \varepsilon  \mbox{ for } z\in  [-2M, -2M+L_{\delta_2}].
\end{equation}
In what follows we fix an $M>M(\delta_1, \delta_2)>M(\varepsilon)$.

Since the solution $\eta (t)$ of the problem
$$
\eta_t = f(\eta),\quad \eta(0)=1+\|u_0\|_{L^\infty}
$$
is an upper solution of \eqref{p} and since $\eta(t)\to 1$ as $t\to \infty$, there exists
a time $T_1 =T_1 (\varepsilon) >0$ such that
\begin{equation}\label{u<1+ep}
u(t,x) < 1+\varepsilon \quad \mbox{ for } x\in I(t),\ t>T_1.
\end{equation}

By Theorem \ref{thm:left limit}, there exists $T_2 >T_1$ such that when $t>T_2$ we have
\begin{equation}\label{u-Q small}
|u(t,x)-Q(x-(\beta-c_0)t-\theta(t))| < \kappa_1 \varepsilon \mbox{ for } x\in I_l (t) := [g(t),(\beta-c_0)t+\theta(t)+M],
\end{equation}
where $\theta(t)$ is a continuous positive function with $\theta(t)=o(t)$ and $\theta(t)\to \infty\ (t\to \infty)$.
By Proposition \ref{prop:asymptotic profile near x=h(t)}, there exists $T_3 >T_2$ such that
when $t>T_3$ we have
\begin{equation}\label{u-U small}
|u(t,x)-U^*(x-c^*t-H_\infty)|<\kappa_2 \varepsilon \mbox{ for } x\in I_r (t):=[h(t)-2M,h(t)],
\end{equation}
(we extend $U^* (z)$ to be zero for $z>0$ if necessary). We now prove
\begin{equation}\label{u not small middle}
u(t,x) \geqslant 1- \varepsilon\mbox{ for } x\in I_c (t):= [(\beta-c_0)t + \theta(t)+M, h(t)-2M] \mbox{ and large } t.
\end{equation}
Once this is proved, combining it with the above results we obtain \eqref{convergence 2 0} with $C=3$.

In the following we prove \eqref{u not small middle} by contradiction.
Assume that there exist a time sequence $\{t_n\}_{n=1}^\infty$ with
$t_n \to \infty$  and a sequence $\{y_n\}$ with $y_n\in I_c (t_n)$ for each $n$ such that
\begin{equation}\label{u(t,yn)<1-2ep}
  u(t_n,y_n) <  1-\varepsilon\quad\mbox{for all } n.
\end{equation}
We divide the interval $I_c(t)$ into $I^1_c (t)$ and $I^2_c (t)$, where
$$
I^1_c (t):= [(\beta-c_0)t+\theta(t)+M, (\beta-c_0+\delta_0)t],\quad I^2_c (t):= [(\beta-c_0+\delta_0)t, h(t)-2M].
$$
Our idea to derive contradictions is the following. We put a compactly supported traveling wave
$W_{\delta_1}(x-(\beta-c_0+\delta_1)t+C_1)$ (resp. $W_{\delta_2}(x-(\beta-c_0+\delta_2)t+C_2)$)
under $u(t+\tau, x)$ at time $t=0$ in the interval $I_l(\tau)$ (resp. $I_r(\tau)$), and then
as $t$ increases to $t_n -\tau$, its maximum point exactly reaches $y_n\in I^1_c (t_n)$ (resp.
$y_n\in I^2_c (t_n)$), this leads to a contradiction.

First we consider the case that $\{y_n\}$ has a subsequence (denoted it again by $\{y_n\}$)
such that $y_n\in I^1_c (t_n)$ for each $n$. Define a continuous function of $\tau$:
$$
\psi_n^{(1)}(\tau) := \delta_1 \tau -\theta (\tau) + y_n -(\beta-c_0 +\delta_1)t_n - M +z_{\delta_1} \mbox{ for } \tau\leqslant t_n.
$$
Since $y_n \in I^1_c (t_n)$, it is easily seen that
$$
\psi_n^{(1)}(t_n) = y_n -(\beta -c_0)t_n -\theta (t_n) -M +z_{\delta_1} \geqslant z_{\delta_1} >0,
$$
and for $\rho_1 = \frac12 (1-\frac{\delta_0}{\delta_1})\in (0,1)$ we have
$$
\psi_n^{(1)}(\rho_1 t_n) \leqslant ( \rho_1 \delta_1 +\delta_0 -\delta_1 ) t_n + O(1) \to -\infty\mbox{ as } n\to \infty.
$$
Hence, when $n$ is sufficiently large, there exists $\tau_n \in (\rho_1 t_n, t_n)$ such that $\psi_n^{(1)} (\tau_n ) =0$.
For such a large $n$, by \eqref{u-Q small} and \eqref{UQ > Ddelta} we have
$$
u(\tau_n, x) \geqslant Q(x-(\beta-c_0)\tau_n -\theta(\tau_n)) -\kappa_1 \varepsilon \geqslant D_{\delta_1}
\geqslant W_{\delta_1} (x-X) \quad \mbox{ for } x\in [X-L_{\delta_1}, X],
$$
where $X:= (\beta -c_0)\tau_n +\theta(\tau_n) +M$. Using comparison principle we have
$$
u(t+\tau_n, x) \geqslant W_{\delta_1} (x-(\beta -c_0 +\delta_1)t -X) \mbox{ for } x\in J_1(t),\ t>0,
 $$
where $J_1(t) := [(\beta-c_0+\delta_1)t +X-L_{\delta_1}, (\beta-c_0+\delta_1)t +X]$.
By $\psi_n^{(1)}(\tau_n)=0$ we have
$$
y_n =(\beta -c_0 +\delta_1) (t_n -\tau_n) +X -z_{\delta_1} \in J_1 (t_n -\tau_n).
$$
Hence by taking $t=t_n -\tau_n$ and $x=y_n$ we have
$$
1-\varepsilon > u(t_n,y_n) \geqslant W_{\delta_1} (y_n -(\beta-c_0+\delta_1)(t_n -\tau_n) -X) =W_{\delta_1} (-z_{\delta_1})
= D_{\delta_1} > 1-\varepsilon,
$$
a contradiction.

Next we consider the case that $\{y_n\}$ has a subsequence (denoted it again by $\{y_n\}$)
such that $y_n\in I^2_c (t_n)$ for each $n$. The proof is similar as above. Define a continuous function
$$
\psi_n^{(2)}(\tau) := h(\tau) - (\beta -c_0 +\delta_2) \tau - 2M + L_{\delta_2} -y_n +(\beta -c_0+\delta_2) t_n -z_{\delta_2}
 \mbox{ for } \tau\leqslant t_n.
$$
Since $y_n \in I^2_c (t_n)$, it is easily seen that
$$
\psi_n^{(2)}(t_n) = h(t_n)- 2M +L_{\delta_2} - y_n - z_{\delta_2} \geqslant L_{\delta_2} - z_{\delta_2} >0,
$$
and for $\rho_2 \in (0,1)$ we have
\begin{eqnarray*}
\psi_n^{(2)}(\rho_2 t_n) & \leqslant & h(\rho_2 t_n) -(\beta -c_0 +\delta_2) \rho_2 t_n -y_n +(\beta -c_0+\delta_2) t_n +O(1) \\
& = & c^*(\beta) \rho_2 t_n + (\beta -c_0 +\delta_2) (1-\rho_2) t_n -y_n  +O(1) \\
& \leqslant &  [-\rho_2 (\beta -c_0 +\delta_2) +c^*(\beta) \rho_2 - \delta_0 +\delta_2] t_n + O(1) .
\end{eqnarray*}
Since $\delta_0 > \delta_2$, the coefficient of $t_n$ in the last line is negative when $\rho_2 >0$ is sufficiently small. Hence for such a $\rho_2$,
$\psi_n^{(2)}(\rho_2 t_n)\to -\infty$ as $n\to \infty$. Consequently, for any large $n$, there exists $\tau'_n \in (\rho_2 t_n, t_n)$ such that
$\psi_n^{(2)} (\tau'_n) =0$. By \eqref{u-U small} and \eqref{UQ > Ddelta} we have
$$
u(\tau'_n, x) \geqslant U^* (x-c^* \tau'_n  -H_\infty) -\kappa_2  \varepsilon >  D_{\delta_2}
\geqslant W_{\delta_2} (x-X') \quad \mbox{ for } x\in [X'-L_{\delta_2}, X'],
$$
where $X':= h(\tau'_n) -2M +L_{\delta_2}$. By the comparison principle we have
$$
u(t+\tau'_n, x) \geqslant W_{\delta_2} (x-(\beta -c_0 +\delta_2)t -X') \mbox{ for } x\in J_2(t),\ t>0,
 $$
where $J_2(t) := [(\beta-c_0+\delta_2)t +X'-L_{\delta_2}, (\beta-c_0+\delta_2)t +X']$.
By $\psi_n^{(2)}(\tau'_n)=0$ we have
$$
y_n =(\beta -c_0 +\delta_2) (t_n -\tau'_n) +X' -z_{\delta_2} \in J_2 (t_n -\tau'_n).
$$
Hence at $t=t_n -\tau'_n$ and $x=y_n$ we have
$$
1-\varepsilon > u(t_n,y_n) \geqslant W_{\delta_2} (y_n -(\beta-c_0+\delta_2)(t_n -\tau'_n) -X') =W_{\delta_2} (-z_{\delta_2})
= D_{\delta_2} > 1-\varepsilon,
$$
a contradiction.

This completes the proof of Theorem \ref{thm:profile of spreading sol}.
{\hfill $\Box$}



\begin{thebibliography}{123456}
\bibitem{A}
S.B.~Angenent,
{\em The zero set of a solution of a parabolic equation}, J. Reine Angew. Math., 390 (1988), 79-96.

\bibitem{AW}
D.G.~Aronson and H.F.~Weinberger,
{\em Multidimensional nonlinear diffusions arising in population genetics}, Adv. Math., 30 (1978), 33-76.


\bibitem{Ave}
I.E.~Averill,
{\em The Effect of Intermediate Advection on Two Competing Species},
Doctor of Philosophy, Ohio State University, Mathematics, 2012.


\bibitem{BS}
M.~Ballyk and H.~Smith,
{\em A model of microbial growth in a plug flow reactor with wall attachment},
Math. Biosci., 158 (1999), 95-126.

\bibitem{BC}
F.~Belgacem and C.~Cosner,
{\em The effect of dispersal along environmental gradients on the dynamics of populations in heterogeneous environment},
Can. Appl. Math. Quart., 3 (1995), 379-397.

\bibitem{BH}
H.~Berestycki and F.~Hamel,
{\em Front propagation in periodic excitable media},
Comm. Pure Appl. Math., 55 (2002), no.8, 949-1032.

\bibitem{BDK}
G.~Bunting, Y.~Du, and K.~Krakowski,
{\em Spreading speed revisited: analysis of a free boundary model},
Netw. Heterog. Media 7(4) (2012), 583-603.

\bibitem{BP}
J.~Byers and J.~Pringle,
{\em Going against the flow: retention, range limits and invasions in advective environments},
Mar. Ecol. Prog. Ser., 313 (2006), 27-41.

\bibitem{CF}
X.~Chen and A.~Friedman,
{\em A free boundary problem arising in a model of wound healing},
SIAM J. Math. Anal., 32 (2001), 778-800.

\bibitem{DuGuo}
Y.~Du and Z.M.~Guo,
{\em The Stefan problem for the Fisher-KPP equation}, J. Diff. Eqns., 253 (2012), 996-1035.


\bibitem{DuGuo2}
Y.~Du and Z.M.~Guo,
{\em Spreading-vanishing dichotomy in the diffusive logistic model with a free boundary II},
J. Diff. Eqns., 250 (2011), 4336-4366.

\bibitem{DuLin}
Y.~Du and Z.~Lin,
{\em Spreading-vanishing dichtomy in the diffusive logistic model
with a free boundary}, SIAM J. Math. Anal., 42 (2010), 377--405.

\bibitem{DuLou}
Y. Du and B. Lou,
{\em Spreading and vanishing in nonlinear diffusion problems with free boundaries},
J. Eur. Math. Soc., to appear. (arXiv:1301.5373)

\bibitem{DLZ}
Y. Du, B. Lou and M. Zhou,
{\em Nonlinear diffusion problems with free boundaries: convergence and transition speed},
preprint.

\bibitem{DM}
Y. Du and H. Matano,
{\em Convergence and sharp thresholds for propagation in nonlinear diffusion problems}, J. Eur. Math. Soc., 12 (2010), 279-312.

\bibitem{DMW}
Y. Du, H. Matano and K. Wang,
{\em Regularity and asymptotic behavior of nonlinear Stefan problems},
Arch. Rational Mech. Anal., 212 (2014), 957-1010.


\bibitem{DMZ}
Y. Du, H. Matsuzawa and M. Zhou,
{\em Sharp estimate of the spreading speed determined by nonlinear free boundary problems},
SIAM J. Math. Anal., 46 (2014), 375-396.

\bibitem{DGM}
A. Ducrot, T. Giletti and H. Matano,
{\em Existence and convergence to a propagating terrace in one-dimensional reaction-diffusion equations},
Trans. Amer. Math. Soc., 366 (2014), 5541-5566.

\bibitem{F}
F.J. Fernandez, Unique continuation for parabolic operators. II, {\it Comm. Part. Diff. Eqns.}
28 (2003), 1597-1604.

\bibitem{G}
H.~Gu,
{\em On a non-monotone traveling semi-wave of the Fisher-KPP equation with advection}, preprint.

\bibitem{GLiu}
H.~Gu and X.~Liu,
{\em Long time behavior of solutions of a reaction-advection-diffusion equation with 
free and Robin boundary conditions}, in preparation.

\bibitem{GL}
H.~Gu and B.~Lou, 
{\em On the Allen-Cahn equation with advection and free boundaries}, in preparation. 


\bibitem{GLL1}
H.~Gu, Z.~Lin and B.~Lou,
{\em Long time behavior of solutions of a diffusion-advection logistic model with free boundaries},
Appl. Math. Letters, 37 (2014), 49-53.


\bibitem{GLL2}
H.~Gu, Z.~Lin and B.~Lou,
{\em Different asymptotic spreading speeds induced by advection in a diffusion
problem with free boundaries}, Proc. Amer. Math. Soc., to appear. (arXiv:1302.6345)



\bibitem{HNRR}
F.~Hamel, J.~Nolen, J.~Roquejoffre, and L.~Ryzhik,
{\em A short proof of the logarithmic Bramson correction in Fisher-KPP equations},
Netw. Heterog. Media, 8 (2013), 275-289.

\bibitem{HIMN}
D. Hilhorst, M. Iida, M. Mimura and H. Ninomiya,
{\em A competition-diffusion system approximation to the classical two-phase stefan problem},
Japan J. Indust. Appl. Math., 18 (2001), 161-180.

\bibitem{HMS}
D. Hilhorst, M. Mimura and R. Sch\"{a}tzle,
{\em Vanishing latent heat limit in a Stefan-like problem arising in biology},
Nonlinear Anal. RWA, 4 (2003), 261-285

\bibitem{HL}
S.B.~Hsu and Y.~Lou,
{\em Single phytoplankton species growth with light and advection in a water column},
SIMA J. Appl. Math. 70 (2010), 2942-2974.

\bibitem{HAES}
J.~Huisman, M.~Arrayas, U.~Ebert, and B.~Sommeijer,
{\em How do sinking phytoplankton species manage to persist?},
Amer. Naturalist, 159 (2002), 245-254.

\bibitem{KM}
Y.~Kanako and H.~Matsuzawa,
{\em Spreading speed and sharp asymptotic profiles of solutions in free boundary problems for reaction-advection-diffusion equations
}, preprint.

\bibitem{KY}
Y. Kaneko and Y. Yamada,
{\em A free boundary problem for a reaction-diffusion equation appearing in ecology},
Adv. Math. Sci. Appl. 21 (2011) 467-492.

\bibitem{LL1}
X. Liu and B. Lou,
{\em Asymptotic behavior of solutions to diffusion problems with Robin and free boundary conditions},
Math. Model. Nat. Phenom., 8 (2013), 18-32.

\bibitem{LL2}
X. Liu and B. Lou,
{\em Long time behavior of solutions of a reaction-diffusion equation with Robin and free boundary conditions},
preprint.

\bibitem{Maiyang}
N. A. Maidana, H. Yang,
{\em Spatial spreading of West Nile Virus described by traveling waves},
J. Theoretical Bio., 258 (2009), 403-417.

\bibitem{Pet}
I.G.~Petrovski, {\em Ordinary Differential Equations},
Prentice-Hall, Englewood Cliffs, New Jersey, 1966, Dover, New York,
1973.

\bibitem{PL}
A.~Potapov and M.~Lewis,
{\em Climate and competition: the effect of moving range boundaries on habitat invasibility},
Bull. Math. Biol., 66(5) (2004), 975-1008.

\bibitem{RSB}
G.A.~Riley, H.~Stommel, and D.F.~Bumpus, {\em Quantitative ecology of the plankton of the western North Atlantic},
Bulletin of the Bingham Oceanographic Collection Yale University 12 (1949), 1-169.

\bibitem{R}
L.I.~Rubinstein,
{\em The Stefan Problem},
American Mathematical Society, Providence, RI, 1971.


\bibitem{SO}
N.~Shigesada and A.~Okubo, {\em Analysis of the self-shading effect on algal vertical distribution in natural waters},
J. Math. Biol., 12 (1981), 311-326.


\bibitem{SG}
D.C.~Speirs and W.S.C.~Gurney,
{\em Population persistence in rivers and estuaries},
Ecology 8(5) (2001), 1219-1237.

\bibitem{VL}
O.~Vasilyeva and F.~Lutscher,
{\em Population dynamics in rivers: analysis of steady states},
Can. Appl. Math. Q., 18(4) (2010), 439-469.

\bibitem{Wang}
M.X.~Wang,
{\em The diffusive logistic equation with a free boundary and sign-changing coefficient},
J. Diff. Eqns., (258) 2015, 1252-1266.

\bibitem{ZX}
P.~Zhou and D.M.~Xiao,
{\em The diffusive logistic model with a free boundary in heterogeneous environment},
J. Diff. Eqns., 256 (2014), 1927-1954.

\end{thebibliography}
\end{document}